\documentclass[a4paper, 12pt]{amsart}
\usepackage{amsmath, amssymb, amscd, amsbsy, amsthm,braket,bm}

\usepackage[normalem]{ulem}
\usepackage{xcolor}
\usepackage{mathdots}
\usepackage[all]{xy}
\usepackage{tikz}
\usetikzlibrary{calc}
\usetikzlibrary{intersections,backgrounds}
\usepackage{url}

\newcommand{\Q}{\mathbb{Q}}

\newcommand{\T}{\mathbb{T}}
\newcommand{\Y}{\mathbb{Y}}

\newcommand{\Tr}{\mathrm{Tr}}

\newcommand{\sign}{\mathrm{sign}}

\newcommand{\GL}{\mathrm{GL}}
\newcommand{\cok}{\mathfrak{cok}}
\newcommand{\Ker}{\operatorname{Ker}}
\renewcommand{\Im}{\operatorname{Im}}
\newcommand{\ES}{{\sf ES}}

\renewcommand{\H}{\mathcal{H}}
\newcommand{\Sp}{\mathrm{Sp}}

\newcommand{\h}{\mathfrak{h}}
\newcommand{\ord}{\textrm{ord}}
\newcommand{\br}{\textrm{br}}
\newcommand{\m}{\mathfrak{m}}
\newtheorem{theorem}{Theorem}[section]
\newtheorem{lemma}[theorem]{Lemma}
\newtheorem{proposition}[theorem]{Proposition}

\theoremstyle{definition}
\newtheorem{definition}[theorem]{Definition}

\theoremstyle{remark}
\newtheorem{remark}[theorem]{Remark}
\numberwithin{equation}{section}

\newcommand{\thetagraph}[5]{%
\begin{tikzpicture}[scale=0.5, baseline={(0,0.3)},show background rectangle,
background rectangle/.style={fill=none},
        inner frame sep=1mm]
\draw (2.35,0) to[bend right=30] (2.5,0.15)--(2.5,1.35) to[bend right=30] (2.35,1.5) -- (0.5,1.5);
\draw (-2.35,0) to[bend left=30] (-2.5,0.15) -- (-2.5,1.35) to[bend left=30] (-2.35,1.5)-- (-0.5,1.5);
\draw (-2.35,0)--(2.35,0);
\draw (1.5,0) -- (1.5,0.85) to[bend right=20] (1.35,1) -- (0.5,1);
\draw (-1.5,0)-- (-1.5,0.85) to[bend left=20] (-1.35,1) --(-0.5,1);
\draw [red, very thick, dotted,   ->-={.5}{red}] (0.5,1.5)--(-0.5,1.5);
\draw [red, very thick, dotted,  ->-={.5}{red}] (0.5,1)--(-0.5,1);
\ifnum#1>0
\foreach \x in {1,...,#1}
\draw (-2+#1*0.075-\x*0.15+0.075,0) -- (-2+#1*0.075-\x*0.15+0.075,-0.5);
\fi
\ifnum#2>0
\foreach \x in {1,...,#2}
\draw (-1+#2*0.075-\x*0.15+0.075,1) -- (-1+#2*0.075-\x*0.15+0.075,1-0.5);
\fi
\ifnum#3>0
\foreach \x in {1,...,#3}
\draw (0-#3*0.075+\x*0.15-0.075,0) -- (0-#3*0.075+\x*0.15-0.075,-0.5);
\fi
\ifnum#4>0
\foreach \x in {1,...,#4}
\draw (1-#4*0.075+\x*0.15-0.075,1) -- (1-#4*0.075+\x*0.15-0.075,1-0.5);
\fi
\ifnum#5>0
\foreach \x in {1,...,#5}
\draw (2-#5*0.075+\x*0.15-0.075,0) -- (2-#5*0.075+\x*0.15-0.075,-0.5);
\fi
\end{tikzpicture}
}

\newcommand{\kushii}[1]
{\hspace*{-0.9cm}
\begin{tikzpicture}[domain=-3:4,scale=0.45, baseline={(0,0.4)}]
\draw (-3,0)--(4,0);
\draw (-2,0)--(-2,1) coordinate (P2);
\draw (-1,0)--(-1,1) coordinate(P3);
\draw (0,0)--(0,1) coordinate(P4);
\draw (1,0)--(1,1) coordinate (P5);
\draw (2,0)--(2,1)coordinate(P6);
\draw (3,0)--(3,1)coordinate(P7);
\ifnum#1=1
\draw[red, very thick, dotted,   ->-={.5}{red}] (-3,0) to [out=200,in=-20] (4,0);
\else
\draw[blue, very thick, dashed, ->-={.5}{blue}] (-3,0) to [out=200,in=-20] (4,0);
\fi
\ifnum#1=2
\draw[red, very thick, dotted, -<-={.5}{red}] (P2) to [out=80,in=100] (P7);
\else
\draw[blue, very thick, dashed, -<-={.5}{blue}] (P2) to [out=80,in=100] (P7);
\fi
\ifnum#1=3
\draw[red, very thick, dotted, -<-={.5}{red}] (P3) to [out=80,in=100] (P6);
\else
\draw[blue, very thick, dashed, -<-={.5}{blue}] (P3) to [out=80,in=100] (P6);
\fi
\ifnum#1=4
\draw[red, very thick, dotted, -<-={.5}{red}] (P4) to [out=80,in=100] (P5);
\else
\draw[blue, very thick, dashed, -<-={.5}{blue}] (P4) to [out=80,in=100] (P5);
\fi
\end{tikzpicture}
\hspace*{-0.9cm}}

\newcommand{\kushiitwo}[2]
{\hspace*{-0.9cm}
\begin{tikzpicture}[domain=-3:4,scale=0.45, baseline={(0,0.4)}]
\draw (-3,0)--(4,0);
\draw (-2,0)--(-2,1) coordinate (P2);
\draw (-1,0)--(-1,1) coordinate(P3);
\draw (0,0)--(0,1) coordinate(P4);
\draw (1,0)--(1,1) coordinate (P5);
\draw (2,0)--(2,1)coordinate(P6);
\draw (3,0)--(3,1)coordinate(P7);
\ifnum#1=1
\draw[red, very thick, dotted, ->-={.5}{red}] (-3,0) to [out=200,in=-20] (4,0);
\else 
\draw[blue, very thick, dashed, ->-={.5}{blue}] (-3,0) to [out=200,in=-20] (4,0);
\fi
\ifnum#1=2
\draw[red, very thick, dotted, -<-={.5}{red}] (P2) to [out=80,in=100] (P7);
\else \ifnum#2=2
\draw[red, very thick, dotted, -<-={.5}{red}] (P2) to [out=80,in=100] (P7);
\else
\draw[blue, very thick, dashed, -<-={.5}{blue}] (P2) to [out=80,in=100] (P7);
\fi\fi
\ifnum#1=3
\draw[red, very thick, dotted, -<-={.5}{red}] (P3) to [out=80,in=100] (P6);
\else \ifnum#2=3
\draw[red, very thick, dotted, -<-={.5}{red}] (P3) to [out=80,in=100] (P6);
\else
\draw[blue, very thick, dashed, -<-={.5}{blue}] (P3) to [out=80,in=100] (P6);
\fi\fi
\ifnum#1=4
\draw[red, very thick, dotted, -<-={.5}{red}] (P4) to [out=80,in=100] (P5);
\else \ifnum#2=4
\draw[red, very thick, dotted, -<-={.5}{red}] (P4) to [out=80,in=100] (P5);
\else
\draw[blue, very thick, dashed, -<-={.5}{blue}] (P4) to [out=80,in=100] (P5);
\fi\fi
\end{tikzpicture}
\hspace*{-0.9cm}}

\newcommand{\kushiii}
{\hspace*{-0.9cm}
\begin{tikzpicture}[domain=-3:4,scale=0.5, baseline={(0,0.2)}]
\draw (-3,0)--(4,0);
\draw (-2,0)-- (-2,1) coordinate (P2);
\draw (-1,0)--(-1,1) coordinate(P3);
\draw (0,0)--(0,1) coordinate(P4);
\draw (1,0)--(1,1) coordinate (P5);
\draw (2,0)--(2,1)coordinate(P6);
\draw (3,0)--(3,1)coordinate(P7);
\draw[blue, very thick, dashed, ->-={.5}{blue}] (-3,0) to [out=200,in=-20] (4,0);
\draw[blue, very thick, dashed, -<-={.15}{blue}] (P2) to [out=40,in=140] (P5);
\draw[blue, very thick, dashed, -<-={.5}{blue}] (P3) to [out=90,in=90] (P6);
\draw[blue, very thick, dashed, -<-={.85}{blue}] (P4) to [out=40,in=140] (P7);
\end{tikzpicture}
\hspace*{-0.9cm}}

\newcommand{\kushiiitwo}[2]
{\hspace*{-0.9cm}
\begin{tikzpicture}[domain=-3:4,scale=0.45, baseline={(0,0.4)}]
\draw (-3,0)--(4,0);
\draw (-2,0)--(-2,1) coordinate (P2);
\draw (-1,0)--(-1,1) coordinate(P3);
\draw (0,0)--(0,1) coordinate(P4);
\draw (1,0)--(1,1) coordinate (P5);
\draw (2,0)--(2,1)coordinate(P6);
\draw (3,0)--(3,1)coordinate(P7);
\ifnum#1=1
\draw[red, very thick, dotted, ->-={.5}{red}] (-3,0) to [out=200,in=-20] (4,0);
\else 
\draw[blue, very thick, dashed, ->-={.5}{blue}] (-3,0) to [out=200,in=-20] (4,0);
\fi
\ifnum#1=2
\draw[red, very thick, dotted, -<-={.15}{red}] (P2) to [out=40,in=140] (P5);
\else \ifnum#2=2
\draw[red, very thick, dotted, -<-={.15}{red}] (P2) to [out=40,in=140] (P5);
\else
\draw[blue, very thick, dashed, -<-={.15}{blue}] (P2) to [out=40,in=140] (P5);
\fi\fi
\ifnum#1=3
\draw[red, very thick, dotted, -<-={.5}{red}] (P3) to [out=90,in=90] (P6);
\else \ifnum#2=3
\draw[red, very thick, dotted, -<-={.5}{red}] (P3) to [out=90,in=90] (P6);
\else
\draw[blue, very thick, dashed, -<-={.5}{blue}] (P3) to [out=90,in=90] (P6);
\fi\fi
\ifnum#1=4
\draw[red, very thick, dotted, -<-={.85}{red}] (P4) to [out=40,in=140] (P7);
\else \ifnum#2=4
\draw[red, very thick, dotted, -<-={.85}{red}] (P4) to [out=40,in=140] (P7);
\else
\draw[blue, very thick, dashed, -<-={.85}{blue}] (P4) to [out=40,in=140] (P7);
\fi\fi
\end{tikzpicture}
\hspace*{-0.9cm}}

\newcommand{\newbigtheta}[1]{
\ifnum#1=1
\draw (0,6) -- (0,2);
\draw (0,-6) -- (0,-2);
\draw[red, very thick, dotted, ->-]  (0,-2) -- (0,2);
\draw[red, very thick, dotted, ->-] (9,-2) -- (9,2); 
\draw (-9,2) -- (-9,-2); 
\draw (9,2) to (9,5) to [bend right=30] (8,6) to (-8,6) to[bend right=30] (-9,5) to (-9,2);
\draw (9,-2) to (9,-5) to [bend left=30] (8,-6) to (-8,-6) to[bend left=30] (-9,-5) to (-9,-2);
\fi 
\ifnum#1=2
\draw (0,6) -- (0,-6);
\draw[red, very thick, dotted, ->-] (9,-2) -- (9,2); 
\draw[red, very thick, dotted, ->-]  (-9,-2) -- (-9,2);
\draw (9,2) to (9,5) to [bend right=30] (8,6) to (-8,6) to[bend right=30] (-9,5) to (-9,2);
\draw (9,-2) to (9,-5) to [bend left=30] (8,-6) to (-8,-6) to[bend left=30] (-9,-5) to (-9,-2);
\fi
\ifnum#1=3
\draw (0,6) -- (0,2);
\draw (0,-6) -- (0,-2);
\draw[red, very thick, dotted, ->-]  (0,-2) -- (0,2);
\draw[red, very thick, dotted, ->-] (-9,-2) -- (-9,2); 
\draw (9,2) -- (9,-2); 
\draw (9,2) to (9,5) to [bend right=30] (8,6) to (-8,6) to[bend right=30] (-9,5) to (-9,2);
\draw (9,-2) to (9,-5) to [bend left=30] (8,-6) to (-8,-6) to[bend left=30] (-9,-5) to (-9,-2);
\fi 
}
\newcommand{\newsmalltheta}[1]{
\ifnum#1=1
\draw (0,4) -- (0,2);
\draw (0,-4) -- (0,-2);
\draw[red, very thick, dotted, ->-]  (0,-2) -- (0,2);
\draw[red, very thick, dotted, ->-] (9,-2) -- (9,2); 
\draw (-9,2) -- (-9,-2); 
\draw (9,2) to (9,3) to [bend right=30] (8,4) to (-8,4) to[bend right=30] (-9,3) to (-9,2);
\draw (9,-2) to (9,-3) to [bend left=30] (8,-4) to (-8,-4) to[bend left=30] (-9,-3) to (-9,-2);
\fi 
\ifnum#1=2
\draw (0,4) -- (0,-4);
\draw[red, very thick, dotted, ->-] (9,-2) -- (9,2); 
\draw[red, very thick, dotted, ->-]  (-9,-2) -- (-9,2);
\draw (9,2) to (9,3) to [bend right=30] (8,4) to (-8,4) to[bend right=30] (-9,3) to (-9,2);
\draw (9,-2) to (9,-3) to [bend left=30] (8,-4) to (-8,-4) to[bend left=30] (-9,-3) to (-9,-2);
\fi
\ifnum#1=3
\draw (0,4) -- (0,2);
\draw (0,-4) -- (0,-2);
\draw[red, very thick, dotted, ->-]  (0,-2) -- (0,2);
\draw[red, very thick, dotted, ->-] (-9,-2) -- (-9,2); 
\draw (9,2) -- (9,-2); 
\draw (9,2) to (9,3) to [bend right=30] (8,4) to (-8,4) to[bend right=30] (-9,3) to (-9,2);
\draw (9,-2) to (9,-3) to [bend left=30] (8,-4) to (-8,-4) to[bend left=30] (-9,-3) to (-9,-2);
\fi 
}
\newcommand{\newdumbbell}{
\draw (-3,0) -- (3,0);
\draw[red, very thick, dotted, ->-] (10,-2) -- (10,2); 
\draw[red, very thick, dotted, ->-]  (-10,-2) -- (-10,2);
\draw (10,2) to (10,3) to [bend right=30] (9,4) to (4,4) to[bend right=30] (3,3) to (3,-3) to[bend right=30] (4,-4) to (9,-4) to[bend right=30] (10,-3) to (10,-2);
\draw (-10,2) to (-10,3) to [bend left=30] (-9,4) to (-4,4) to[bend left=30] (-3,3) to (-3,-3) to[bend left=30] (-4,-4) to (-9,-4) to[bend left=30] (-10,-3) to (-10,-2);
}

\tikzset{every path/.style={thick}}
\tikzset{ut/.style={line width=2.5pt}}
\tikzset{->-/.style 2 args={
    postaction={decorate},
    decoration={markings, mark=at position #1 with {\arrow[thick, #2]{>}}}
    },
    ->-/.default={0.5}{}
}

\tikzset{-<-/.style 2 args={
    postaction={decorate},
    decoration={markings, mark=at position #1 with {\arrow[thick, #2]{<}}}
    },
    -<-/.default={0.5}{}
}

\AtBeginDocument{\colorlet{normalcolor}{.}}

\usetikzlibrary{decorations.markings}

\begin{document}

\title{On the $2$-loop part of the Johnson cokernel}
\author{Yusuke Kuno}
\address{
Department of Mathematics, 
Tsuda University, 
2-1-1 Tsuda-machi, Kodaira-shi, Tokyo 187-8577, 
Japan
}
\email{kunotti@tsuda.ac.jp}
\date{}

\author{Masatoshi Sato}
\address{
Department of Mathematics and Data Science,
Tokyo Denki University,
5 Senjuasahi-cho, Adachi-ku, Tokyo 120-8551,
Japan}
\email{msato@mail.dendai.ac.jp}

\makeatletter
\@namedef{subjclassname@2020}{\textup{2020} Mathematics Subject Classification}
\makeatother


\begin{abstract}
We study stable $\Sp$-decompositions of the cokernel of the Johnson homomorphism. 
Continuing the work of Conant in 2016, which identified the $1$-loop part of the Johnson cokernel as the Enomoto-Satoh obstruction, we study the $2$-loop part.
Using the corresponding $2$-loop trace map, we capture all the components of the Johnson cokernels in degree 6 that cannot be detected by the Enomoto-Satoh trace. 
\end{abstract}

\maketitle

\section{Introduction}

\subsection{The Johnson cokernel and its loop decomposition}

We will work over the rationals $\Q$.
Let $H$ be a symplectic vector space of dimension $2g$, and let $\mathcal{L} = \bigoplus_{n = 1}^{\infty} \mathcal{L}_n$ be the free Lie algebra generated by $H$.
The Lie algebra $\h = \bigoplus_{n=1}^{\infty} \h(n)$ of (positive) symplectic derivations on $\mathcal{L}$, also known as the Lie algebra of $H$-labeled tree-shaped Jacobi diagrams, appears in low-dimensional topology in various interesting ways.
Our viewpoint in this paper comes from one of its origins in the work of Morita~\cite{Mor93}, namely the natural target space of the Johnson homomorphism of the mapping class group of a surface (of genus $g$ with one boundary component).

Let $\mathcal{M}$ be the mapping class group of a once bordered surface of genus $g$, and let $\pi$ be the fundamental group of the surface, which is a free group of rank $2g$.
Roughly speaking, the Johnson homomorphism is a linearization of the Dehn-Nielsen-Baer isomorphism
\[
%{\rm DN}\colon 
\mathcal{M} \overset{\cong}{\to} {\rm Aut}_{\partial}(\pi),
\]
where the right hand side is the group of automorphisms of $\pi$ that preserve the boundary loop. 
The linearization process first replaces the free group $\pi$ with the free Lie algebra $\mathcal{L}$, and the boundary loop in $\pi$ with the symplectic element $\omega$ in $\mathcal{L}_2$.
Then the group ${\rm Aut}_{\partial}(\pi)$ is replaced with the graded Lie algebra ${\rm Der}_{\omega}^+(\mathcal{L}) = \h$, and the mapping class group $\mathcal{M}$ is replaced with a graded Lie algebra ${\rm Lie}_{\rm gr}(\mathcal{I})$ over $\Q$ produced from a certain filtration on the Torelli subgroup $\mathcal{I}$ of $\mathcal{M}$.
The Johnson homomorphism is an injective graded Lie homomorphism
\[
\tau \colon {\rm Lie}_{\rm gr}(\mathcal{I}) \hookrightarrow \h.
\] 
Let $\m = \bigoplus_{n=1}^{\infty} \m(n)$ be the Lie subalgebra of $\h$ generated by the degree one part $\h(1)$.
By Hain's seminal result~\cite{Hai97}, this coincides with the image of the Johnson homomorphism.
Given this fact, we refer to the quotient $\h/\m$ as the Johnson cokernel.
For the precise definition of the Johnson homomorphism, see \cite{Joh83I,Joh83sur}.
See also \cite{Mor99,Mat13,KK16,Sat16,Ha20}
for developments in the subject. 

There are several reasons for studying the Johnson cokernel. 
\begin{enumerate}
\item
The basic fact is that the Johnson cokernel is nontrivial, as was first observed by Morita~\cite{Mor93}.
This suggests that the linearization process in the construction of the Johnson homomorphism may be too naive, and we might be overlooking some hidden structure on the surface interacting compatibly with the mapping class group. 
Computing the Johnson cokernel is closely related to finding such a structure and would lead to a deeper understanding of self-homeomorphisms of the surface.

\item 
Garoufalidis and Levine \cite{GL05} introduced a 3-dimensional enlargement of the mapping class group, called the homology cobordism group $\mathcal{H}$ of homology cylinders over the surface.
The mapping class group $\mathcal{M}$ is a subgroup of $\mathcal{H}$, and the filtration on the Torelli group $\mathcal{I}$ extends naturally to a filtration on the Torelli part $\mathcal{IH}$ of $\mathcal{H}$. 
This yields the associated Johnson homomorphism $\tau^{\mathcal{H}}\colon {\rm Lie}_{\rm gr}(\mathcal{IH})\to \h$. In contrast to $\tau$, the map $\tau^{\mathcal{H}}$ is an isomorphism. 
Thus, the Johnson cokernel detects the difference between the group arising from 3-dimensional topology and the mapping class group. 

\item It is known that there are components in the Johnson cokernel coming from the absolute Galois group ${\rm Gal}(\overline{\Q}/\Q)$ (see \cite{Mat13} for details).
%Explicit description of these ``Galois obstructions'' is still \eyk
Hence, the Johnson cokernel is a site of fruitful connection between number theory and low-dimensional topology.

\item
The Johnson cokernel is also related to a certain twisted cohomology of the outer automorphism group of a free group.
Conant and Kassabov \cite{CK16} constructed a highly non-trivial homomorphism from the Johnson cokernel $\h/\m$ to the twisted cohomology group, which is related to the graph homology of a certain cyclic operad. 
Conant \cite{Con17} also studied the abelianization $\h/[\h,\h]$ and extracted information about $\h/\m$ using the natural surjective homomorphism $\h/\m\to \h/[\h,\h]$.
\end{enumerate}

It has been a long-standing problem to characterize $\h/\m$.
In particular, for each $n$, the degree $n$ part $\cok(n):= \h(n)/\m(n)$ is naturally a polynomial representation space of the symplectic group $\Sp = \Sp(2g;\Q)$, and people have been working on its $\Sp$-irreducible decomposition.
In low degrees, the decomposition was given in \cite{Joh83I, Mor89, AN95, Hai97, Mor99, MSS15AM}. 
It is known that $\cok(n)$ decomposes as
\begin{equation} \label{eq:cok_decomp}
\cok(n) = \cok_{n, \{ n \} } \oplus \cok_{n, \{ n-2 \} } \oplus \cok_{n, \{ n-4 \} } \oplus \cdots,
\end{equation}
where $\cok_{n, \{ k \} }$ is a direct sum of irreducible $\Sp$-submodules of $\cok(n)$ corresponding to Young diagrams of size $k$.
In a stable range, a more explicit description was given by Conant~\cite{Con16}.
Using the hairy Lie graph complex~\cite{CKV13, CKV15}, he introduced the spaces
\[
\widetilde{\Omega}_{r, \langle n+2-2r \rangle} = \H_{1,n,r} \langle V \rangle/ \beta^{n-1}(\H_{n,n,r}^{\ord} \langle V \rangle). 
%\H_{1,n} \langle V \rangle/ \beta^{n-1}(\H_{n,n}^{\ord} \langle V \rangle)
\]
Here, we follow Conant's notation on the right hand side.
Let $H^{\langle m \rangle}$ be the intersection of the kernels of the pairwise contractions $H^{\otimes m} \to H^{\otimes (m-2)}$ by the symplectic form. 
Briefly, the space $\widetilde{\Omega}_{r, \langle n+2-2r \rangle}$ is generated by tree-shaped Jacobi diagrams of degree $n$ whose $2r$ univalent vertices are paired by $r$ dotted edges, and the remaining $n+2-2r$ univalent vertices are colored with tensors in $H^{\langle n+2-2r \rangle}$.
Below is an example for the case $(r,n) = (2,8)$:  
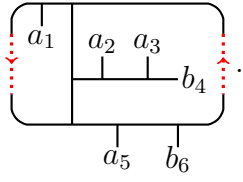
\[
\begin{tikzpicture}[baseline=-3pt, x=2mm, y=2mm]
%%%
%%% A hairy Lie graph %%%
%%%
%%% middle vertical line %%%
\draw (0,4) -- (0,-4);  
%%% right dotted edge %%%
\draw[red, very thick, dotted, -<-] (10,2) -- (10,-2);   
%%% left dotted edge %%%
\draw[red, very thick, dotted, ->-] (-4,2) -- (-4,-2); 
%%% big boundaries %%%
\draw (0,4) to (-3,4) to[bend right=30] (-4,3) to (-4,2);
\draw (0,-4) to (-3,-4) to[bend left=30] (-4,-3) to (-4,-2);
\draw (0,4) to (9,4) to[bend left=30] (10,3) to (10,2);
\draw (0,-4) to (9,-4) to[bend right=30] (10,-3) to (10,-2);
%%% hairs %%%
%% upper left %%
\draw (-2,4) -- (-2,2.5) node[below=-3pt]{$a_1$};
%% middle horizontal %%
\draw (0,-1) -- (7,-1) node[right=-3pt]{$b_4$};
\draw (2,-1) -- (2,0.5) node[above=-3pt]{$a_2$};
\draw (5,-1) -- (5,0.5) node[above=-3pt]{$a_3$};
%% lower right %%
\draw (3,-4) -- (3,-5.5) node[below=-3pt]{$a_5$};
\draw (7,-4) -- (7,-5.5) node[below=-5pt]{$b_6$};
\end{tikzpicture} \ . 
\]
Here, we fix a symplectic basis $\{ a_i, b_i \}_{i=1}^g \subset H$ and the $H$-colorings are chosen from it.
Thus the tensor $a_1 \otimes a_2 \otimes a_3 \otimes b_4 \otimes a_5 \otimes b_6$ made from the $H$-colorings lies in $H^{\langle 6 \rangle}$. 
In \cite[Theorem~2.1]{Con16},
Conant showed that, for $g \gg n$, the trace map of Conant, Kassabov and Vogtmann~\cite{CKV13} induces the $\Sp$-decomposition
\[
\cok(n)\cong \bigoplus_{r=1}^{\lfloor n/2 \rfloor + 1} 
\widetilde{\Omega}_{r, \langle n+2-2r \rangle}
= \widetilde{\Omega}_{1, \langle n \rangle} \oplus
\widetilde{\Omega}_{2, \langle n-2 \rangle} \oplus 
\widetilde{\Omega}_{3, \langle n-4 \rangle} \oplus \cdots
\]
which realizes the decomposition \eqref{eq:cok_decomp}.
By the graphical nature of its description, we call $\widetilde{\Omega}_{r, \langle n+2-2r \rangle}$ the $r$-loop part of the Johnson cokernel in degree $n$.
Moreover, he identified the $1$-loop part $\cok_{n, \{ n \} } = \widetilde{\Omega}_{1, \langle n \rangle}$ as $H^{\langle n \rangle}/D_{2n}$,
the space of top level partitions in the $\Sp$-decomposition of the dihedral coinvariants of $H^{\otimes n}$. 
Note also that in his earlier work \cite{Con15}, Conant showed that the projection onto the $1$-loop part decomposes as $\cok(n) \xrightarrow{\ES} H^{\otimes n}/D_{2n} \to H^{\langle n \rangle}/D_{2n}$, where $\ES$ is the Enomoto-Satoh trace~\cite{ES14} and the second map is the natural projection.
 
The purpose of this paper is to study the 2-loop part of the Johnson cokernel and examine it in detail in degree $6$.
In fact, we consider not only the spaces $\widetilde{\Omega}_{r, \langle n+2-2r \rangle}$ but also bigger spaces $\widetilde{\Omega}_{r, n+2-2r}$, which are defined similarly to $\widetilde{\Omega}_{r, \langle n + 2 - 2r \rangle}$ but without any restriction on $H$-coloring: the univalent vertices can be colored with any tensors in $H^{\otimes(n+2-2r)}$.
Our basic framework is the study of the following map constructed by Conant~\cite{Con16}:
\[
\cok(n)\xrightarrow{\widetilde{\Tr}}\bigoplus_{r=1}^{\lfloor n/2 \rfloor + 1} \widetilde{\Omega}_{r, n + 2 -2r}
\xrightarrow{\pi}
\bigoplus_{r=1}^{\lfloor n/2 \rfloor + 1} \widetilde{\Omega}_{r, \langle n + 2 -2r \rangle}.
\]
Here, $\widetilde{\Tr}$ is induced from the Conant-Kassabov-Vogtmann trace mentioned above.
Its $1$-loop part $\widetilde{\Tr}_1\colon\cok(n)\to \widetilde{\Omega}_{1,n}$ coincides with $\ES$.

\subsection{Main results}
\subsubsection{Description of the $2$-loop space}

Our first result is an explicit presentation of the $2$-loop space $\widetilde{\Omega}_2 = \bigoplus_{m=0}^{\infty} \widetilde{\Omega}_{2,m}$.
Let $\mathcal{T} = \bigoplus_{m=0}^{\infty} H^{\otimes m}$ be the tensor algebra generated by $H$.
We omit the symbol $\otimes$ for the multiplication of $\mathcal{T}$.
Any homogeneous element of degree $m$ can be written as a $\Q$-linear combination of pure tensors, i.e., elements of the form $u_1 u_2 \cdots u_m$, where $u_i \in H$.
The algebra $\mathcal{T}$ is equipped with the structure of a Hopf algebra, whose coproduct and antipode are given by the formula
\begin{align*}
\Delta(u_1\cdots u_m) &= \sum_{I \sqcup J = \{ 1,\ldots, m\}} u_I \otimes u_J, \\
\overline{u_1 \cdots u_m} &= (-1)^m\, u_m \cdots u_1,
\end{align*}
where we write $u_I = u_{i_1} \cdots u_{i_a}$ for $I = \{ i_1, \ldots, i_a\}$ with $i_1 < \cdots < i_a$.
We also use the Sweedler notation for the coproduct: for $u\in \mathcal{T}$, 
\begin{equation} \label{eq:Sweedler}
\Delta(u) = u' \otimes u''.
\end{equation} 
For pure tensors $t, u, v, w \in \mathcal{T}$,
let $\Theta(t,u,v,w)$ be the element in $\widetilde{\Omega}_2$ represented by the following hairy Lie graph:
%%%
\[
\Theta(t,u,v,w) = \, 
\begin{tikzpicture}[baseline=-3pt, x=2mm, y=2mm]
%%%
%%% Figure of Theta(t,u,v,w) %%%
%%%
%%% middle vertical line %%%
\draw (0,4) -- (0,-4);  
%%% right dotted edge %%%
\draw[red, very thick, dotted, -<-] (10,2) -- (10,-2);   
%%% left dotted edge %%%
\draw[red, very thick, dotted, -<-] (-10,2) -- (-10,-2); 
%%% big boundaries %%%
\draw (0,4) to (-9,4) to[bend right=30] (-10,3) to (-10,2);
\draw (0,-4) to (-9,-4) to[bend left=30] (-10,-3) to (-10,-2);
\draw (0,4) to (9,4) to[bend left=30] (10,3) to (10,2);
\draw (0,-4) to (9,-4) to[bend right=30] (10,-3) to (10,-2);
%%% hairs %%%
%% upper left %%
\draw (-8,4) -- (-8,5.5) node[above=-3pt]{$v_k$};
\draw (-2,4) -- (-2,5.5) node[above=-3pt]{$v_1$};
%% upper right %%
\draw (8,4) -- (8,5.5) node[above=-3pt]{$w_1$};
\draw (2,4) -- (2,5.5) node[above=-3pt]{$w_l$};
%% lower left %%
\draw (-8,-4) -- (-8,-5.5) node[below=-3pt]{$t_1$};
\draw (-2,-4) -- (-2,-5.5) node[below=-3pt]{$t_i$};
%% lower right %%
\draw (8,-4) -- (8,-5.5) node[below=-3pt]{$u_j$};
\draw (2,-4) -- (2,-5.5) node[below=-3pt]{$u_1$};
%%% dots between hairs %%%
\draw (-5,5) node{$\cdots$}; 
\draw (5,5) node{$\cdots$}; 
\draw (-5,-5) node{$\cdots$}; 
\draw (5,-5) node{$\cdots$}; 
\end{tikzpicture} \ .
\]
Here, $t = t_1 \cdots t_i$, etc. 
Extending multi-linearly, we define $\Theta(t,u,v,w)$ for all $t,u,v,w \in \mathcal{T}$.

\begin{theorem} \label{thm:omega2_presentation}
The space $\widetilde{\Omega}_2$ is generated by elements $\Theta(t,u,v,w)$ for all $t,u,v,w \in \mathcal{T}$ subject to the following relations:
\begin{enumerate}
\item[$(i)$]
the multi-linearity about four slots of $\Theta$;

\item[$(ii)$] 
for any $t,u,v,w \in \mathcal{T}$, one has
\[
\Theta(t,u,v,w) = \Theta(\overline{v},\overline{w},\overline{t},\overline{u}) = \Theta(\overline{u},\overline{t},\overline{w},\overline{v}).
\]

\item[$(iii)$]
for any $t,u,v,w \in \mathcal{T}$ and $a \in H$,
\[
\Theta(t,u,v,wa) - \Theta(t,u,av,w)
= \Theta(t,au,v,w) - \Theta(ta,u,v,w);
\]

\item[$(iv)$]
for any $t,u,v,w \in \mathcal{T}$, 
\[
\Theta(t,u,v,w) + \Theta(\overline{t'},t''u,\overline{v'},wv'') + \Theta(tu'',\overline{u'},w''v,\overline{w'}) = 0. 
\]
\end{enumerate}
\end{theorem}

\subsubsection{The Johnson cokernel in degree $6$}

Next, we study the degree $6$ part of the Johnson cokernel.
As explained below, this is the smallest degree such that $\widetilde{\Tr}_1 = \ES$ is not injective.
Since the composition
\[
\cok(6) \xrightarrow{\widetilde{\Tr}}
\bigoplus_{r=1}^4 \widetilde{\Omega}_{r,8-2r}
\to 
\bigoplus_{r=1}^4 \widetilde{\Omega}_{r,\langle 8-2r \rangle}
\]
is an isomorphism, $\widetilde{\Tr} = \bigoplus_{r=1}^4 \widetilde{\Tr}_r$ is injective. 
We will show the following: 

\begin{theorem}\label{thm:main2}
When $g$ is sufficiently large, the $\Sp$-homomorphism
\[
\widetilde{\Tr}_1 \oplus \widetilde{\Tr}_2 \colon \cok(6) \to \widetilde{\Omega}_{1,6}\oplus \widetilde{\Omega}_{2,4}
\]
is injective.
\end{theorem}

The $\Sp$-module $\cok(n)$ is described in  Morita, Sakasai and Suzuki \cite[Theorem~1.7]{MSS15AM} for $n\le6$,
where $\m$ coincides with the direct sum of $\Im\tau_g$ and $\mathcal{L}_g$ in Table~1 therein. 
See also Asada~\cite[Lemma~6(i)]{Asa96}.
For the $\Sp$-irreducible decomposition of $\mathcal{L}_g$, see \cite[Theorem~4.3(ii)]{MSS15AM}.
Up to trivial $\Sp$-representations, $\Im\tau_g$ is calculated as follows.
Let $\mathfrak{t}_g$ be the associated graded of the lower central series of the Torelli group for a closed surface, tensored with $\Q$.
Under the assumption that $\mathfrak{t}_g$ is Koszul in a stable range,
Garoufalidis and Getzler \cite[Theorem~1.3]{GaGe17} have computed the stable character of $\mathfrak{t}_g$ as an $\Sp$-representation.
Kupers and Randal-Williams~\cite[Theorem~3.3]{KuRW20} and Felder, Naef and Willwacher~\cite[Corollary~7]{FNW23} showed that $\mathfrak{t}_g$ is Koszul in a
stable range.
In \cite[Theorem~B]{KuRW20}, it was shown that $\mathfrak{t}_g$ is isomorphic to $\Im\tau_g$ modulo trivial $\Sp$-representations.

As in \cite[Section~7.5]{ES14} and \cite{Sak17},
$\widetilde{\Tr}_1 = \ES\colon \cok(n)\to\widetilde{\Omega}_{1,n}$ is injective when $n\le 5$.
On the other hand, $\widetilde{\Tr}_1$ does not capture $\cok(6)$: namely $\m(6) \subsetneq \Ker \widetilde{\Tr}_1 \subset \h(6)$.
The following result is obtained by Morita, Sakasai, and Suzuki. It is partially written in \cite[Theorem~7.5(iii)]{MSS15AM}.
See also \cite{Sak17}.
\begin{theorem}[Morita, Sakasai and Suzuki] \label{thm:cok(6)}
When $g$ is sufficiently large,
\[
(\h(6) \cap \Ker\widetilde{\Tr}_1) / \m(6)\cong [1^4]_\Sp+[1^2]_\Sp+[0]_\Sp
\]
as $\Sp$-modules. 
\end{theorem}

Theorem~\ref{thm:cok(6)} plays an important role in the proof of Theorem~\ref{thm:main2}.
\begin{remark}
The Sp-invariant part $[0]_\Sp$ on the right hand side of Theorem 1.3 essentially comes from 
the absolute Galois group $\operatorname{Gal}(\overline{\Q}/\Q)$.
See \cite{Mat96, Nak96} and a survey article \cite{Mat13} for more details.
\end{remark}

\subsubsection{Relation between two $2$-loop traces}

In \cite{Con15}, Conant introduced the graded spaces $\Omega_r = \bigoplus_{m = 0}^{\infty} \Omega_{r,m}$  by similar ideas used in the definition of $\widetilde{\Omega}_r$. 
For applications of the space $\Omega_r$, see \cite{CK16,Con17}.
Moreover, he showed that the Conant-Kassabov-Vogtmann trace induces maps $\Tr_r^{\rm C} \colon \h \to \Omega_r$ which vanish on $\m$.
In fact, the construction in \cite{Con16} can be seen as a refinement of this.
We will show in Proposition~\ref{prop:refine} that there is a natural map $\widetilde{\Omega}_r \to \Omega_r$ such that the following diagram commutes: 
\[
\xymatrix@R=0.5em@C=4em{
 & \widetilde{\Omega}_r \ar[dd] \\
\h \ar[ur]^{\widetilde{\Tr}_r} \ar[dr]_{\Tr_r^{\rm C}} & \\
 & \Omega_r.
}
\]
The $1$-loop parts of the two constructions are the same:
$\Omega_1 = \widetilde{\Omega}_1$ and 
$\Tr_1^{\rm C} = \widetilde{\Tr}_1 =\ES$.
In the 2-loop parts, there is a difference.
While $\widetilde{\Tr}_2$ does not factor through $\widetilde{\Tr}_1$, we have the following result.

\begin{proposition}[{$=$ Proposition~\ref{prop:factorthrough}}] \label{prop:facthr_short}
There is a map $\Phi\colon \Omega_1 \to \Omega_2$ such that
\[
\Phi \circ \Tr_1^{\rm C} = 3\, \Tr_2^{\rm C}.
\]
\end{proposition}

Thus, any component in the Johnson cokernel captured by $\Tr^{\rm C}_2$ is already captured by the Enomoto-Satoh trace. 

\begin{remark}
The Enomoto-Satoh trace $\ES$ admits several topological interpretations \cite{AKKN_hg, KK15AIF, KK16, MaSa20, NSS23TAMS}.
It would be interesting to give a similar topological meaning for the $2$-loop trace $\widetilde{\Tr}_2$ (and the higher traces $\widetilde{\Tr}_r$). 
\end{remark}

\subsection{Organization} \label{subsec:organization}

In Section~\ref{sec:hlg}, we review the definition of the hairy Lie graph complex and Conant's work~\cite{Con16} on the Johnson cokernel in the stable range.  
In Section~\ref{sec:2-loop}, we investigate the $2$-loop space $\widetilde{\Omega}_2$ and prove Theorem~\ref{thm:omega2_presentation}.
In Section~\ref{sec:deg6}, we focus on the degree $6$ part of the Johnson cokernel and prove Theorem~\ref{thm:main2}.
In Section~\ref{sec:Omega1_Omega_2}, we compare the spaces $\widetilde{\Omega}_r$ and $\Omega_r$ and prove Proposition~\ref{prop:facthr_short}.

\subsection{Notation} \label{subsec:notation}

\begin{itemize}
    \item 
We denote by $(\ \cdot \ )\colon H \times H \to \Q$ the symplectic form on $H$.
It is skew-symmetric and nondegenerate.
A {\em symplectic basis} of $H$ is a subset $\{ a_i, b_i \}_{i=1}^g$ of $H$ such that $(a_i\cdot b_j) = \delta_{ij}$ and $(a_i\cdot a_j) = (b_i\cdot b_j) = 0$.
    \item
We use the standard terminology for Young diagrams \cite{FH91}.
For a Young diagram $\lambda$, we denote by $\lambda_\Sp$ and $\lambda_\GL$ the corresponding $\Sp$-irreducible and $\GL$-irreducible representations, respectively, where $\Sp = \Sp(2g;\Q)$ and $\GL = \GL(2g;\Q)$.
Note that $\lambda_{\Sp}$ appears as the unique $\Sp$-irreducible component of maximal size in the $\Sp$-irreducible decomposition of $\lambda_{\GL}$.

\item 
Let $V$ be a $\GL$-representation with $\GL$-irreducible decomposition $V=\bigoplus_{i=1}^n(\lambda_i)_\GL$. 
By restriction, we can view $V$ as an $\Sp$-representation.
We call the $\Sp$-submodule $\bigoplus_{i=1}^n(\lambda_i)_\Sp\subset V$ the {\em space of top level partitions} of $V$. 

\item
For any $n \ge 0$, there is a natural $\Sp$-module decomposition
\begin{equation} \label{eq:Hn_Sp_decomp}
H^{\otimes n}=
H^{\braket{n}}\oplus H^{\braket{n}}_{\{n-2\}}\oplus\cdots \oplus H^{\braket{n}}_{\{n-2[n/2]\}}
\end{equation}
induced by the contraction maps $(\ \cdot \ )$,
where $H^{\braket{n}}_{\{l\}}$ is a direct sum of irreducible $\Sp$-submodules of $H^{\otimes n}$ corresponding to Young diagrams of size $l$.
See \cite[Lemma~17.15]{FH91}.
In particular, $H^{\langle n \rangle}$ is nothing but the space of top level partitions of $H^{\otimes n}$.
\end{itemize}

\subsection*{Acknowledgments}
The authors are grateful to Nariya Kawazumi for useful discussions in the early stage of this work.
They also thank Shigeyuki Morita, Takuya Sakasai, Masaaki Suzuki and Yuta Nozaki for their helpful suggestions and warm encouragement. 
This study is supported by JSPS KAKENHI Grant Numbers JP23K03121, JP24K00520, and JP22K03298.

\section{Hairy Lie graph complex and the Johnson cokernel} \label{sec:hlg}

\subsection{Hairy Lie graph complex} \label{subsec:hlg}

We recall the definition of the hairy Lie graph complex following \cite{CKV13, CKV15, Con15, Con16}.
\begin{definition}
A {\em hairy Lie graph} with $k$ trees is a union of $k$ unitrivalent trees equipped with the following data.
\begin{itemize}
\item The trees are numbered from $1$ to $k$.
\item Every trivalent vertex in the trees is cyclically oriented, i.e., the three half-edges incident to the vertex are endowed with a cyclic ordering.
\item There are several pairs of univalent vertices, each of which is connected by a directed edge.
We call such a directed edge a {\em dotted edge}.
\item The univalent vertices in the trees not paired by dotted edges are colored with elements of $H$.
We call such a univalent vertex a {\em leaf}.
The edge adjacent to a leaf is called a {\em hair}.
\end{itemize}
\end{definition}
We define the {\em Lie degree} of a hairy Lie graph to be the total number of trivalent vertices of the unitrivalent trees and the {\em homological degree} to be the number of leaves. 
For example, 
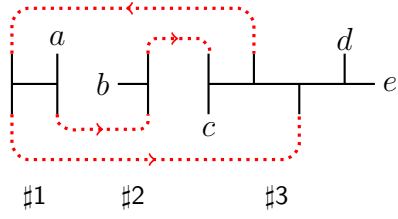
\[
\begin{tikzpicture}[x=2mm, y=2mm]
%%% a hairy Lie graph %%%
%%%
%%% left tree %%%
\draw (0,0) -- (3,0);
\draw (0,-2) -- (0,2);
\draw (3,-2) -- (3,2);
%%% middle tree %%%
\draw (7,0) -- (9,0);
\draw (9,-2) -- (9,2);
%%% right tree %%%
\draw (13,0) -- (24,0);
\draw (13,-2) -- (13,2); 
\draw (16,0) -- (16,2); 
\draw (19,-2) -- (19,0); 
\draw (22,0) -- (22,2);
%%% dotted edges %%% 
\draw[red, very thick, dotted, ->-] (3,-2) to (3,-2.5) to[bend right=30] (4.5,-3) to (8.5,-3) to[bend right=30] (9,-2.5) to (9,-2);
\draw[red, very thick, dotted, ->-] (0,-2) to (0,-4) to[bend right=30] (1,-5) to (18,-5) to[bend right=30] (19,-4) to (19,-2);
\draw[red, very thick, dotted, -<-] (0,2) to (0,4) to[bend left=30] (1,5) to (15,5) to[bend left=30] (16,4) to (16,2);
\draw[red, very thick, dotted, ->-] (9,2) to (9,2.5) to[bend left=30] (9.5,3) to (12.5,3) to (13,2.5) to[bend left=30] (13,2);
%%% H-colorings %%%
\draw (3,3) node{$a$};
\draw (6,0) node{$b$};
\draw (13,-3) node{$c$};
\draw (22,3) node{$d$};
\draw (25,0) node{$e$};
%%% numbering to trees %%%
\draw (1.5,-7.5) node{\footnotesize ${\sf \sharp 1}$};
\draw (8,-7.5) node{\footnotesize ${\sf \sharp 2}$};
\draw (17.5,-7.5) node{\footnotesize ${\sf \sharp 3}$};
\end{tikzpicture}
\]
is a hairy Lie graph with $3$ trees.
The numbering of the trees is indicated by the labels ${\sf \sharp 1}$, etc.
The cyclic ordering at each trivalent vertex is given by the planar structure of the trees in the figure. 
There are $4$ dotted edges.
The Lie degree is $2+1+4 = 7$, and the homological degree is $5$.

\begin{remark} 
In a more general context, the unitrivalent trees in a hairy Lie graph are called internal vertices (see \cite{CKV13}).
In \cite{Con15}, dotted edges are called external edges.
\end{remark}

Let $C_k \H$ be the $\Q$-vector space spanned by hairy Lie graphs with $k$ trees modulo the following relations:
\begin{enumerate}
\item[(i)] the IHX relation within unitrivalent trees 
\[
%% IHX relation
\begin{tikzpicture}[baseline=-0.5em, x=2em, y=2em]
\draw (0,-1/2) -- (0,1/2);
\draw (0,1/2)--({sqrt(3)/2},1);
\draw (0,1/2)--({-sqrt(3)/2},1);
\draw (0,-1/2)--({sqrt(3)/2},-1);
\draw (0,-1/2)--({-sqrt(3)/2},-1);
\end{tikzpicture}
\quad - \quad 
\begin{tikzpicture}[baseline=-0.5em, x=2em, y=2em]
\draw (-1/2,0) -- (1/2,0);
\draw (1/2,0)--({sqrt(3)/2},1);
\draw (-1/2,0)--({-sqrt(3)/2},1);
\draw (1/2,0)--({sqrt(3)/2},-1);
\draw (-1/2,0)--({-sqrt(3)/2},-1);
\end{tikzpicture}
\quad + \quad 
\begin{tikzpicture}[baseline=-0.5em, x=2em, y=2em]
\draw (-1/2,0) -- (1/2,0);
\draw (-1/2,0)--({sqrt(3)/2},1);
\draw (1/2,0)--({-sqrt(3)/2},1);
\draw (1/2,0)--({sqrt(3)/2},-1);
\draw (-1/2,0)--({-sqrt(3)/2},-1);
\end{tikzpicture}
= 0;
\]

\item[(ii)] the AS relation within unitrivalent trees
\[
%% AS relation
\begin{tikzpicture}[baseline=-1em, x=2em, y=2em]
\draw (0,0)--(0,-1);
\draw (0,0) to[out=150, out looseness=2](0.866,1/2);
\draw (0,0) to[out=30, out looseness=2](-0.866,1/2);
\end{tikzpicture}
= -\begin{tikzpicture}[baseline=-1em, x=2em, y=2em]
\draw (0,0)--(0,-1);
\draw (0,0)--({sqrt(3)/2},1/2);
\draw (0,0)--({-sqrt(3)/2},1/2);
\end{tikzpicture};
\]

\item[(iii)] the multi-linearity relation on $H$-colorings;
\item[(iv)] switching direction of a dotted edge gives a minus sign;
\item[(v)] renumbering the trees by a transposition gives a minus sign.  
\end{enumerate}
We also use a variant introduced by Conant~\cite{Con16}: let $C_k \H^{\ord}$ be the $\Q$-vector space spanned by hairy Lie graphs with $k$ trees modulo the relations (i) to (iv) above.
In $C_k \H^{\ord}$, the $k$ trees in a hairy Lie graph have a well-defined linear ordering, explaining the notation. 

The Lie degree and the number of dotted edges are preserved by the relations (i) to (v).
Thus, the space $C_k \H$ admits the direct sum decomposition
\[
C_k \H = \bigoplus_{n \ge 1} C_k \H(n) = 
\bigoplus_{n \ge 1} \bigoplus_{r \ge 0} C_{k,r} \H(n),
\]
where $n$ stands for the Lie degree and $r$ the number of dotted edges.
Similarly, 
\[
C_k \H^{\ord} = \bigoplus_{n \ge 1} C_k \H^{\ord}(n) = 
\bigoplus_{n \ge 1} \bigoplus_{r \ge 0} C_{k,r} \H^{\ord}(n).
\]
The homological degree $m$ of a hairy Lie graph in $C_{k,r} \H(n)$ and $C_{k,r} \H^{\ord}(n)$ is given by the following formula:
\[
n + 2k = m + 2r.
\]

\begin{remark}
With the above notation, $\h(n) = C_{1,0} \H(n)$.
\end{remark}

When $k=1$, we introduce the notation $\T_{r,m} := C_{1,r} \H(m + 2r - 2)$ to emphasize the homological degree $m$.
An element of $\T_{r,m}$ is a tree-shaped Jacobi diagram with $r$ dotted edges and $m$ univalent $H$-colored vertices.
Thus we have
\[
C_1 \H (n) = \bigoplus_{r=0}^{\lfloor n/2 \rfloor + 1} C_{1,r} \H(n)
= \bigoplus_{r=0}^{\lfloor n/2 \rfloor + 1} \T_{r, n + 2 -2r}.
\]
Here, note that a tree-shaped Jacobi diagram of degree $n$ has $n+2$ univalent vertices; if $r$ dotted edges are attached to such a diagram, there remain $n+2-2r$ univalent vertices.
In the same manner, we set $\Y_{r,m}^{\ord}:= C_{(m + 2r)/3, r} \H^{\ord}((m + 2r)/3)$.
An element of $\Y_{r,m}^{\ord}$ is a union of ordered $(m+2r)/3$ tripods with $r$ dotted edges and $m$ univalent $H$-colored vertices. 
We have
\[
C_n \H^{\ord} (n) = \bigoplus_{r=0}^{\lfloor 3n/2 \rfloor} C_{n,r} \H^{\ord}(n)
= \bigoplus_{r=0}^{\lfloor 3n/2 \rfloor} \Y_{r, 3n - 2r}^{\ord}.
\]
Here, note that a union of $n$ tripods has $3n$ univalent vertices; if $r$ dotted edges are attached to such a diagram, there remain $3n - 2r$ univalent vertices.

\subsection{Conant's description of the Johnson cokernel}\label{subsec:tracemap}

We review a graphical description of the Johnson cokernel in the stable range given by Conant~\cite{Con16}, which is based on an important result by Hain~\cite{Hai97} that the image of the Johnson homomorphism is generated by the degree one part $\h(1) = {\rm Im}\,\tau_1^\Q$.

First we recall the trace map introduced by Conant, Kassabov and Vogtmann~\cite{CKV13}.
In fact, we use the degree one part only, which is of the form 
\[
\Tr: \h(n) \to C_1 \H(n). 
\]
In more detail, $\Tr$ decomposes into the components
\[
\Tr = \bigoplus_{r=0}^{\lfloor n/2 \rfloor + 1} \Tr_r: \h(n) \to 
\bigoplus_{r=0}^{\lfloor n/2 \rfloor + 1} C_{1,r} \H(n),
\]
and $\Tr_r: \h(n) \to C_{1,r} \H(n)$ is defined by adding $r$ dotted edges to an $H$-colored tree-shaped Jacobi diagram in all possible unordered ways.
Adding a dotted edge means to pick up a pair $(l_1,l_2)$ of leaves of an $H$-colored tree-shaped Jacobi diagram, add a dotted edge directed from $l_1$ to $l_2$, and multiply by $(x_1 \cdot x_2)$, where $x_i \in H$ is the coloring of $l_i$.
The relation (iv) for $C_1 \H$ implies that the result depends only on the unordered pair $\{ l_1, l_2\}$.
For example, if 
\[
X = 
\begin{tikzpicture}[baseline=-3pt, x=2mm, y=2mm]
%%% X: an H-colored tree-shaped Jacobi tree %%%
%%%
%%% backbone tree %%%
\draw (0,0) -- (24,0);
\draw (0,-2) -- (0,2);
\draw (3,0) -- (3,-2);
\draw (6,0) -- (6,2);
\draw (9,0) -- (9,-2);
\draw (12,0) -- (12,2);
\draw (15,0) -- (15,-2);
\draw (18,0) -- (18,2);
\draw (21,0) -- (21,-2);
%%% H-colorings %%%
\draw (0,3) node{$a_1$};
\draw (0,-3) node{$a_4$};
\draw (3,-3) node{$a_3$};
\draw (6,3) node{$b_5$};
\draw (9,-3) node{$b_2$};
\draw (12,3) node{$b_1$};
\draw (15,-3) node{$a_2$};
\draw (18,3) node{$a_6$};
\draw (21,-3) node{$b_3$};
\draw (25.5,0) node{$a_7$};
\end{tikzpicture},
\] 
where $a_i$ and $b_i$ stand for members of a symplectic basis of $H$, then 
\[
\Tr_2(X) = 
\begin{tikzpicture}[baseline=-2pt, x=1mm, y=1mm]
%%% 1st term in Tr_2(X) %%%
%%%
%%% backbone tree %%%
\draw (0,0) -- (24,0);
\draw (0,-2) -- (0,2);
\draw (3,0) -- (3,-2);
\draw (6,0) -- (6,2);
\draw (9,0) -- (9,-2);
\draw (12,0) -- (12,2);
\draw (15,0) -- (15,-2);
\draw (18,0) -- (18,2);
\draw (21,0) -- (21,-2);
%%% H-colorings %%%
%\draw (0,3) node{$a_1$};
\draw (-0.25,-3.5) node{\small $a_4$};
\draw (3.75,-3.5) node{\small $a_3$};
\draw (6,3.6) node{\small $b_5$};
%\draw (9,-3) node{$b_2$};
%\draw (12,3) node{$b_1$};
%\draw (15,-3) node{$a_2$};
\draw (18,3.5) node{\small $a_6$};
\draw (21,-3.5) node{\small $b_3$};
\draw (26,0) node{\small $a_7$};
%%% dotted edges %%%
\draw[red, very thick, dotted, ->-] (0,2) -- (0,6) to[bend left=30] (1,7) -- (11,7) to[bend left=30] (12,6) -- (12,2);
\draw[red, very thick, dotted, -<-] (9,-2) -- (9,-3) to[bend right=30] (10,-4) -- (14,-4) to[bend right=30] (15,-3) -- (15,-2);
\end{tikzpicture}
+ 
\begin{tikzpicture}[baseline=-2pt, x=1mm, y=1mm]
%%% 2nd term in Tr_2(X) %%%
%%%
%%% backbone tree %%%
\draw (0,0) -- (24,0);
\draw (0,-2) -- (0,2);
\draw (3,0) -- (3,-2);
\draw (6,0) -- (6,2);
\draw (9,0) -- (9,-2);
\draw (12,0) -- (12,2);
\draw (15,0) -- (15,-2);
\draw (18,0) -- (18,2);
\draw (21,0) -- (21,-2);
%%% H-colorings %%%
%\draw (0,3) node{$a_1$};
\draw (0,-3.5) node{\small $a_4$};
%\draw (3,-3) node{$a_3$};
\draw (6,3.6) node{\small $b_5$};
\draw (9,-3.5) node{\small $b_2$};
%\draw (12,3) node{$b_1$};
\draw (15,-3.5) node{\small $a_2$};
\draw (18,3.5) node{\small $a_6$};
%\draw (21,-3) node{$b_3$};
\draw (26,0) node{\small $a_7$};
%%% dotted edges %%%
\draw[red, very thick, dotted, ->-] (0,2) -- (0,6) to[bend left=30] (1,7) -- (11,7) to[bend left=30] (12,6) -- (12,2);
\draw[red, very thick, dotted, ->-] (3,-2) -- (3,-6) to[bend right=30] (4,-7) -- (20,-7) to[bend right=30] (21,-6) -- (21,-2);
\end{tikzpicture}
+
\begin{tikzpicture}[baseline=-2pt, x=1mm, y=1mm]
%%% 3rd term in Tr_2(X) %%%
%%%
%%% backbone tree %%%
\draw (0,0) -- (24,0);
\draw (0,-2) -- (0,2);
\draw (3,0) -- (3,-2);
\draw (6,0) -- (6,2);
\draw (9,0) -- (9,-2);
\draw (12,0) -- (12,2);
\draw (15,0) -- (15,-2);
\draw (18,0) -- (18,2);
\draw (21,0) -- (21,-2);
%%% H-colorings %%%
\draw (0,3.5) node{\small $a_1$};
\draw (0,-3.5) node{\small $a_4$};
%\draw (3,-3) node{$a_3$};
\draw (6,3.6) node{\small $b_5$};
%\draw (9,-3) node{$b_2$};
\draw (12,3.6) node{\small $b_1$};
%\draw (15,-3) node{$a_2$};
\draw (18,3.5) node{\small $a_6$};
%\draw (21,-3) node{$b_3$};
\draw (26,0) node{\small $a_7$};
%%% dotted edges %%%
\draw[red, very thick, dotted, -<-] (9,-2) -- (9,-3) to[bend right=30] (10,-4) -- (14,-4) to[bend right=30] (15,-3) -- (15,-2);
\draw[red, very thick, dotted, ->-] (3,-2) -- (3,-6) to[bend right=30] (4,-7) -- (20,-7) to[bend right=30] (21,-6) -- (21,-2);
\end{tikzpicture}.
\]
Let
\[
\Tr^{\ord}: \h(1)^{\otimes n} \to C_n \H^{\ord}(n)
\]
be the map defined in the same manner: it adds several dotted edges to an ordered $n$-tuple of $H$-colored tripods, which represents an element in $\h(1)^{\otimes n}$, in all possible unordered ways.

Next we recall the map 
\[
\beta: C_k \H^{\ord} \to C_{k-1} \H^{\ord}.
\]
Let $X$ be a hairy Lie graph with $k$ trees, and we arrange that any dotted edge between the first and second trees is directed from the first tree.
Then, $\beta(X)$ is defined by changing a dotted edge joining the first and second trees to a solid edge in all possible ways.
We number the resulting $(k-1)$ trees so that the tree obtained from the first and second trees comes first followed by the remaining $(k-2)$ trees keeping their linear ordering.
For example, if
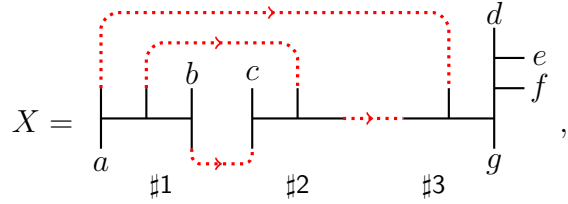
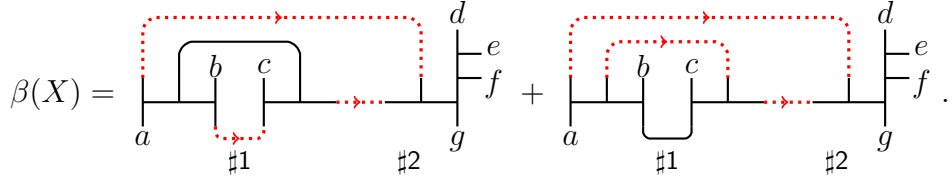
\[
X = 
\begin{tikzpicture}[baseline=-5pt, x=2mm, y=2mm]
%%% X: sample graph for beta %%% 
%%%
%%% 3 backbone trees %%%
\draw (0,0) -- (6,0);
\draw (0,-2) -- (0,2);
\draw (3,0) -- (3,2);
\draw (6,-2) -- (6,2);
\draw (10,0) -- (16,0);
\draw (10,-2) -- (10,2);
\draw (13,0) -- (13,2);
\draw (20,0) -- (26,0);
\draw (26,6) -- (26,-2);
\draw (23,0) -- (23,2);
\draw (26,2) -- (28,2);
\draw (26,4) -- (28,4);
%%% H-colorings %%%
\draw (0,-3) node{$a$};
\draw (6,3) node{$b$};
\draw (10,3) node{$c$};
\draw (26,7) node{$d$};
\draw (29,4) node{$e$};
\draw (29,2) node{$f$};
\draw (26,-3) node{$g$};
%%% dotted edges %%%
\draw[red, very thick, dotted, ->-] (0,2) -- (0,6) to[bend left=30] (1,7) -- (22,7) to[bend left=30] (23,6) -- (23,2);
\draw[red, very thick, dotted, ->-] (3,2) -- (3,4) to[bend left=30] (4,5) -- (12,5) to[bend left=30] (13,4) -- (13,2);
\draw[red, very thick, dotted, ->-] (6,-2) -- (6,-2.5) to[bend right=30] (6.5,-3) -- (9.5,-3) to[bend right=30] (10,-2.5) -- (10,-2);
\draw[red, very thick, dotted, ->-] (16,0) -- (20,0);
%%% numbering to trees %%%
\draw (4,-4.5) node{\footnotesize ${\sf \sharp 1}$};
\draw (13,-4.5) node{\footnotesize ${\sf \sharp 2}$};
\draw (22,-4.5) node{\footnotesize ${\sf \sharp 3}$};
\end{tikzpicture},
\]
then 
\[
\beta(X) = 
\begin{tikzpicture}[baseline=0pt, x=1.6mm, y=1.6mm]
%%% X: sample graph for beta %%% 
%%%
%%% 3 backbone trees %%%
\draw (0,0) -- (6,0);
\draw (0,-2) -- (0,2);
\draw (3,0) -- (3,2);
\draw (6,-2) -- (6,2);
\draw (10,0) -- (16,0);
\draw (10,-2) -- (10,2);
\draw (13,0) -- (13,2);
\draw (20,0) -- (26,0);
\draw (26,6) -- (26,-2);
\draw (23,0) -- (23,2);
\draw (26,2) -- (28,2);
\draw (26,4) -- (28,4);
%%% H-colorings %%%
\draw (0,-3) node{$a$};
\draw (6,3) node{$b$};
\draw (10,3) node{$c$};
\draw (26,7.5) node{$d$};
\draw (29,4.5) node{$e$};
\draw (29,1.5) node{$f$};
\draw (26,-3.25) node{$g$};
%%% solid edge %%%
\draw (3,2) -- (3,4) to[bend left=30] (4,5) -- (12,5) to[bend left=30] (13,4) -- (13,2);
%%% dotted edges %%%
\draw[red, very thick, dotted, ->-] (0,2) -- (0,6) to[bend left=30] (1,7) -- (22,7) to[bend left=30] (23,6) -- (23,2);
\draw[red, very thick, dotted, ->-] (6,-2) -- (6,-2.5) to[bend right=30] (6.5,-3) -- (9.5,-3) to[bend right=30] (10,-2.5) -- (10,-2);
\draw[red, very thick, dotted, ->-] (16,0) -- (20,0);
%%% numbering to trees %%%
\draw (8,-5) node{\footnotesize ${\sf \sharp 1}$};
\draw (22,-5) node{\footnotesize ${\sf \sharp 2}$};
\end{tikzpicture}
+ 
\begin{tikzpicture}[baseline=0pt, x=1.6mm, y=1.6mm]
%%% X: sample graph for beta %%% 
%%%
%%% 3 backbone trees %%%
\draw (0,0) -- (6,0);
\draw (0,-2) -- (0,2);
\draw (3,0) -- (3,2);
\draw (6,-2) -- (6,2);
\draw (10,0) -- (16,0);
\draw (10,-2) -- (10,2);
\draw (13,0) -- (13,2);
\draw (20,0) -- (26,0);
\draw (26,6) -- (26,-2);
\draw (23,0) -- (23,2);
\draw (26,2) -- (28,2);
\draw (26,4) -- (28,4);
%%% H-colorings %%%
\draw (0,-3) node{$a$};
\draw (6,3) node{$b$};
\draw (10,3) node{$c$};
\draw (26,7.5) node{$d$};
\draw (29,4.5) node{$e$};
\draw (29,1.5) node{$f$};
\draw (26,-3.25) node{$g$};
%%% solid edge %%%
\draw (6,-2) -- (6,-2.5) to[bend right=30] (6.5,-3) -- (9.5,-3) to[bend right=30] (10,-2.5) -- (10,-2);
%%% dotted edges %%%
\draw[red, very thick, dotted, ->-] (0,2) -- (0,6) to[bend left=30] (1,7) -- (22,7) to[bend left=30] (23,6) -- (23,2);
\draw[red, very thick, dotted, ->-] (3,2) -- (3,4) to[bend left=30] (4,5) -- (12,5) to[bend left=30] (13,4) -- (13,2);
\draw[red, very thick, dotted, ->-] (16,0) -- (20,0);
%%% numbering to trees %%%
\draw (8,-5) node{\footnotesize ${\sf \sharp 1}$};
\draw (22,-5) node{\footnotesize ${\sf \sharp 2}$};
\end{tikzpicture}.
\]

Successive application of $\beta$ yields the map
\[
\beta^{n-1} : C_n \H^{\ord}(n) \to C_1 \H^{\ord}(n) = C_1 \H(n)
\]
which keeps the homological degree.
Note that $\beta^{n-1}$ maps $C_{n,r} \H^{\ord}(n)$ to $C_{1,r-n+1} \H (n)$, and $\beta^{n-1}$ vanishes if $r <  n-1$.

Now, for each $r$, the space $C_{1,r} \H(n) = \T_{r,n+2-2r}$ is the image of an $\Sp$-homomorphism from a direct sum of finitely many copies of the tensor space $H^{\otimes(n+2-2r)}$.
Thus the decomposition \eqref{eq:Hn_Sp_decomp} yields an $\Sp$-decomposition of $\T_{r, n+2-2r}$. 
Let $\T_{r, \langle n+2-2r \rangle}$ be the space of top level partitions in $\T_{r, n+2-2r}$, namely the sum of irreducible $\Sp$-submodules corresponding to Young diagrams of size $n+2-2r$.
Set
\[
C_1 \H_{\rm top} (n) := \bigoplus_{r=0}^{\lfloor (n/2) \rfloor + 1} \T_{r, \langle n+2-2r \rangle}.
\] 
Assembling the projections $\T_{r, n+2-2r} \to \T_{r, \langle n+2-2r \rangle}$,
we obtain the surjective map 
\[
\pi : C_1 \H(n) \to C_1 \H_{\rm top} (n).
\]
Similarly, we introduce the $\Sp$-submodule $\Y^{\ord}_{r, \langle 3n-2r \rangle} \subset \Y^{\ord}_{r, 3n-2r}$ and the surjective map 
$
\pi : C_n \H^{\ord}(n) \to C_n \H^{\ord}_{\rm top} (n) := \bigoplus_{r=0}^{\lfloor 3n/2 \rfloor} \Y^{\ord}_{r, \langle 3n-2r \rangle}
$.
The map $\beta^{n-1}$ restricts to the map $ C_n \H^{\ord}_{\rm top}(n) \to C_1 \H_{\rm top} (n)$,
which we denote by the same letter.

Conant \cite{Con16} considered the following commutative diagram:
\begin{equation} \label{eq:ConantAddendum}
\xymatrix{
\h(1)^{\otimes n} \ar[r]^{\hspace{-1em} \Tr^{\ord}} \ar[d]_{\br} & C_n \H^{\ord}(n) \ar[d]^{\beta^{n-1}} \ar[r]^{\pi} & C_n \H^{\ord}_{\rm top}(n) \ar[d]^{\beta^{n-1}} \\
\h(n) \ar[r]^{\hspace{-1em} \Tr} & C_1 \H(n) \ar[r]^{\pi} & C_1 \H_{\rm top} (n)
}
\end{equation}
%$C_{1,r,m}\H(n)$
Here, the left vertical map is the left-nested iterated Lie bracket:
\[
\br(X_1 \otimes X_2 \otimes \cdots \otimes X_n) = [[\cdots [X_1, X_2],X_3],\cdots,X_{n-1}],X_n].
\]
Hain's result~\cite{Hai97} is rephrased as ${\rm Im}\, \tau_n^\Q = \m(n) = \br(\h(1)^{\otimes n})$, where $\tau_n^\Q$ is the $n$-th Johnson homomorphism.
Conant observed that the two horizontal maps $\pi \circ \Tr^{\ord}$ and $\pi \circ \Tr$ are both isomorphisms if $g \gg n$. From this, he deduced the following.

\begin{theorem}[Conant \cite{Con16}, Theorem 2.1] \label{thm:conant_detail}
If $g \gg n$, the composition $\pi \circ \Tr$ induces an $\Sp$-module isomorphism
\begin{equation} \label{eq:graphical_h/m}
\h(n)/\m(n) \cong C_1 \H_{\rm top}(n) / \beta^{n-1}(C_n \H^{\rm ord}_{\rm top}(n)).
\end{equation}
\end{theorem}

For each $n \ge 1$, let
\begin{align*} 
\widetilde{\Omega}(n) &:= C_1 \H(n)/ \beta^{n-1}(C_n \H^{\ord}(n)),\\
\widetilde{\Omega}_{\rm top}(n) &:= C_1 \H_{\rm top}(n) / \beta^{n-1}(C_n \H^{\ord}_{\rm top}(n)), 
\end{align*} 
and set $\widetilde{\Omega} := \bigoplus_{n=1}^{\infty} \widetilde{\Omega}(n)$.
The number of dotted edges defines a direct sum decomposition of $\widetilde{\Omega}(n)$.
Set
\begin{align*}
\widetilde{\Omega}_{r,m} &:=  
\T_{r,m}/\beta^{m+2r-3}(\Y_{m+3r-3, m}^{\ord}), \\
\widetilde{\Omega}_{r,\langle m \rangle} &:=
\T_{r,\langle m \rangle}/\beta^{m+2r-3}(\Y_{m+3r-3, \langle m \rangle}^{\ord}).
\end{align*}
Here, recall that $\T_{r,m}$ is the $\Q$-vector space spanned by tree-shaped Jacobi diagrams with $r$ dotted edges and $m$ univalent $H$-colored vertices.
Then, we have 
\[
\widetilde{\Omega}(n) = \bigoplus_{r=0}^{\lfloor n/2 \rfloor + 1} \widetilde{\Omega}_{r,n+2-2r}, 
\qquad 
\widetilde{\Omega}_{\rm top}(n) = 
\bigoplus_{r=0}^{\lfloor n/2 \rfloor + 1} \widetilde{\Omega}_{r,\langle n+2-2r \rangle}.
\]
Note that the space $\widetilde{\Omega}_{0,m}$ is trivial. 
Let $\widetilde{\Tr}$ be the composition
\[
\widetilde{\Tr}: \h(n) \xrightarrow{\Tr} C_1 \H(n) \to \widetilde{\Omega}(n),
\]
where the second map is the natural projection.
It induces the map $\widetilde{\Tr}: \cok(n) = \h(n)/\m(n) \to \widetilde{\Omega}(n)$,
which we denote by the same letter.
Theorem~\ref{thm:conant_detail} says that for $g \gg n$ the composition 
\begin{equation} \label{eq:r-loop_part}
\cok(n) \xrightarrow{\widetilde{\Tr}}
\widetilde{\Omega}(n)
\xrightarrow{\pi}
\widetilde{\Omega}_{\rm top}(n)
\end{equation}
is an $\Sp$-isomorphism, and in particular $\widetilde{\Tr} = \bigoplus_{r=1}^{\lfloor n/2 \rfloor +1} \widetilde{\Tr}_r$ is injective.
We call $\widetilde{\Omega}_{r, \langle n+2-2r \rangle}$ the {\em $r$-loop part} of the Johnson cokernel in degree $n$.

\section{Description of the $2$-loop space} \label{sec:2-loop}

We have $\widetilde{\Omega} = \bigoplus_{r \ge 1} \widetilde{\Omega}_r$, where $\widetilde{\Omega}_r=\bigoplus_{m=0}^{\infty} \widetilde{\Omega}_{r,m}$.
We call $\widetilde{\Omega}_r$ the $r$-loop part of $\widetilde{\Omega}$. 
In this section, we give a presentation of the $2$-loop space $\widetilde{\Omega}_2$. 

We use the following shorthand notation: for any pure tensor $x= x_1\cdots x_m \in \mathcal{T}$ with $x_i \in H$, we define
\[
\begin{tikzpicture}[baseline=-7pt, x=1.6mm, y=1.6mm]
\draw (-3,0) -- (3,0);
\draw[ut] (0,0) -- (0,-1.5) node[below=-3pt]{$x$};
\end{tikzpicture}
:=
\begin{tikzpicture}[baseline=-7pt, x=1.6mm, y=1.6mm]
\draw (-5,0) -- (5,0);
\draw (-3,0) -- (-3,-1.5);
\draw (3,0) -- (3,-1.5);
\draw (-2.75,-2.5) node{$x_1$};
\draw (0,-1.5) node{$\dots$};
\draw (3.5,-2.5) node{$x_m$};
\end{tikzpicture}
\]
and extend it multi-linearly for any $x\in \mathcal{T}$. 
We have
$
\begin{tikzpicture}[baseline=-7pt, x=1.6mm, y=1.6mm]
\draw (-3,0) -- (3,0);
\draw[ut] (0,0) -- (0,-1.5) node[below=-3pt]{$x$};
\end{tikzpicture}
= 
\begin{tikzpicture}[baseline=0pt, x=1.6mm, y=1.6mm]
\draw (-3,0) -- (3,0);
\draw[ut] (0,0) -- (0,1.5) node[above=-3pt]{$\overline{x}$};
\end{tikzpicture}
$, where $\overline{x} = (-1)^m x_m \cdots x_1$.
In a similar fashion, we abbreviate some part of a hairy Lie graph using letters, e.g., we draw   
$
\begin{tikzpicture}[baseline=2pt, x=1.6mm, y=1.6mm]
\draw (-3,0) -- (3,0);
\draw (0,0) -- (0,1.5) node[above=-3pt]{$T$};
\end{tikzpicture}
$
for a rooted $H$-colored tree-shaped Jacobi diagram $T$.

\subsection{Description of $C_{1,2}\H$} \label{subsec:C12H}

We begin with a description of the space $C_{1,2}\H = \bigoplus_{n\ge 1} C_{1,2}\H(n)$.
Up to isomorphism and changing direction of dotted edges, any hairy Lie graph with one tree and two dotted edges of homological degree zero is one of the following two hairy Lie graphs, the theta and dumbbell graphs:
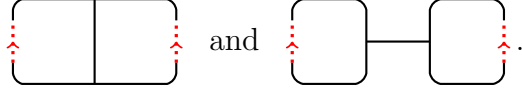
\[
\begin{tikzpicture}[baseline=-3pt, x=1.2mm, y=1.4mm]
\newsmalltheta{2}
\end{tikzpicture}
\quad \text{and} \quad 
\begin{tikzpicture}[baseline=-3pt, x=1.4mm, y=1.4mm]
\newdumbbell
\end{tikzpicture} \ .
\]
Any hairy Lie graph in $C_{1,2}\H$ is obtained from one of these core graphs by attaching several rooted $H$-colored tree-shaped Jacobi diagrams.
Repeated use of the IHX relation at edges adjacent to the core graph and of the AS relation shows that $C_{1,2}\H$ is generated by elements of the following form: 
\begin{equation} \label{eq:C12generators}
\begin{tikzpicture}[baseline=-3pt, x=1.2mm, y=1.44mm]
\newsmalltheta{2}
%%% hairs %%%
%% middle %%%
\draw[ut] (0,0) -- (-1.5,0) node[left=-3pt]{$x$};
%% upper left %%
\draw[ut] (-4.5,4) -- (-4.5,5.5) node[above=-3pt]{$v$};
%% upper right %%
\draw[ut] (4.5,4) -- (4.5,5.5) node[above=-3pt]{$w$};
%% lower left %%
\draw[ut] (-4.5,-4) -- (-4.5,-5.5) node[below=-3pt]{$t$};
%% lower right %%
\draw[ut] (4.5,-4) -- (4.5,-5.5) node[below=-3pt]{$u$};
\end{tikzpicture}
\quad \text{and} \quad 
\begin{tikzpicture}[baseline=-3pt, x=1.4mm, y=1.4mm]
\newdumbbell 
%% middle %%
\draw[ut] (0,0) -- (0,1.5) node[above=-3pt]{$x$};
\draw[ut] (-6.5,-4) -- (-6.5,-5.5) node[below=-3pt]{$t$};
\draw[ut] (6.5,-4) -- (6.5,-5.5) node[below=-3pt]{$u$};
\draw[ut] (-6.5,4) -- (-6.5,5.5) node[above=-3pt]{$v$};
\draw[ut] (6.5,4) -- (6.5,5.5) node[above=-3pt]{$w$};
\end{tikzpicture} \ .
\end{equation}
These generators are subject to the following four types of relations:
\begin{description} 
\item[\rm (ML)] The multi-linearity relation
on $H$-colorings.

\item[\rm (CS)] {\em The core symmetry,} namely, the relation coming from symmetry of the core graphs and the AS relation.
For instance, turning the left cycle of the dumbbell graph upside down gives
\[
\begin{tikzpicture}[baseline=-3pt, x=1.4mm, y=1.4mm]
\newdumbbell 
%% middle %%
\draw[ut] (0,0) -- (0,1.5) node[above=-3pt]{$x$};
\draw[ut] (-6.5,-4) -- (-6.5,-5.5) node[below=-3pt]{$t$};
\draw[ut] (6.5,-4) -- (6.5,-5.5) node[below=-3pt]{$u$};
\draw[ut] (-6.5,4) -- (-6.5,5.5) node[above=-3pt]{$v$};
\draw[ut] (6.5,4) -- (6.5,5.5) node[above=-3pt]{$w$};
\end{tikzpicture}
\ = \    
\begin{tikzpicture}[baseline=-3pt, x=1.4mm, y=1.4mm]
\newdumbbell 
%% middle %%
\draw[ut] (0,0) -- (0,1.5) node[above=-3pt]{$x$};
\draw[ut] (-6.5,-4) -- (-6.5,-5.5) node[below=-3pt]{$\overline{v}$};
\draw[ut] (6.5,-4) -- (6.5,-5.5) node[below=-3pt]{$u$};
\draw[ut] (-6.5,4) -- (-6.5,5.5) node[above=-3pt]{$\overline{t}$};
\draw[ut] (6.5,4) -- (6.5,5.5) node[above=-3pt]{$w$};
\end{tikzpicture} \ .
\]

\item[\rm (CIHX)] {\em The core IHX relation}, namely, the IHX relation around the trivalent vertices of the core graphs. 
For instance, around the upper trivalent vertex of the core theta graph, 
\[ 
\begin{tikzpicture}[baseline=-3pt, x=1.2mm, y=1.44mm]
\newsmalltheta{2}
%%% hairs %%%
%% middle %%%
\draw (0,2) -- (-1.5,2) node[left=-3pt]{$a$};
\draw[ut] (0,-1) -- (-1.5,-1) node[left=-3pt]{$x$};
%% upper left %%
\draw[ut] (-6,4) -- (-6,5.5) node[above=-3pt]{$v$};
%% upper right %%
\draw[ut] (6,4) -- (6,5.5) node[above=-3pt]{$w$};
%% lower left %%
\draw[ut] (-4.5,-4) -- (-4.5,-5.5) node[below=-3pt]{$t$};
%% lower right %%
\draw[ut] (4.5,-4) -- (4.5,-5.5) node[below=-3pt]{$u$};
\end{tikzpicture}
\ = \ 
\begin{tikzpicture}[baseline=-3pt, x=1.2mm, y=1.44mm]
\newsmalltheta{2}
%%% hairs %%%
%% middle %%%
\draw[ut] (0,0) -- (-1.5,0) node[left=-3pt]{$x$};
%% upper left %%
\draw[ut] (-6,4) -- (-6,5.5) node[above=-3pt]{$v$};
%% upper right %%
\draw (3,4) -- (3,5.5) node[above=-3pt]{$a$};
\draw[ut] (6,4) -- (6,5.5) node[above=-3pt]{$w$};
%% lower left %%
\draw[ut] (-4.5,-4) -- (-4.5,-5.5) node[below=-3pt]{$t$};
%% lower right %%
\draw[ut] (4.5,-4) -- (4.5,-5.5) node[below=-3pt]{$u$};
\end{tikzpicture}
\ - \ 
\begin{tikzpicture}[baseline=-3pt, x=1.2mm, y=1.44mm]
\newsmalltheta{2}
%%% hairs %%%
%% middle %%%
\draw[ut] (0,0) -- (-1.5,0) node[left=-3pt]{$x$};
%% upper left %%
\draw (-3,4) -- (-3,5.5) node[above=-3pt]{$a$};
\draw[ut] (-6,4) -- (-6,5.5) node[above=-3pt]{$v$};
%% upper right %%
\draw[ut] (6,4) -- (6,5.5) node[above=-3pt]{$w$};
%% lower left %%
\draw[ut] (-4.5,-4) -- (-4.5,-5.5) node[below=-3pt]{$t$};
%% lower right %%
\draw[ut] (4.5,-4) -- (4.5,-5.5) node[below=-3pt]{$u$};
\end{tikzpicture} \ .
\]
Here, $a\in H$ while the other labels are arbitrary elements in $\mathcal{T}$.
\item[\rm (CC)] {\em The core change relation}. In $C_{1,2}\H$, there is a relation between dumbbell and theta graphs. 
Applying the IHX relation to the handle of the dumbbell graph with $x=1$, we have
\begin{align*}
\begin{tikzpicture}[baseline=-3pt, x=1.2mm, y=1.2mm]
%%%
%%% Mini dumbbell(t,u,v,w) %%%
%%%
%%% middle horizontal line %%%
\draw (-3,0) -- (3,0);  
%%% right dotted edge %%%
\draw[red, very thick, dotted, ->-] (10,-2) -- (10,2);   
%%% left dotted edge %%%
\draw[red, very thick, dotted, ->-] (-10,-2) -- (-10,2); 
%%% big boundaries %%%
\draw (-3,0) to (-3,3) to[bend right=30] (-4,4) to (-9,4) to[bend right=30] (-10,3) to (-10,2);
\draw (-3,0) to (-3,-3) to[bend left=30] (-4,-4) to (-9,-4) to[bend left=30] (-10,-3) to (-10,-2);
\draw (3,0) to (3,3) to[bend left=30] (4,4) to (9,4) to[bend left=30] (10,3) to (10,2);
\draw (3,0) to (3,-3) to[bend right=30] (4,-4) to (9,-4) to[bend right=30] (10,-3) to (10,-2);
%%% hairs %%%
%% upper left %%
\draw[ut] (-6.5,4) -- (-6.5,5.5) node[above=-3pt]{$v$};
%% upper right %%
\draw[ut] (6.5,4) -- (6.5,5.5) node[above=-3pt]{$w$};
%% lower left %%
\draw[ut] (-6.5,-4) -- (-6.5,-5.5) node[below=-3pt]{$t$};
%% lower right %%
\draw[ut] (6.5,-4) -- (6.5,-5.5) node[below=-3pt]{$u$};
\end{tikzpicture}
\ = \ &  
\begin{tikzpicture}[baseline=-3pt, x=1.2mm, y=1.2mm]
%%%
%%% Mini Theta 1 %%%
%%%
%%% right dotted edge %%%
\draw[red, very thick, dotted, ->-] (10,-2) -- (10,2);   
%%% left dotted edge %%%
\draw[red, very thick, dotted, ->-] (-10,-2) -- (-10,2); 
%%% big boundaries %%%
\draw (0,4) to (-9,4) to[bend right=30] (-10,3) to (-10,2);
\draw (0,-4) to (-9,-4) to[bend left=30] (-10,-3) to (-10,-2);
\draw (0,4) to (9,4) to[bend left=30] (10,3) to (10,2);
\draw (0,-4) to (9,-4) to[bend right=30] (10,-3) to (10,-2);
\draw (3,-4) -- (3,4);
%%% hairs %%%
%% upper left %%
\draw[ut] (-6.5,4) -- (-6.5,5.5) node[above=-3pt]{$v$};
%% upper right %%
\draw[ut] (6.5,4) -- (6.5,5.5) node[above=-3pt]{$w$};
%% lower left %%
\draw[ut] (-6.5,-4) -- (-6.5,-5.5) node[below=-3pt]{$t$};
%% lower right %%
\draw[ut] (6.5,-4) -- (6.5,-5.5) node[below=-3pt]{$u$};
\end{tikzpicture}
\ - \  
\begin{tikzpicture}[baseline=-3pt, x=1.2mm, y=1.2mm]
%%%
%%% Mini Theta 2 %%%
%%%
%%% right dotted edge %%%
\draw[red, very thick, dotted, ->-] (10,-2) -- (10,2);   
%%% left dotted edge %%%
\draw[red, very thick, dotted, ->-] (-10,-2) -- (-10,2); 
%%% big boundaries %%%
\draw (-3,4) to (-9,4) to[bend right=30] (-10,3) to (-10,2);
\draw (-3,-4) to (-9,-4) to[bend left=30] (-10,-3) to (-10,-2);
\draw (3,4) to (9,4) to[bend left=30] (10,3) to (10,2);
\draw (3,-4) to (9,-4) to[bend right=30] (10,-3) to (10,-2);
\draw (3,-4) -- (3,4);
%%% crossing in the middle %%%
\draw {[rounded corners=2pt] (-3,-4) -- (-2,-4) -- (2,4) -- (3,4)};
\draw {[rounded corners=2pt] (-3,4) -- (-2,4) -- (-0.5,1)};
\draw {[rounded corners=2pt] (3,-4) -- (2,-4) -- (0.5,-1)};
%%% hairs %%%
%% upper left %%
\draw[ut] (-6.5,4) -- (-6.5,5.5) node[above=-3pt]{$v$};
%% upper right %%
\draw[ut] (6.5,4) -- (6.5,5.5) node[above=-3pt]{$w$};
%% lower left %%
\draw[ut] (-6.5,-4) -- (-6.5,-5.5) node[below=-3pt]{$t$};
%% lower right %%
\draw[ut] (6.5,-4) -- (6.5,-5.5) node[below=-3pt]{$u$};
\end{tikzpicture} \\
\ = \ &
\begin{tikzpicture}[baseline=-3pt, x=1.2mm, y=1.44mm]
\newsmalltheta{2}
%%% hairs %%%
%% upper left %%
\draw[ut] (-4.5,4) -- (-4.5,5.5) node[above=-3pt]{$v$};
%% upper right %%
\draw[ut] (4.5,4) -- (4.5,5.5) node[above=-3pt]{$w$};
%% lower left %%
\draw[ut] (-4.5,-4) -- (-4.5,-5.5) node[below=-3pt]{$t$};
%% lower right %%
\draw[ut] (4.5,-4) -- (4.5,-5.5) node[below=-3pt]{$u$};
\end{tikzpicture}
\ + \ 
\begin{tikzpicture}[baseline=-3pt, x=1.2mm, y=1.44mm]
\newsmalltheta{2}
%%% hairs %%%
%% upper left %%
\draw[ut] (-4.5,4) -- (-4.5,5.5) node[above=-3pt]{$\overline{t}$};
%% upper right %%
\draw[ut] (4.5,4) -- (4.5,5.5) node[above=-3pt]{$w$};
%% lower left %%
\draw[ut] (-4.5,-4) -- (-4.5,-5.5) node[below=-3pt]{$\overline{v}$};
%% lower right %%
\draw[ut] (4.5,-4) -- (4.5,-5.5) node[below=-3pt]{$u$};
\end{tikzpicture} \ .
\end{align*}
\end{description}
Thus we have
\begin{equation} \label{eq:C12H}
C_{1,2}\H \cong 
\frac{\Q \{ \text{generators in \eqref{eq:C12generators}} \}}{\text{(ML), ({CS}), ({CIHX}), ({CC})}}.
\end{equation}

\subsection{Generalized Conant's third relation}

We show the following theorem, which describes $\widetilde{\Omega}_2$ as a quotient of $C_{1,2}\H$ by a single type of relations.

\begin{theorem}\label{thm:omega2}
\par
The $\Q$-vector space $\widetilde{\Omega}_2$ is isomorphic to the quotient of $C_{1,2}\H$ by the following relation: 
for any $x,y,t,u,v,w \in \mathcal{T}$, 
\[
\begin{tikzpicture}[baseline=-3pt, x=1.6mm, y=1.6mm]
\newbigtheta{1}
%%% hairs %%%
%% middle%%
\draw[ut] (0,4) -- (-1.5,4) node[left=-3pt]{$x$};
\draw[ut] (0,-4) -- (-1.5,-4) node[left=-3pt]{$y$};
%% upper left %%
\draw[ut] (-4,6) -- (-4,7.5) node[above=-3pt]{$v$};
%% upper right %%
\draw[ut] (4,6) -- (4,7.5) node[above=-3pt]{$w$};
%% lower left %%
\draw[ut] (-4,-6) -- (-4,-7.5) node[below=-3pt]{$t$};
%% lower right %%
\draw[ut] (4,-6) -- (4,-7.5) node[below=-3pt]{$u$};
\end{tikzpicture}
\, + \, 
\begin{tikzpicture}[baseline=-3pt, x=1.6mm, y=1.6mm]
\newbigtheta{2}
%%% hairs %%%
%% middle%%
\draw[ut] (0,4) -- (-1.5,4) node[left=-3pt]{$x$};
\draw[ut] (0,-4) -- (-1.5,-4) node[left=-3pt]{$y$};
%% upper left %%
\draw[ut] (-4,6) -- (-4,7.5) node[above=-3pt]{$v$};
%% upper right %%
\draw[ut] (4,6) -- (4,7.5) node[above=-3pt]{$w$};
%% lower left %%
\draw[ut] (-4,-6) -- (-4,-7.5) node[below=-3pt]{$t$};
%% lower right %%
\draw[ut] (4,-6) -- (4,-7.5) node[below=-3pt]{$u$};
\end{tikzpicture}
\, + \, 
\begin{tikzpicture}[baseline=-3pt, x=1.6mm, y=1.6mm]
\newbigtheta{3}
%%% hairs %%%
%% middle%%
\draw[ut] (0,4) -- (-1.5,4) node[left=-3pt]{$x$};
\draw[ut] (0,-4) -- (-1.5,-4) node[left=-3pt]{$y$};
%% upper left %%
\draw[ut] (-4,6) -- (-4,7.5) node[above=-3pt]{$v$};
%% upper right %%
\draw[ut] (4,6) -- (4,7.5) node[above=-3pt]{$w$};
%% lower left %%
\draw[ut] (-4,-6) -- (-4,-7.5) node[below=-3pt]{$t$};
%% lower right %%
\draw[ut] (4,-6) -- (4,-7.5) node[below=-3pt]{$u$};
\end{tikzpicture}
= 0.
\tag{GC3}
\]
\end{theorem}

As a special case ($x=v=w=1$) of the relation (GC3), we obtain the relation (C3) for another $r$-loop space introduced by Conant~\cite{Con15} (see Section~\ref{subsec:another_r}).
We call (GC3) the {\em generalized Conant's third relation}.

Let $R$ be the subspace of $C_{1,2}\H$ generated by elements of the form on the left hand side of (GC3).
By definition, the space $\widetilde{\Omega}_2$ is the cokernel of 
\begin{equation} \label{eq:cokbeta}
\bigoplus_{n\ge 1} \beta^{n-1}: \bigoplus_{n\ge 1} C_{n,n+1} \H^{\ord}(n)
\to C_{1,2} \H.
\end{equation}
We show that the image of this map coincides with $R$.

For $t,u,v,w\in \mathcal{T}$, set 
\begin{equation} \label{eq:GC3_special}
\Theta_R(t,u,v,w) :=
\begin{tikzpicture}[baseline=-3pt, x=1.2mm, y=1.44mm]
\newsmalltheta{1}
%%% hairs %%%
%% upper left %%
\draw[ut] (-4.5,4) -- (-4.4,5.5) node[above=-3pt]{$v$};
%% upper right %%
\draw[ut] (4.5,4) -- (4.5,5.5) node[above=-3pt]{$w$};
%% lower left %%
\draw[ut] (-4.5,-4) -- (-4.5,-5.5) node[below=-3pt]{$t$};
%% lower right %%
\draw[ut] (4.5,-4) -- (4.5,-5.5) node[below=-3pt]{$u$};
\end{tikzpicture}
\, + \, 
\begin{tikzpicture}[baseline=-3pt, x=1.2mm, y=1.44mm]
\newsmalltheta{2}
%%% hairs %%%
%% upper left %%
\draw[ut] (-4.5,4) -- (-4.5,5.5) node[above=-3pt]{$v$};
%% upper right %%
\draw[ut] (4.5,4) -- (4.5,5.5) node[above=-3pt]{$w$};
%% lower left %%
\draw[ut] (-4.5,-4) -- (-4.5,-5.5) node[below=-3pt]{$t$};
%% lower right %%
\draw[ut] (4.5,-4) -- (4.5,-5.5) node[below=-3pt]{$u$};
\end{tikzpicture}
\, + \, 
\begin{tikzpicture}[baseline=-3pt, x=1.2mm, y=1.44mm]
\newsmalltheta{3}
%%% hairs %%%
%% upper left %%
\draw[ut] (-4.5,4) -- (-4.5,5.5) node[above=-3pt]{$v$};
%% upper right %%
\draw[ut] (4.5,4) -- (4.5,5.5) node[above=-3pt]{$w$};
%% lower left %%
\draw[ut] (-4.5,-4) -- (-4.5,-5.5) node[below=-3pt]{$t$};
%% lower right %%
\draw[ut] (4.5,-4) -- (4.5,-5.5) node[below=-3pt]{$u$};
\end{tikzpicture} \ .
\end{equation}
Any generator of $R$ can be written as a linear combination of elements of this type, since we can move the hairs colored with $x,y$ on the middle vertical edges to horizontal edges.
Therefore, the elements $\Theta_R(t,u,v,w)$ generate the space $R$.

Theorem~\ref{thm:omega2} is proved in the next subsection.
The rest of this subsection is devoted to some preliminaries.
\begin{lemma}\label{lem:elementsR}
The following elements in $C_{1,2}\H$ are contained in $R$.
\begin{enumerate}
\renewcommand{\theenumi}{\arabic{enumi}}
\renewcommand{\labelenumi}{(\theenumi)}
\item
For any $t,u,v,w,x,y \in \mathcal{T}$ and $a \in\mathcal{T}$, 
\[
\begin{tikzpicture}[baseline=-3pt, x=1.4mm, y=1.4mm]
\newbigtheta{3}
%%% hairs %%%
%% middle%%
\draw[ut] (0,4) -- (-1.5,4) node[left=-3pt]{$x$};
\draw[ut] (0,-4) -- (-1.5,-4) node[left=-3pt]{$y$};
%% upper left %%
\draw[ut] (-6,6) -- (-6,7.5) node[above=-3pt]{$a$};
\draw[ut] (-3,6) -- (-3,7.5) node[above=-3pt]{$v$};
%% upper right %%
\draw[ut] (4,6) -- (4,7.5) node[above=-3pt]{$w$};
%% lower left %%
\draw[ut] (-3,-6) -- (-3,-7.5) node[below=-3pt]{$t$};
%% lower right %%
\draw[ut] (4,-6) -- (4,-7.5) node[below=-3pt]{$u$};
\end{tikzpicture}
\, - \, 
\begin{tikzpicture}[baseline=-3pt, x=1.4mm, y=1.4mm]
\newbigtheta{3}
%%% hairs %%%
%% middle%%
\draw[ut] (0,4) -- (-1.5,4) node[left=-3pt]{$x$};
\draw[ut] (0,-4) -- (-1.5,-4) node[left=-3pt]{$y$};
%% upper left %%
\draw[ut] (-3,6) -- (-3,7.5) node[above=-3pt]{$v$};
%% upper right %%
\draw[ut] (4,6) -- (4,7.5) node[above=-3pt]{$w$};
%% lower left %%
\draw[ut] (-6,-6) -- (-6,-7.5) node[below=-2pt]{$a$};
\draw[ut] (-3,-6) -- (-3,-7.5) node[below=-4pt]{$t$};
%% lower right %%
\draw[ut] (4,-6) -- (4,-7.5) node[below=-3pt]{$u$};
\end{tikzpicture}
\, + \, 
\begin{tikzpicture}[baseline=-3pt, x=1.4mm, y=1.4mm]
\newbigtheta{2}
%%% hairs %%%
%% middle%%
\draw[ut] (0,4) -- (-1.5,4) node[left=-3pt]{$x$};
\draw[ut] (0,-4) -- (-1.5,-4) node[left=-3pt]{$y$};
%% upper left %%
\draw[ut] (-6,6) -- (-6,7.5) node[above=-3pt]{$a$};
\draw[ut] (-3,6) -- (-3,7.5) node[above=-3pt]{$v$};
%% upper right %%
\draw[ut] (4,6) -- (4,7.5) node[above=-3pt]{$w$};
%% lower left %%
\draw[ut] (-3,-6) -- (-3,-7.5) node[below=-3pt]{$t$};
%% lower right %%
\draw[ut] (4,-6) -- (4,-7.5) node[below=-3pt]{$u$};
\end{tikzpicture}
\, - \, 
\begin{tikzpicture}[baseline=-3pt, x=1.4mm, y=1.4mm]
\newbigtheta{2}
%%% hairs %%%
%% middle%%
\draw[ut] (0,4) -- (-1.5,4) node[left=-3pt]{$x$};
\draw[ut] (0,-4) -- (-1.5,-4) node[left=-3pt]{$y$};
%% upper left %%
\draw[ut] (-3,6) -- (-3,7.5) node[above=-3pt]{$v$};
%% upper right %%
\draw[ut] (4,6) -- (4,7.5) node[above=-3pt]{$w$};
%% lower left %%
\draw[ut] (-6,-6) -- (-6,-7.5) node[below=-2pt]{$a$};
\draw[ut] (-3,-6) -- (-3,-7.5) node[below=-4pt]{$t$};
%% lower right %%
\draw[ut] (4,-6) -- (4,-7.5) node[below=-3pt]{$u$};
\end{tikzpicture} \ .
\]

\item 
For any $t,v \in \mathcal{T}$ and $a \in  \mathcal{T}$, 
\[
\begin{tikzpicture}[baseline=-3pt, x=1.4mm, y=1.4mm]
\newsmalltheta{2}
%% upper left %%
\draw[ut] (-6,4) -- (-6,5.5) node[above=-3pt]{$a$};
\draw[ut] (-3,4) -- (-3,5.5) node[above=-3pt]{$v$};
%% lower left %%
\draw[ut] (-3,-4) -- (-3,-5.5) node[below=-3pt]{$t$};
\end{tikzpicture}
\ - \ 
\begin{tikzpicture}[baseline=-3pt, x=1.4mm, y=1.4mm]
\newsmalltheta{2}
%% upper left %%
\draw[ut] (-3,4) -- (-3,5.5) node[above=-3pt]{$v$};
%% lower left %%
\draw[ut] (-6,-4) -- (-6,-5.5) node[below=-2pt]{$a$};
\draw[ut] (-3,-4) -- (-3,-5.5) node[below=-4pt]{$t$};
\end{tikzpicture} \ .  
\]

\item 
For any $t,u,v,w,x \in \mathcal{T}$ and $a,b \in \mathcal{T}$, 
\[
\begin{tikzpicture}[baseline=-3pt, x=1.4mm, y=1.4mm]
\newbigtheta{2}
%%% hairs %%%
%% middle%%
\draw[ut] (0,0) -- (-1.5,0) node[left=-3pt]{$x$};
%% upper left %%
\draw[ut] (-6,6) -- (-6,7.5) node[above=-3pt]{$a$};
\draw[ut] (-3,6) -- (-3,7.5) node[above=-3pt]{$v$};
%% upper right %%
\draw[ut] (3,6) -- (3,7.5) node[above=-3pt]{$w$};
\draw[ut] (6,6) -- (6,7.5) node[above=-3pt]{$b$}; 
%% lower left %%
\draw[ut] (-3,-6) -- (-3,-7.5) node[below=-4pt]{$t$};
%% lower right %%
\draw[ut] (3,-6) -- (3,-7.5) node[below=-3pt]{$u$};
\end{tikzpicture}
\, - \, 
\begin{tikzpicture}[baseline=-3pt, x=1.4mm, y=1.4mm]
\newbigtheta{2}
%%% hairs %%%
%% middle%%
\draw[ut] (0,0) -- (-1.5,0) node[left=-3pt]{$x$};
%% upper left %%
\draw[ut] (-6,6) -- (-6,7.5) node[above=-3pt]{$a$};
\draw[ut] (-3,6) -- (-3,7.5) node[above=-3pt]{$v$};
%% upper right %%
\draw[ut] (3,6) -- (3,7.5) node[above=-3pt]{$w$};
%% lower left %%
\draw[ut] (-3,-6) -- (-3,-7.5) node[below=-4pt]{$t$};
%% lower right %%
\draw[ut] (3,-6) -- (3,-7.5) node[below=-3pt]{$u$}; 
\draw[ut] (6,-6) -- (6,-7.5) node[below=-3pt]{$b$};
\end{tikzpicture}
\, - \, 
\begin{tikzpicture}[baseline=-3pt, x=1.4mm, y=1.4mm]
\newbigtheta{2}
%%% hairs %%%
%% middle%%
\draw[ut] (0,0) -- (-1.5,0) node[left=-3pt]{$x$};
%% upper left %%
\draw[ut] (-3,6) -- (-3,7.5) node[above=-3pt]{$v$};
%% upper right %%
\draw[ut] (6,6) -- (6,7.5) node[above=-3pt]{$b$}; 
\draw[ut] (3,6) -- (3,7.5) node[above=-3pt]{$w$};
%% lower left %%
\draw[ut] (-6,-6) -- (-6,-7.5) node[below=-3pt]{$a$};
\draw[ut] (-3,-6) -- (-3,-7.5) node[below=-4pt]{$t$};
%% lower right %% 
\draw[ut] (3,-6) -- (3,-7.5) node[below=-3pt]{$u$};
\end{tikzpicture}
\, + \, 
\begin{tikzpicture}[baseline=-3pt, x=1.4mm, y=1.4mm]
\newbigtheta{2}
%%% hairs %%%
%% middle%%
\draw[ut] (0,0) -- (-1.5,0) node[left=-3pt]{$x$};
%% upper left %%
\draw[ut] (-3,6) -- (-3,7.5) node[above=-3pt]{$v$};
%% upper right %%
\draw[ut] (3,6) -- (3,7.5) node[above=-3pt]{$w$};
%% lower left %%
\draw[ut] (-6,-6) -- (-6,-7.5) node[below=-3pt]{$a$};
\draw[ut] (-3,-6) -- (-3,-7.5) node[below=-4pt]{$t$};
%% lower right %% 
\draw[ut] (6,-6) -- (6,-7.5) node[below=-3pt]{$b$}; 
\draw[ut] (4,-6) -- (4,-7.5) node[below=-3pt]{$u$};
\end{tikzpicture} \ .
\]
\end{enumerate}
\end{lemma}

\begin{proof}
(1)
Moving the hairs colored with $x$ and colored with $y$ to the upper and lower horizontal edges by the IHX relation, respectively, we may assume that $x=y=1$.
Then, the element in (1) is equal to
\begin{align*}
& 
\begin{tikzpicture}[baseline=-3pt, x=1.4mm, y=1.4mm]
\newsmalltheta{1}
%% upper left %%
\draw[ut] (-6,4) -- (-6,5.5) node[above=-3pt]{$a$};
\draw[ut] (-3,4) -- (-3,5.5) node[above=-3pt]{$v$};
%% lower left %%
\draw[ut] (-3,-4) -- (-3,-5.5) node[below=-3pt]{$t$};
%% upper right %%
\draw[ut] (4,4) -- (4,5.5) node[above=-3pt]{$w$};
%% lower right %% 
\draw[ut] (4,-4) -- (4,-5.5) node[below=-3pt]{$u$};
\end{tikzpicture}
\ + \ 
\begin{tikzpicture}[baseline=-3pt, x=1.4mm, y=1.4mm]
\newsmalltheta{2}
%% upper left %%
\draw[ut] (-6,4) -- (-6,5.5) node[above=-3pt]{$a$};
\draw[ut] (-3,4) -- (-3,5.5) node[above=-3pt]{$v$};
%% lower left %%
\draw[ut] (-3,-4) -- (-3,-5.5) node[below=-3pt]{$t$};
%% upper right %%
\draw[ut] (4,4) -- (4,5.5) node[above=-3pt]{$w$};
%% lower right %% 
\draw[ut] (4,-4) -- (4,-5.5) node[below=-3pt]{$u$};
\end{tikzpicture}
\ + \ 
\begin{tikzpicture}[baseline=-3pt, x=1.4mm, y=1.4mm]
\newsmalltheta{3}
%% upper left %%
\draw[ut] (-6,4) -- (-6,5.5) node[above=-3pt]{$a$};
\draw[ut] (-3,4) -- (-3,5.5) node[above=-3pt]{$v$};
%% lower left %%
\draw[ut] (-3,-4) -- (-3,-5.5) node[below=-3pt]{$t$};
%% upper right %%
\draw[ut] (4,4) -- (4,5.5) node[above=-3pt]{$w$};
%% lower right %% 
\draw[ut] (4,-4) -- (4,-5.5) node[below=-3pt]{$u$};
\end{tikzpicture} \\
& - 
\begin{tikzpicture}[baseline=-3pt, x=1.4mm, y=1.4mm]
\newsmalltheta{1}
%% upper left %%
\draw[ut] (-6,-4) -- (-6,-5.5) node[below=-3pt]{$a$};
\draw[ut] (-3,4) -- (-3,5.5) node[above=-3pt]{$v$};
%% lower left %%
\draw[ut] (-3,-4) -- (-3,-5.5) node[below=-3pt]{$t$};
%% upper right %%
\draw[ut] (4,4) -- (4,5.5) node[above=-3pt]{$w$};
%% lower right %% 
\draw[ut] (4,-4) -- (4,-5.5) node[below=-3pt]{$u$};
\end{tikzpicture}
\ - \ 
\begin{tikzpicture}[baseline=-3pt, x=1.4mm, y=1.4mm]
\newsmalltheta{2}
%% upper left %%
\draw[ut] (-6,-4) -- (-6,-5.5) node[below=-3pt]{$a$};
\draw[ut] (-3,4) -- (-3,5.5) node[above=-3pt]{$v$};
%% lower left %%
\draw[ut] (-3,-4) -- (-3,-5.5) node[below=-3pt]{$t$};
%% upper right %%
\draw[ut] (4,4) -- (4,5.5) node[above=-3pt]{$w$};
%% lower right %% 
\draw[ut] (4,-4) -- (4,-5.5) node[below=-3pt]{$u$};
\end{tikzpicture}
\ - \ 
\begin{tikzpicture}[baseline=-3pt, x=1.4mm, y=1.4mm]
\newsmalltheta{3}
%% upper left %%
\draw[ut] (-6,-4) -- (-6,-5.5) node[below=-3pt]{$a$};
\draw[ut] (-3,4) -- (-3,5.5) node[above=-3pt]{$v$};
%% lower left %%
\draw[ut] (-3,-4) -- (-3,-5.5) node[below=-3pt]{$t$};
%% upper right %%
\draw[ut] (4,4) -- (4,5.5) node[above=-3pt]{$w$};
%% lower right %% 
\draw[ut] (4,-4) -- (4,-5.5) node[below=-3pt]{$u$};
\end{tikzpicture} \\ 
= \ & \, \Theta_R(t,u,va,w) - \Theta_R(at,u,v,w).
\end{align*}
Hence, it is an element of $R$.

(2)
This is a special case of (1). Setting $x=y=w=u=1$ in (1) and applying the AS relation, we obtain the assertion. 

(3)
By the IHX relation, we may assume that $x=1$.
Then, similarly to the proof of (1), the element in (3) is equal to 
\[
\Theta_R(t,u,va,bw) - \Theta_R(t,ub,va,w)
- \Theta_R(at,u,v,bw) + \Theta_R(at,ub,v,w).
\]
This completes the proof of the lemma.
\end{proof}

\begin{lemma} \label{lem:elementsR2}
The following elements in $C_{1,2}\H$ are contained in $R$.
\begin{enumerate}
    \item For any $t,v,x \in \mathcal{T}$ and $a \in \mathcal{T}$,
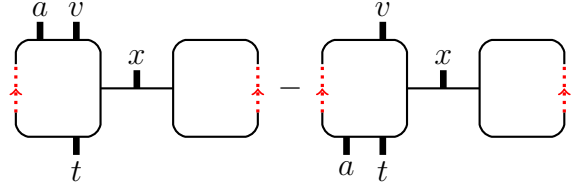
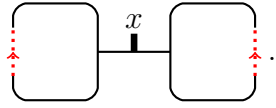
\[
\begin{tikzpicture}[baseline=-3pt, x=1.6mm, y=1.6mm]
\newdumbbell
\draw[ut] (0,0) -- (0,1.5) node[above=-3pt]{$x$};
\draw[ut] (-8,4) -- (-8,5.5) node[above=-3pt]{$a$};
\draw[ut] (-5,4) -- (-5,5.5) node[above=-3pt]{$v$};
\draw[ut] (-5,-4) -- (-5,-5.5) node[below=-3pt]{$t$}; 
\end{tikzpicture}
\ - \ 
\begin{tikzpicture}[baseline=-3pt, x=1.6mm, y=1.6mm]
\newdumbbell
\draw[ut] (0,0) -- (0,1.5) node[above=-3pt]{$x$};
\draw[ut] (-8,-4) -- (-8,-5.5) node[below=-3pt]{$a$};
\draw[ut] (-5,4) -- (-5,5.5) node[above=-3pt]{$v$};
\draw[ut] (-5,-4) -- (-5,-5.5) node[below=-3pt]{$t$}; 
\end{tikzpicture} \ . 
\]
\item For any $x \in \mathcal{T}$, 
\[
\begin{tikzpicture}[baseline=-3pt, x=1.6mm, y=1.6mm]
\newdumbbell 
\draw[ut] (0,0) -- (0,1.5) node[above=-3pt]{$x$};
\end{tikzpicture} \ .
\]
\end{enumerate}
\end{lemma}

\begin{proof}
(1) Moving the hair colored with $x$ to the left part of the dumbbell by the IHX relation, we may assume that $x=1$.
Applying the IHX relation to the handle of the dumbbell, we have
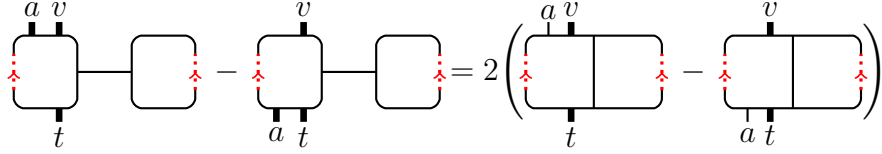
\[
\begin{tikzpicture}[baseline=-3pt, x=1.2mm, y=1.2mm]
\newdumbbell
\draw[ut] (-8,4) -- (-8,5.5) node[above=-3pt]{$a$};
\draw[ut] (-5,4) -- (-5,5.5) node[above=-3pt]{$v$};
\draw[ut] (-5,-4) -- (-5,-5.5) node[below=-3pt]{$t$}; 
\end{tikzpicture}
\ - \ 
\begin{tikzpicture}[baseline=-3pt, x=1.2mm, y=1.2mm]
\newdumbbell
\draw[ut] (-8,-4) -- (-8,-5.5) node[below=-3pt]{$a$};
\draw[ut] (-5,4) -- (-5,5.5) node[above=-3pt]{$v$};
\draw[ut] (-5,-4) -- (-5,-5.5) node[below=-3pt]{$t$}; 
\end{tikzpicture}
= 
2 \Biggl( 
\begin{tikzpicture}[baseline=-3pt, x=1mm, y=1.2mm]
\newsmalltheta{2}
%% upper left %%
\draw (-6,4) -- (-6,5.5) node[above=-3pt]{$a$};
\draw[ut] (-3,4) -- (-3,5.5) node[above=-3pt]{$v$};
%% lower left %%
\draw[ut] (-3,-4) -- (-3,-5.5) node[below=-3pt]{$t$};
\end{tikzpicture}
\ - \ 
\begin{tikzpicture}[baseline=-3pt, x=1mm, y=1.2mm]
\newsmalltheta{2}
%% upper left %%
\draw[ut] (-3,4) -- (-3,5.5) node[above=-3pt]{$v$};
%% lower left %%
\draw (-6,-4) -- (-6,-5.5) node[below=-2pt]{$a$};
\draw[ut] (-3,-4) -- (-3,-5.5) node[below=-4pt]{$t$};
\end{tikzpicture} \Biggr)
\]
in $C_{1,2}\H$. 
By Lemma~\ref{lem:elementsR}(2), the right hand side lies in $R$.

(2) We may assume that $x$ is a pure tensor. 
Setting $t=u=v=w=x=y=1$ on the right hand side of (GC3), we see that the element 
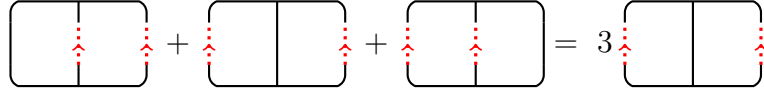
\[
\begin{tikzpicture}[baseline=-3pt, x=1mm, y=1.4mm]
\newsmalltheta{1}
\end{tikzpicture} 
\ + \ 
\begin{tikzpicture}[baseline=-3pt, x=1mm, y=1.4mm]
\newsmalltheta{2}
\end{tikzpicture} 
\ + \ 
\begin{tikzpicture}[baseline=-3pt, x=1mm, y=1.4mm]
\newsmalltheta{3}
\end{tikzpicture}
= \ 3 \ 
\begin{tikzpicture}[baseline=-3pt, x=1mm, y=1.4mm]
\newsmalltheta{2}
\end{tikzpicture} 
\]
lies in $R$.
Furthermore, by the IHX relation we have 
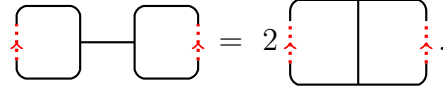
\[
\begin{tikzpicture}[baseline=-3pt, x=1.2mm, y=1.2mm]
\newdumbbell 
\end{tikzpicture}
\ = \ 
2 \ \begin{tikzpicture}[baseline=-3pt, x=1mm, y=1.4mm]
\newsmalltheta{2}
\end{tikzpicture} \ .
\]
This proves the assertion for the case $\deg x=0$. 
If $\deg x >0$, then one can write $x = x'a$ with $a \in H$ and $x' \in \mathcal{T}$.
By the IHX relation, we have
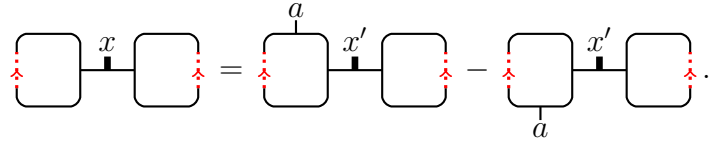
\[
\begin{tikzpicture}[baseline=-3pt, x=1.2mm, y=1.2mm]
\newdumbbell 
\draw[ut] (0,0) -- (0,1.5) node[above=-3pt]{$x$};
\end{tikzpicture}
\ = \  
\begin{tikzpicture}[baseline=-3pt, x=1.2mm, y=1.2mm]
\newdumbbell 
\draw[ut] (0,0) -- (0,1.5) node[above=-3pt]{$x'$};
\draw (-6.5,4) -- (-6.5,5.5) node[above=-3pt]{$a$};
\end{tikzpicture}
\ - \
\begin{tikzpicture}[baseline=-3pt, x=1.2mm, y=1.2mm]
\newdumbbell 
\draw[ut] (0,0) -- (0,1.5) node[above=-3pt]{$x'$};
\draw (-6.5,-4) -- (-6.5,-5.5) node[below=-3pt]{$a$};
\end{tikzpicture} \ .
\]
By (1), the right hand side lies in $R$. 
\end{proof}

Let $X$ be a hairy Lie graph with ordered $n$ tripods $Y_1,\ldots,Y_n$. 
We regard $X$ as an element of $C_n\H^\ord(n)$. 
We arrange that any dotted edge connecting two tripods $Y_i$ and $Y_j$ with $i<j$ is directed from $Y_i$ to $Y_j$.
Here is an example: 
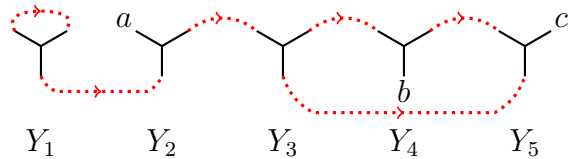
\[
\begin{tikzpicture}[x=4mm, y=4mm]
%%% 1st tree %%%%
\draw (0,0)--(0,-1);
\draw (0,0)--({sqrt(3)/2},1/2);
\draw (0,0)--({-sqrt(3)/2},1/2);
%%% 2nd tree %%%%
\draw (4,0)--(4,-1);
\draw (4,0)--({sqrt(3)/2 + 4},1/2);
\draw (4,0)--({-sqrt(3)/2 + 4},1/2) node[above left=-3pt]{$a$};
%%% 3rd tree %%%%
\draw (8,0)--(8,-1);
\draw (8,0)--({sqrt(3)/2 + 8},1/2);
\draw (8,0)--({-sqrt(3)/2 + 8},1/2);
%%% 4th tree %%%%
\draw (12,0)--(12,-1) node[below=-3pt]{$b$};
\draw (12,0)--({sqrt(3)/2 + 12},1/2);
\draw (12,0)--({-sqrt(3)/2 + 12},1/2);
%%% 5th tree %%%%
\draw (16,0)--(16,-1);
\draw (16,0)--({sqrt(3)/2 + 16},1/2) node[above right=-3pt]{$c$};
\draw (16,0)--({-sqrt(3)/2 + 16},1/2);
%%% labels to Y_i %%%
\draw (0,-3.2) node{$Y_1$};
\draw (4,-3.2) node{$Y_2$};
\draw (8,-3.2) node{$Y_3$};
\draw (12,-3.2) node{$Y_4$};
\draw (16,-3.2) node{$Y_5$};
%%%% dotted edges
\draw[red, very thick, dotted, ->-] (-0.866,1/2) to[bend left=20] (-0.9,0.7) to[bend left=120] (0.9,0.7) to[bend left=20] (0.866,1/2);
\draw[red, very thick, dotted, ->-] (0,-1) to[bend right=20] (0.5,-1.5) to (3.5,-1.5) to[bend right=20] (4,-1);
\draw[red, very thick, dotted, ->-] ({sqrt(3)/2 + 4},1/2) to[bend left=40] ({-sqrt(3)/2 + 8},1/2);
\draw[red, very thick, dotted, ->-] (8,-1) to[bend right=20] (9,-2.2) to (15,-2.2) to[bend right=20] (16,-1);
\draw[red, very thick, dotted, ->-] ({sqrt(3)/2 + 8},1/2) to[bend left=40] ({-sqrt(3)/2 + 12},1/2);
\draw[red, very thick, dotted, ->-] ({sqrt(3)/2 + 12},1/2) to[bend left=40] ({-sqrt(3)/2 + 16},1/2);
\end{tikzpicture}
\]
Let $i \in \{1,\ldots,n\}$.
A {\em self-loop edge} of $Y_i$ is a dotted edge of $X$ whose both ends are attached to $Y_i$.
A {\em backward edge} of $Y_i$ is a dotted edge
attached to $Y_i$ whose other end is attached to $Y_j$ for some $j < i$. 
If $n\ge 2$ and there exists some $i \in \{ 1,\ldots, n\}$ such that $Y_i$ has no backward edge, then $\beta^{n-1}(X) = 0$. 
For $i \ge 2$, let $N(Y_i)$ be the number of self-loop edges and backward edges of $Y_i$ minus one.
We also set $N(Y_1) = 1$ if $Y_1$ has a self-loop edge, and $N(Y_1) =0$ otherwise.
Note that one backward edge of $Y_i$ will be changed to a solid edge when we apply to $X$ the $(i-1)$-th application of $\beta$, and the number $N(Y_i)$ counts ``extra'' dotted edges attached to $Y_i$ which become dotted edges of a term in $\beta^{n-1}(X)$.
Thus, if $\beta^{n-1}(X)\ne 0$ and lies in $C_{1,r}\H(n)$,
namely it has $r$ dotted edges, 
then $X \in C_{n,n-1+r}\H^\ord(n)$ and $\sum_{i=1}^n N(Y_i) = r$.

\subsection{Proof of Theorem~\ref{thm:omega2}}
We divide the proof into several steps. 

{\em Step 1}. 
We first show that $\beta^{n-1}(C_{n,n+1}\H^\ord(n))\subset R$ for any $n$.
Let $X \in C_{n,n+1}\H^\ord(n)$ be a hairy Lie graph such that $\beta^{n-1}(X) \neq 0$. 
Then, for any $2 \le i \le n$, the tripod $Y_i$ has at least one backward edge. 
Since $\sum_{i=1}^n N(Y_i) = 2$, there are two cases: 
\begin{enumerate}
\item[(I)] $N(Y_i)=2$ for some $i \in \{2,\ldots, n\}$,
\item[(II)] $N(Y_i)=1$ and $N(Y_{i'})=1$ for some $i, i' \in \{1,\ldots,n\}$ with $i \ne i'$.
\end{enumerate}

In the case (I), $Y_i$ is incident to three backward edges and $\beta^{n-1}(X)$ is a sum of three terms, which is a generator of $R$ (with $v = w = x = 1$).

In the case (II), we have three subcases: 
\begin{enumerate}
\item[(II-i)] Both of $Y_i$ and $Y_{i'}$ have self-loop edges; 
\item[(II-ii)] $Y_i$ has a self-loop edge and $Y_{i'}$ does not have a self-loop edge; 
\item[(II-iii)] Neither $Y_i$ nor $Y_{i'}$ has a self-loop edge.
\end{enumerate}

In the case (II-i), $\beta^{n-1}(X)$ consists of only one term, since $Y_j$ has only one backward edge for all $j \in \{ 2, \ldots, n\}$.
In $\beta^{n-1}(X)$, there is a unique simple edge path connecting the trivalent vertices of $Y_i$ and $Y_{i'}$.
Using the IHX and AS relations, one can move the $H$-colored hairs to this edge.
Thus $\beta^{n-1}(X)$ is a linear combination of hairy Lie graphs of the form 
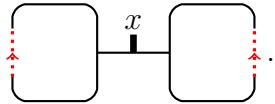
\[
\begin{tikzpicture}[baseline=-3pt, x=1.6mm, y=1.6mm]
\newdumbbell 
\draw[ut] (0,0) -- (0,1.5) node[above=-3pt]{$x$};
\end{tikzpicture} \ .
\]
The two dotted edges in the figure correspond to the self-loop edges of $Y_i$ and $Y_{i'}$.
By Lemma~\ref{lem:elementsR2}(2), we have $\beta^{n-1}(X) \in R$. 

In the case (II-ii), $Y_{i'}$ has two backward edges.
There are two terms in $\beta^{n-1}(X)$, each of which corresponds to the choice of a backward edge of $Y_{i'}$.
By using the IHX and AS relations in the same way as in the case (II-i), we see that $\beta^{n-1}(X)$ is a linear combination of elements of the form 
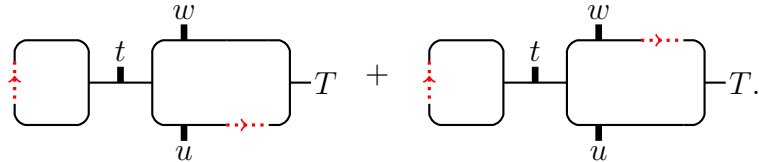
\[
\begin{tikzpicture}[baseline=0pt, x=1.4mm, y=1.4mm]
%%% middle horizontal line %%%
\draw (-3,0) -- (3,0);
%%% left dotted edge %%%
\draw[red, very thick, dotted, -<-] (-10,2) -- (-10,-2); 
%%% right  loop 1 %%%
\draw (10,4) to (4,4) to[bend right=30] (3,3) to (3,-3) to[bend right=30] (4,-4) to (10,-4);
%%% right loop 2 %%%
\draw (14,-4) -- (15,-4) to[bend right=30] (16,-3) -- (16,3) to[bend right=30] (15,4) -- (14,4);
%%% right horizontal line %%%
\draw (16,0) -- (18,0) node[right=-3pt]{$T$};
%%% left loop %%%
\draw (-10,2) to (-10,3) to [bend left=30] (-9,4) to (-4,4) to[bend left=30] (-3,3) to (-3,-3) to[bend left=30] (-4,-4) to (-9,-4) to[bend left=30] (-10,-3) to (-10,-2);
%%% piece of Y_{i'} %%%
\draw (10,4) -- (14,4);
\draw[red, very thick, dotted, ->-] (10,-4) -- (14,-4);
%%% labels %%%
\draw[ut] (0,0) -- (0,1.5) node[above=-3pt]{$t$};
\draw[ut] (6,4) -- (6,5.5) node[above=-3pt]{$w$};
\draw[ut] (6,-4) -- (6,-5.5) node[below=-3pt]{$u$};
\end{tikzpicture}
\ + \hspace{1em} 
\begin{tikzpicture}[baseline=0pt, x=1.4mm, y=1.4mm]
%%% middle horizontal line %%%
\draw (-3,0) -- (3,0);
%%% left dotted edge %%%
\draw[red, very thick, dotted, -<-] (-10,2) -- (-10,-2); 
%%% right  loop 1 %%%
\draw (10,4) to (4,4) to[bend right=30] (3,3) to (3,-3) to[bend right=30] (4,-4) to (10,-4);
%%% right loop 2 %%%
\draw (14,-4) -- (15,-4) to[bend right=30] (16,-3) -- (16,3) to[bend right=30] (15,4) -- (14,4);
%%% right horizontal line %%%
\draw (16,0) -- (18,0) node[right=-3pt]{$T$.};
%%% left loop %%%
\draw (-10,2) to (-10,3) to [bend left=30] (-9,4) to (-4,4) to[bend left=30] (-3,3) to (-3,-3) to[bend left=30] (-4,-4) to (-9,-4) to[bend left=30] (-10,-3) to (-10,-2);
%%% piece of Y_{i'} %%%
\draw (10,-4) -- (14,-4);
\draw[red, very thick, dotted, ->-] (10,4) -- (14,4);
%%% labels %%%
\draw[ut] (0,0) -- (0,1.5) node[above=-3pt]{$t$};
\draw[ut] (6,4) -- (6,5.5) node[above=-3pt]{$w$};
\draw[ut] (6,-4) -- (6,-5.5) node[below=-3pt]{$u$};
\end{tikzpicture}
\]
Here, $T$ is a rooted $H$-colored tree-shaped Jacobi diagram and $t,u,w \in \mathcal{T}$.
In both terms, the left dotted edge corresponds to the self-loop edge of $Y_i$ and the rightmost trivalent vertex to the trivalent vertex of $Y_{i'}$. 
Applying Lemma~\ref{lem:elementsR2}(1), we conclude that $\beta^{n-1}(X) \in R$. 

In the case (II-iii), 
both $Y_i$ and $Y_{i'}$ have two backward edges. 
We may assume $i<i'$ without loss of generality.
There is only one term in $\beta^{i-2}(X)$.
Then $\beta^{i-1}(X)$ consists of two terms, each obtained from $\beta^{i-2}(X)$ by changing one backward edge of $Y_i$ to a solid edge.
Since $N(Y_j) = 0$ for $i < j < i'$, $\beta^{i'-2}(X)$ also consists of two terms.
Each term of $\beta^{i'-2}(X)$ has $n - i' + 2$ trees, and the union of the first tree and the backward edge of $Y_i$ which remains a dotted edge is a $1$-looped graph.
The difference between the two terms is the place of the dotted edge.
Hence, by using the IHX and AS relations, we may assume that $\beta^{i'-2}(X)$ is a linear combination of elements of the following form: 
\[
\begin{tikzpicture}[baseline=-3pt, x=1.1mm, y=1.1mm]
%%% middle horizontal line %%%
\draw (3,0) to (21,0) to [bend left=30] (22,-1) to (22,-5) to [bend left=30] (21,-6) to (3, -6); 
\draw (-3,0) to (-21,0) to [bend right=30] (-22,-1) to (-22,-5) to [bend right=30] (-21,-6) to (-3,-6); 
\draw [red, very thick, dotted, ->-](-3,0) to (3,0);
\draw (-3,-6) to (3,-6);
%%% left dotted edge %%%
%% hairs %%
\draw (-18,0) -- (-18,3) node[above=-3pt] {$V_1$};
\draw (-11,2) node {$\cdots$};
\draw (-6,0) -- (-6,3) node[above=-3pt] {$V_j$};
\draw (6,0) -- (6,3) node[above=-3pt] {$V_{j+1}$};
\draw (12,2) node {$\cdots$};
\draw (18,0) -- (18,3) node[above=-3pt] {$V_{i'-1}$};
\end{tikzpicture}
\ - \ \, 
\begin{tikzpicture}[baseline=-3pt, x=1.1mm, y=1.1mm]
%%% middle horizontal line %%%
\draw (3,0) to (21,0) to [bend left=30] (22,-1) to (22,-5) to [bend left=30] (21,-6) to (3, -6); 
\draw (-3,0) to (-21,0) to [bend right=30] (-22,-1) to (-22,-5) to [bend right=30] (-21,-6) to (-3,-6); 
\draw (-3,0) to (3,0);
\draw [red, very thick, dotted, -<-](-3,-6) to (3,-6);
%%% left dotted edge %%%
%% hairs %%
\draw (-18,0) -- (-18,3) node[above=-3pt] {$V_1$};
\draw (-11,2) node {$\cdots$};
\draw (-6,0) -- (-6,3) node[above=-3pt] {$V_j$};
\draw (6,0) -- (6,3) node[above=-3pt] {$V_{j+1}$};
\draw (12,2) node {$\cdots$};
\draw (18,0) -- (18,3) node[above=-3pt] {$V_{i'-1}$};
\end{tikzpicture}.
\]
Here, we draw the first tree only, and $V_1,\ldots,V_{i'-1}$ are its univalent vertices.
The vertices $V_1, \ldots, V_{i'-1}$ are either colored with elements of $H$ or attached to the other tripods $Y_{i'}, \ldots, Y_n$ by dotted edges. 
In particular, the two backward edges of $Y_{i'}$ are connected to one of the vertices $V_1,\ldots,V_{i'-1}$. 
There are two cases.
First, if two backward edges of $Y_{i'}$ are both connected to $V_1,\ldots, V_j$, then $\beta^{n-1}(X)$ is a linear combination of elements of the following form:
\begin{align*}
&\begin{tikzpicture}[baseline=-3pt, x=1.1mm, y=1mm]
%%% middle horizontal line %%%
\draw (3,0) to (21,0) to [bend left=30] (22,-1) to (22,-5) to [bend left=30] (21,-6) to (3, -6); 
\draw (-3,0) to (-27,0) to [bend right=30] (-28,-1) to (-28,-5) to [bend right=30] (-27,-6) to (-3,-6); 
\draw [red, very thick, dotted, ->-](-3,0) to (3,0);
\draw (-3,-6) to (3,-6);
%%% left dotted edge %%%
%% hairs %%
\draw (-25,0) -- (-25,3) node[above=-3pt]{$T_1$};
\draw (-21,2) node {$\cdots$};
\draw (-18,6) to [bend left=30] (-17,7) to (-13,7) to [bend left=30] (-12,6);
\draw (-18,0) -- (-18,2);
\draw (-12,0) -- (-12,2);
\draw[red, very thick, dotted, ->-] (-18,2) -- (-18,6); 
\draw (-12,2) -- (-12,6); 
\draw (-15,7) -- (-15,10) node[above=-3pt]{$T_{i'}$};
\draw (-15,2) node {$\cdots$};
\draw (-9,2) node {$\cdots$};
\draw (-6,0) -- (-6,3) node[above=-3pt]{$T_j$};
\draw (6,0) -- (6,3) node[above=-3pt]{$T_{j+1}$};
\draw (12,2) node {$\cdots$};
\draw (18,0) -- (18,3) node[above=-3pt]{$T_{i'-1}$};
\end{tikzpicture}
- \ 
\begin{tikzpicture}[baseline=-3pt, x=1.1mm, y=1mm]
%%% middle horizontal line %%%
\draw (3,0) to (21,0) to [bend left=30] (22,-1) to (22,-5) to [bend left=30] (21,-6) to (3, -6); 
\draw (-3,0) to (-27,0) to [bend right=30] (-28,-1) to (-28,-5) to [bend right=30] (-27,-6) to (-3,-6); 
\draw [red, very thick, dotted, ->-](-3,0) to (3,0);
\draw (-3,-6) to (3,-6);
%%% left dotted edge %%%
%% hairs %%
\draw (-25,0) -- (-25,3) node[above=-3pt]{$T_1$};
\draw (-21,2) node {$\cdots$};
\draw (-18,6) to [bend left=30] (-17,7) to (-13,7) to [bend left=30] (-12,6);
\draw (-18,0) -- (-18,2);
\draw (-12,0) -- (-12,2);
\draw (-18,2) -- (-18,6); 
\draw[red, very thick, dotted, -<-] (-12,2) -- (-12,6); 
\draw (-15,7) -- (-15,10) node[above=-3pt]{$T_{i'}$};
\draw (-15,2) node {$\cdots$};
\draw (-9,2) node {$\cdots$};
\draw (-6,0) -- (-6,3) node[above=-3pt]{$T_j$};
\draw (6,0) -- (6,3) node[above=-3pt]{$T_{j+1}$};
\draw (12,2) node {$\cdots$};
\draw (18,0) -- (18,3) node[above=-3pt]{$T_{i'-1}$};
\end{tikzpicture}\\
&- \ 
\begin{tikzpicture}[baseline=-3pt, x=1.1mm, y=1mm]
%%% middle horizontal line %%%
\draw (3,0) to (21,0) to [bend left=30] (22,-1) to (22,-5) to [bend left=30] (21,-6) to (3, -6); 
\draw (-3,0) to (-27,0) to [bend right=30] (-28,-1) to (-28,-5) to [bend right=30] (-27,-6) to (-3,-6); 
\draw (-3,0) to (3,0);
\draw [red, very thick, dotted, -<-](-3,-6) to (3,-6);
%%% left dotted edge %%%
%% hairs %%
\draw (-25,0) -- (-25,3) node[above=-3pt]{$T_1$};
\draw (-21,2) node {$\cdots$};
\draw (-18,6) to [bend left=30] (-17,7) to (-13,7) to [bend left=30] (-12,6);
\draw (-18,0) -- (-18,2);
\draw (-12,0) -- (-12,2);
\draw[red, very thick, dotted, ->-] (-18,2) -- (-18,6); 
\draw (-12,2) -- (-12,6); 
\draw (-15,7) -- (-15,10) node[above=-3pt]{$T_{i'}$};
\draw (-15,2) node {$\cdots$};
\draw (-9,2) node {$\cdots$};
\draw (-6,0) -- (-6,3) node[above=-3pt]{$T_j$};
\draw (6,0) -- (6,3) node[above=-3pt]{$T_{j+1}$};
\draw (12,2) node {$\cdots$};
\draw (18,0) -- (18,3) node[above=-3pt]{$T_{i'-1}$};
\end{tikzpicture}
+ \ 
\begin{tikzpicture}[baseline=-3pt, x=1.1mm, y=1mm]
%%% middle horizontal line %%%
\draw (3,0) to (21,0) to [bend left=30] (22,-1) to (22,-5) to [bend left=30] (21,-6) to (3, -6); 
\draw (-3,0) to (-27,0) to [bend right=30] (-28,-1) to (-28,-5) to [bend right=30] (-27,-6) to (-3,-6); 
\draw (-3,0) to (3,0);
\draw [red, very thick, dotted, -<-](-3,-6) to (3,-6);
%%% left dotted edge %%%
%% hairs %%
\draw (-25,0) -- (-25,3) node[above=-3pt]{$T_1$};
\draw (-21,2) node {$\cdots$};
\draw (-18,6) to [bend left=30] (-17,7) to (-13,7) to [bend left=30] (-12,6);
\draw (-18,0) -- (-18,2);
\draw (-12,0) -- (-12,2);
\draw (-18,2) -- (-18,6); 
\draw[red, very thick, dotted, -<-] (-12,2) -- (-12,6); 
\draw (-15,7) -- (-15,10) node[above=-3pt]{$T_{i'}$};
\draw (-15,2) node {$\cdots$};
\draw (-9,2) node {$\cdots$};
\draw (-6,0) -- (-6,3) node[above=-3pt]{$T_j$};
\draw (6,0) -- (6,3) node[above=-3pt]{$T_{j+1}$};
\draw (12,2) node {$\cdots$};
\draw (18,0) -- (18,3) node[above=-3pt]{$T_{i'-1}$};
\end{tikzpicture}.
\end{align*}
Here, the tripod $Y_{i'}$ is depicted just below $T_{i'}$.
By Lemma~\ref{lem:elementsR}(3),
this is contained in $R$.
Next, if two backward edges of $Y_{i'}$ are connected to one of $V_1$, $\ldots$, $V_j$ and one of $V_{j+1}$, $\ldots$, $V_{i'-1}$, then $\beta^{n-1}(X)$ is a linear combination of
\begin{align*}
&\begin{tikzpicture}[baseline=-3pt, x=1.2mm, y=1.0mm]
%%% middle horizontal line %%%
\draw (3,0) to (21,0) to [bend left=30] (22,-1) to (22,-5) to [bend left=30] (21,-6) to (3, -6); 
\draw (-3,0) to (-21,0) to [bend right=30] (-22,-1) to (-22,-5) to [bend right=30] (-21,-6) to (-3,-6); 
\draw [red, very thick, dotted, ->-] (-3,0) to (3,0);
\draw (-3,-6) to (3,-6);
%%% left dotted edge %%%
%% hairs %%
\draw (-18,0) -- (-18,3) node[above=-3pt]{$T_1$};
\draw (-12,0) to (-12,10) to [bend left=30] (-11,11); 
\draw (-15,2) node {$\cdots$};
\draw (-9,2) node {$\cdots$};
\draw (-6,0) -- (-6,3) node[above=-3pt]{$T_j$};
\draw (6,0) -- (6,3) node[above=-3pt]{$T_{j+1}$};
\draw (9,2) node {$\cdots$};
\draw (12,0) to (12,10) to [bend right=30] (11,11); 
\draw (15,2) node {$\cdots$};
\draw [red, very thick, dotted, ->-](-11,11) to (-3,11);
\draw (-3,11) to (3,11);
\draw (3,11) to (11,11);
\draw (0,11) -- (0,14) node [above=-3pt]{$T_{i'}$};
\draw (18,0) -- (18,3) node[above=-3pt]{$T_{i'-1}$};
\end{tikzpicture}
- \ 
\begin{tikzpicture}[baseline=-3pt, x=1.2mm, y=1.0mm]
%%% middle horizontal line %%%
\draw (3,0) to (21,0) to [bend left=30] (22,-1) to (22,-5) to [bend left=30] (21,-6) to (3, -6); 
\draw (-3,0) to (-21,0) to [bend right=30] (-22,-1) to (-22,-5) to [bend right=30] (-21,-6) to (-3,-6); 
\draw [red, very thick, dotted, ->-] (-3,0) to (3,0);
\draw (-3,-6) to (3,-6);
%%% left dotted edge %%%
%% hairs %%
\draw (-18,0) -- (-18,3) node[above=-3pt]{$T_1$};
\draw (-12,0) to (-12,10) to [bend left=30] (-11,11); 
\draw (-15,2) node {$\cdots$};
\draw (-9,2) node {$\cdots$};
\draw (-6,0) -- (-6,3) node[above=-3pt]{$T_j$};
\draw (6,0) -- (6,3) node[above=-3pt]{$T_{j+1}$};
\draw (9,2) node {$\cdots$};
\draw (12,0) to (12,10) to [bend right=30] (11,11); 
\draw (15,2) node {$\cdots$};
\draw (-11,11) to (-3,11);
\draw (-3,11) to (3,11);
\draw [red, very thick, dotted, ->-](3,11) to (11,11);
\draw (0,11) -- (0,14) node [above=-3pt]{$T_{i'}$};
\draw (18,0) -- (18,3) node[above=-3pt]{$T_{i'-1}$};
\end{tikzpicture}\\
&- \ 
\begin{tikzpicture}[baseline=-3pt, x=1.2mm, y=1.0mm]
%%% middle horizontal line %%%
\draw (3,0) to (21,0) to [bend left=30] (22,-1) to (22,-5) to [bend left=30] (21,-6) to (3, -6); 
\draw (-3,0) to (-21,0) to [bend right=30] (-22,-1) to (-22,-5) to [bend right=30] (-21,-6) to (-3,-6); 
\draw  (-3,0) to (3,0);
\draw [red, very thick, dotted, -<-] (-3,-6) to (3,-6);
%%% left dotted edge %%%
%% hairs %%
\draw (-18,0) -- (-18,3) node[above=-3pt]{$T_1$};
\draw (-12,0) to (-12,10) to [bend left=30] (-11,11); 
\draw (-15,2) node {$\cdots$};
\draw (-9,2) node {$\cdots$};
\draw (-6,0) -- (-6,3) node[above=-3pt]{$T_j$};
\draw (6,0) -- (6,3) node[above=-3pt]{$T_{j+1}$};
\draw (9,2) node {$\cdots$};
\draw (12,0) to (12,10) to [bend right=30] (11,11); 
\draw (15,2) node {$\cdots$};
\draw [red, very thick, dotted, ->-](-11,11) to (-3,11);
\draw (-3,11) to (3,11);
\draw (3,11) to (11,11);
\draw (0,11) -- (0,14) node[above=-3pt]{$T_{i'}$};
\draw (18,0) -- (18,3) node[above=-3pt]{$T_{i'-1}$};
\end{tikzpicture}
+ \ 
\begin{tikzpicture}[baseline=-3pt, x=1.2mm, y=1.0mm]
%%% middle horizontal line %%%
\draw (3,0) to (21,0) to [bend left=30] (22,-1) to (22,-5) to [bend left=30] (21,-6) to (3, -6); 
\draw (-3,0) to (-21,0) to [bend right=30] (-22,-1) to (-22,-5) to [bend right=30] (-21,-6) to (-3,-6); 
\draw  (-3,0) to (3,0);
\draw[red, very thick, dotted, -<-](-3,-6) to (3,-6);
%%% left dotted edge %%%
%% hairs %%
\draw (-18,0) -- (-18,3) node[above=-3pt]{$T_1$};
\draw (-12,0) to (-12,10) to [bend left=30] (-11,11); 
\draw (-15,2) node {$\cdots$};
\draw (-9,2) node {$\cdots$};
\draw (-6,0) -- (-6,3) node[above=-3pt]{$T_j$};
\draw (6,0) -- (6,3) node[above=-3pt]{$T_{j+1}$};
\draw (9,2) node {$\cdots$};
\draw (12,0) to (12,10) to [bend right=30] (11,11); 
\draw (15,2) node {$\cdots$};
\draw (-11,11) to (-3,11);
\draw (-3,11) to (3,11);
\draw [red, very thick, dotted, ->-](3,11) to (11,11);
\draw (0,11) -- (0,14) node [above=-3pt]{$T_{i'}$};
\draw (18,0) -- (18,3) node[above=-3pt]{$T_{i'-1}$};
\end{tikzpicture}.
\end{align*}
By Lemma~\ref{lem:elementsR}(1),
this is contained in $R$.
This settles the case (II-iii) and completes the proof of $\beta^{n-1}(C_{n,n+1} \H^{\ord}(n)) \subset R$.

{\em Step 2}. 
Conversely, we show that the space $R$ is contained in the image of the map \eqref{eq:cokbeta}.
It is sufficient to show that, for any pure tensors $t,u,v,w \in \mathcal{T}$, the element $\Theta_R(t,u,v,w)$ in \eqref{eq:GC3_special} with $n$ trivalent vertices lies in $\beta^{n-1}(C_{n,n+1}\H^{\ord}(n))$.
We prove this by induction on $\deg v + \deg w$.

{\em Step 2-$(i)$}.
First consider the case $\deg v + \deg w = 0$.
The following example proves the assertion for $(\deg t, \deg u) = (3,2)$, and it is easy to generalize it for any $t$ and $u$. 
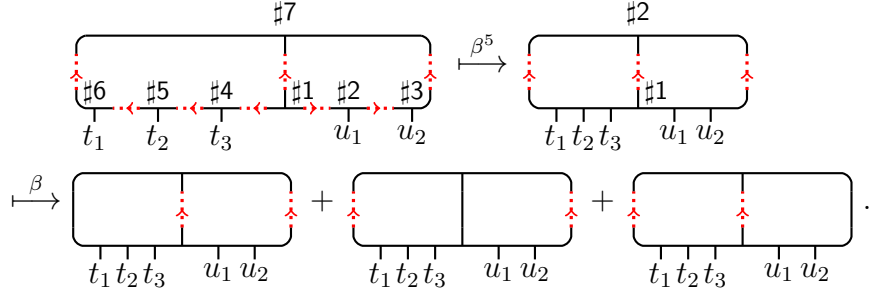
\begin{align*}
& \begin{tikzpicture}[baseline=0pt, x=1.2mm, y=1.2mm]
%%% dotted edges %%%
\draw[red, very thick, dotted, ->-] (0,-2) -- (0,2);
\draw[red, very thick, dotted, ->-] (-23,-2) -- (-23,2);
\draw[red, very thick, dotted, ->-] (16,-2) -- (16,2);
\draw[red, very thick, dotted, ->-] (2,-4) -- (5,-4);
\draw[red, very thick, dotted, ->-] (9,-4) -- (12,-4);
\draw[red, very thick, dotted, ->-] (-2,-4) -- (-5,-4);
\draw[red, very thick, dotted, ->-] (-9,-4) -- (-12,-4);
\draw[red, very thick, dotted, ->-] (-16,-4) -- (-19,-4);
%%% #1 %%%
\draw (0,-2) -- (0,-4);
\draw (-2,-4) -- (2,-4); 
\draw (2,-2.25) node{\footnotesize ${\sf \sharp 1}$};
%%% #2 %%%
\draw (5,-4) -- (9,-4);
\draw (7,-4) -- (7,-5.5) node[below=-3pt]{$u_1$};
\draw (7,-2.25) node{\footnotesize ${\sf \sharp 2}$};
%%% #3 %%%
\draw (12,-4) -- (15,-4) to[bend right=30] (16,-3) -- (16,-2);
\draw (14,-4) -- (14,-5.5) node[below=-3pt]{$u_2$};
\draw (14,-2.25) node{\footnotesize ${\sf \sharp 3}$};
%%% #4 %%%
\draw (-5,-4) -- (-9,-4);
\draw (-7,-4) -- (-7,-5.5) node[below=-3pt]{$t_3$};
\draw (-7,-2.25) node{\footnotesize ${\sf \sharp 4}$};
%%% #5 %%%
\draw (-12,-4) -- (-16,-4);
\draw (-14,-4) -- (-14,-5.5) node[below=-3pt]{$t_2$};
\draw (-14,-2.25) node{\footnotesize ${\sf \sharp 5}$};
%%% #6 %%%
\draw (-19,-4) -- (-22,-4) to[bend left=30] (-23,-3) -- (-23,-2);
\draw (-21,-4) -- (-21,-5.5) node[below=-3pt]{$t_1$};
\draw (-21,-2.25) node{\footnotesize ${\sf \sharp 6}$};
%%% #7 %%%
\draw (0,2) -- (0,4);
\draw (15,4) -- (-22,4);
\draw (15,4) to[bend left=30] (16,3) -- (16,2);
\draw (-22,4) to[bend right=30] (-23,3) -- (-23,2);
\draw (0,6.5) node{\footnotesize ${\sf \sharp 7}$};
\end{tikzpicture}
\ \overset{\beta^5}{\longmapsto} \ 
\begin{tikzpicture}[baseline=0pt, x=1.2mm, y=1.2mm]
\draw (2,-2.25) node{\footnotesize ${\sf \sharp 1}$};
\draw (0,6.5) node{\footnotesize ${\sf \sharp 2}$};
%%% dotted edges %%%
\draw[red, very thick, dotted, ->-] (0,-2) -- (0,2);
\draw[red, very thick, dotted, ->-] (12,-2) -- (12,2);
\draw[red, very thick, dotted, ->-] (-12,-2) -- (-12,2);
%%% middle vertical edges %%%
\draw (0,2) -- (0,4);
\draw (0,-2) -- (0,-4);
%%% lower horizontal edge %%%
\draw (-12,-2) -- (-12,-3) to[bend right=30] (-11,-4) -- (11,-4) to[bend right=30] (12,-3) -- (12,-2);
%%% upper horizontal edge %%%
\draw (-12,2) -- (-12,3) to[bend left=30] (-11,4) -- (11,4) to[bend left=30] (12,3) -- (12,2);
%%% hairs %%%
\draw (-9,-4) -- (-9,-5.5) node[below=-3pt]{$t_1$};
\draw (-6,-4) -- (-6,-5.5) node[below=-3pt]{$t_2$};
\draw (-3,-4) -- (-3,-5.5) node[below=-3pt]{$t_3$};
\draw (4,-4) -- (4,-5.5) node[below=-3pt]{$u_1$};
\draw (8,-4) -- (8,-5.5) node[below=-3pt]{$u_2$};
\end{tikzpicture} \\ 
\overset{\beta}{\longmapsto} \ &
\begin{tikzpicture}[baseline=0pt, x=1.2mm, y=1.2mm]
%%% dotted edges %%%
\draw[red, very thick, dotted, ->-] (0,-2) -- (0,2);
\draw[red, very thick, dotted, ->-] (12,-2) -- (12,2);
\draw (-12,-2) -- (-12,2);
%%% middle vertical edges %%%
\draw (0,2) -- (0,4);
\draw (0,-2) -- (0,-4);
%%% lower horizontal edge %%%
\draw (-12,-2) -- (-12,-3) to[bend right=30] (-11,-4) -- (11,-4) to[bend right=30] (12,-3) -- (12,-2);
%%% upper horizontal edge %%%
\draw (-12,2) -- (-12,3) to[bend left=30] (-11,4) -- (11,4) to[bend left=30] (12,3) -- (12,2);
%%% hairs %%%
\draw (-9,-4) -- (-9,-5.5) node[below=-3pt]{$t_1$};
\draw (-6,-4) -- (-6,-5.5) node[below=-3pt]{$t_2$};
\draw (-3,-4) -- (-3,-5.5) node[below=-3pt]{$t_3$};
\draw (4,-4) -- (4,-5.5) node[below=-3pt]{$u_1$};
\draw (8,-4) -- (8,-5.5) node[below=-3pt]{$u_2$};
\end{tikzpicture}
\ + \ 
\begin{tikzpicture}[baseline=0pt, x=1.2mm, y=1.2mm]
%%% dotted edges %%%
\draw (0,-2) -- (0,2);
\draw[red, very thick, dotted, ->-] (12,-2) -- (12,2);
\draw[red, very thick, dotted, ->-] (-12,-2) -- (-12,2);
%%% middle vertical edges %%%
\draw (0,2) -- (0,4);
\draw (0,-2) -- (0,-4);
%%% lower horizontal edge %%%
\draw (-12,-2) -- (-12,-3) to[bend right=30] (-11,-4) -- (11,-4) to[bend right=30] (12,-3) -- (12,-2);
%%% upper horizontal edge %%%
\draw (-12,2) -- (-12,3) to[bend left=30] (-11,4) -- (11,4) to[bend left=30] (12,3) -- (12,2);
%%% hairs %%%
\draw (-9,-4) -- (-9,-5.5) node[below=-3pt]{$t_1$};
\draw (-6,-4) -- (-6,-5.5) node[below=-3pt]{$t_2$};
\draw (-3,-4) -- (-3,-5.5) node[below=-3pt]{$t_3$};
\draw (4,-4) -- (4,-5.5) node[below=-3pt]{$u_1$};
\draw (8,-4) -- (8,-5.5) node[below=-3pt]{$u_2$};
\end{tikzpicture}
\ + \ 
\begin{tikzpicture}[baseline=0pt, x=1.2mm, y=1.2mm]
%%% dotted edges %%%
\draw[red, very thick, dotted, ->-] (0,-2) -- (0,2);
\draw (12,-2) -- (12,2);
\draw[red, very thick, dotted, ->-] (-12,-2) -- (-12,2);
%%% middle vertical edges %%%
\draw (0,2) -- (0,4);
\draw (0,-2) -- (0,-4);
%%% lower horizontal edge %%%
\draw (-12,-2) -- (-12,-3) to[bend right=30] (-11,-4) -- (11,-4) to[bend right=30] (12,-3) -- (12,-2);
%%% upper horizontal edge %%%
\draw (-12,2) -- (-12,3) to[bend left=30] (-11,4) -- (11,4) to[bend left=30] (12,3) -- (12,2);
%%% hairs %%%
\draw (-9,-4) -- (-9,-5.5) node[below=-3pt]{$t_1$};
\draw (-6,-4) -- (-6,-5.5) node[below=-3pt]{$t_2$};
\draw (-3,-4) -- (-3,-5.5) node[below=-3pt]{$t_3$};
\draw (4,-4) -- (4,-5.5) node[below=-3pt]{$u_1$};
\draw (8,-4) -- (8,-5.5) node[below=-3pt]{$u_2$};
\end{tikzpicture} \ .
\end{align*}

{\em Step 2-$(ii)$}.
Next, we consider the case $\deg v + \deg w = 1$.
We may assume that $(\deg v, \deg w) = (1,0)$.
Let $U \in C_{1,2}\H$ be the following element:
\[
U = \  
\begin{tikzpicture}[baseline=-3pt, x=1mm, y=1.2mm]
\newsmalltheta{2}
%%% hairs %%%
%% lower left %%
\draw[ut] (-3,-4) -- (-3,-5.5) node[below=-3pt]{$t$};
\draw (-6,-4) -- (-6,-5.5) node[below=-3pt]{$v$};
%% lower right %%
\draw[ut] (4.5,-4) -- (4.5,-5.5) node[below=-3pt]{$u$};
\end{tikzpicture}
\ - \ 
\begin{tikzpicture}[baseline=-3pt, x=1mm, y=1.2mm]
\newsmalltheta{2}
%%% hairs %%%
%% upper left %%
\draw (-6,4) -- (-6,5.5) node[above=-3pt]{$v$};
%% lower left %%
\draw[ut] (-3,-4) -- (-3,-5.5) node[below=-3pt]{$t$};
%% lower right %%
\draw[ut] (4.5,-4) -- (4.5,-5.5) node[below=-3pt]{$u$};
\end{tikzpicture}
\ + \ 
\begin{tikzpicture}[baseline=-3pt, x=1mm, y=1.2mm]
\newsmalltheta{3}
%%% hairs %%%
%% lower left %%
\draw[ut] (-3,-4) -- (-3,-5.5) node[below=-3pt]{$t$};
\draw (-6,-4) -- (-6,-5.5) node[below=-3pt]{$v$};
%% lower right %%
\draw[ut] (4.5,-4) -- (4.5,-5.5) node[below=-3pt]{$u$};
\end{tikzpicture}
\ - \ 
\begin{tikzpicture}[baseline=-3pt, x=1mm, y=1.2mm]
\newsmalltheta{3}
%%% hairs %%%
%% upper left %%
\draw (-6,4) -- (-6,5.5) node[above=-3pt]{$v$};
%% lower left %%
\draw[ut] (-3,-4) -- (-3,-5.5) node[below=-3pt]{$t$};
%% lower right %%
\draw[ut] (4.5,-4) -- (4.5,-5.5) node[below=-3pt]{$u$};
\end{tikzpicture} \ .
\]
The following sample computation, in which we assume $(\deg t, \deg u) = (2,2)$, shows that $U$ lies in the image of $\beta^{n-1}$, where $n=3 + \deg t + \deg u$: 
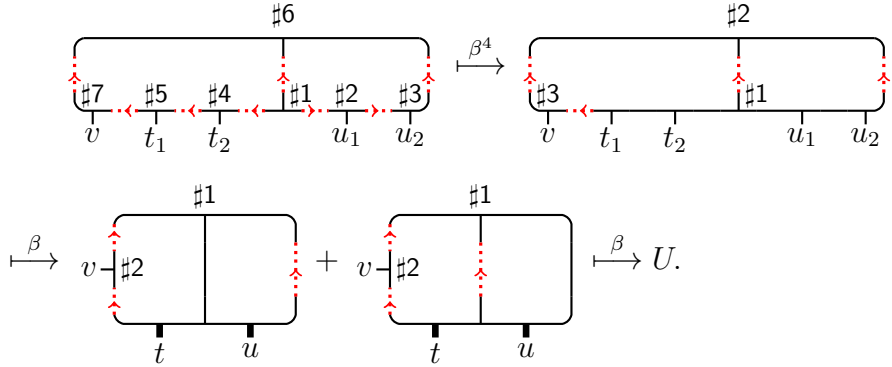
\begin{align*}
& \begin{tikzpicture}[baseline=0pt, x=1.2mm, y=1.2mm]
%%% dotted edges %%%
\draw[red, very thick, dotted, ->-] (0,-2) -- (0,2);
\draw[red, very thick, dotted, ->-] (-23,-2) -- (-23,2);
\draw[red, very thick, dotted, ->-] (16,-2) -- (16,2);
\draw[red, very thick, dotted, ->-] (2,-4) -- (5,-4);
\draw[red, very thick, dotted, ->-] (9,-4) -- (12,-4);
\draw[red, very thick, dotted, ->-] (-2,-4) -- (-5,-4);
\draw[red, very thick, dotted, ->-] (-9,-4) -- (-12,-4);
\draw[red, very thick, dotted, ->-] (-16,-4) -- (-19,-4);
%%% #1 %%%
\draw (0,-2) -- (0,-4);
\draw (-2,-4) -- (2,-4); 
\draw (2,-2.25) node{\footnotesize ${\sf \sharp 1}$};
%%% #2 %%%
\draw (5,-4) -- (9,-4);
\draw (7,-4) -- (7,-5.5) node[below=-3pt]{$u_1$};
\draw (7,-2.25) node{\footnotesize ${\sf \sharp 2}$};
%%% #3 %%%
\draw (12,-4) -- (15,-4) to[bend right=30] (16,-3) -- (16,-2);
\draw (14,-4) -- (14,-5.5) node[below=-3pt]{$u_2$};
\draw (14,-2.25) node{\footnotesize ${\sf \sharp 3}$};
%%% #4 %%%
\draw (-5,-4) -- (-9,-4);
\draw (-7,-4) -- (-7,-5.5) node[below=-3pt]{$t_2$};
\draw (-7,-2.25) node{\footnotesize ${\sf \sharp 4}$};
%%% #5 %%%
\draw (-12,-4) -- (-16,-4);
\draw (-14,-4) -- (-14,-5.5) node[below=-3pt]{$t_1$};
\draw (-14,-2.25) node{\footnotesize ${\sf \sharp 5}$};
%%% #6 %%%
\draw (0,2) -- (0,4);
\draw (15,4) -- (-22,4);
\draw (15,4) to[bend left=30] (16,3) -- (16,2);
\draw (-22,4) to[bend right=30] (-23,3) -- (-23,2);
\draw (0,6.5) node{\footnotesize ${\sf \sharp 6}$};]
%%% #7 %%%
\draw (-19,-4) -- (-22,-4) to[bend left=30] (-23,-3) -- (-23,-2);
\draw (-21,-4) -- (-21,-5.5) node[below=-3pt]{$v$};
\draw (-21,-2.25) node{\footnotesize ${\sf \sharp 7}$};
\end{tikzpicture}
\ \overset{\beta^4}{\longmapsto} \ 
\begin{tikzpicture}[baseline=0pt, x=1.2mm, y=1.2mm]
%%% dotted edges %%%
\draw[red, very thick, dotted, ->-] (0,-2) -- (0,2);
\draw[red, very thick, dotted, ->-] (-23,-2) -- (-23,2);
\draw[red, very thick, dotted, ->-] (16,-2) -- (16,2);
\draw (2,-4) -- (5,-4);
\draw (9,-4) -- (12,-4);
\draw (-2,-4) -- (-5,-4);
\draw (-9,-4) -- (-12,-4);
\draw[red, very thick, dotted, ->-] (-16,-4) -- (-19,-4);
%%% #1 %%%
\draw (0,-2) -- (0,-4);
\draw (-2,-4) -- (2,-4); 
\draw (2,-2.25) node{\footnotesize ${\sf \sharp 1}$};
%%% #2 %%%
\draw (5,-4) -- (9,-4);
\draw (7,-4) -- (7,-5.5) node[below=-3pt]{$u_1$};
%%% #3 %%%
\draw (12,-4) -- (15,-4) to[bend right=30] (16,-3) -- (16,-2);
\draw (14,-4) -- (14,-5.5) node[below=-3pt]{$u_2$};
%%% #4 %%%
\draw (-5,-4) -- (-9,-4);
\draw (-7,-4) -- (-7,-5.5) node[below=-3pt]{$t_2$};
%%% #5 %%%
\draw (-12,-4) -- (-16,-4);
\draw (-14,-4) -- (-14,-5.5) node[below=-3pt]{$t_1$};
%%% #6 %%%
\draw (0,2) -- (0,4);
\draw (15,4) -- (-22,4);
\draw (15,4) to[bend left=30] (16,3) -- (16,2);
\draw (-22,4) to[bend right=30] (-23,3) -- (-23,2);
\draw (0,6.5) node{\footnotesize ${\sf \sharp 2}$};]
%%% #7 %%%
\draw (-19,-4) -- (-22,-4) to[bend left=30] (-23,-3) -- (-23,-2);
\draw (-21,-4) -- (-21,-5.5) node[below=-3pt]{$v$};
\draw (-21,-2.25) node{\footnotesize ${\sf \sharp 3}$};
\end{tikzpicture} \\ 
\overset{\beta}{\longmapsto} \ &
\begin{tikzpicture}[baseline=0pt, x=1.2mm, y=1.2mm]
\draw (0,8) node{\footnotesize ${\sf \sharp 1}$};
\draw (-8,0) node{\footnotesize ${\sf \sharp 2}$};
%%% dotted edges %%%
\draw (0,-3) -- (0,3);
\draw[red, very thick, dotted, ->-] (10,-3) -- (10,3);
\draw[red, very thick, dotted, ->-] (-10,-5) -- (-10,-2);
\draw[red, very thick, dotted, -<-] (-10,5) -- (-10,2);
%%% middle vertical edges %%%
\draw (0,3) -- (0,6);
\draw (0,-3) -- (0,-6);
%%% lower horizontal edge %%%
\draw (-10,-5) to[bend right=30] (-9,-6) -- (9,-6) to[bend right=30] (10,-5) -- (10,-3);
%%% upper horizontal edge %%%
\draw (-10,5) to[bend left=30] (-9,6) -- (9,6) to[bend left=30] (10,5) -- (10,3);
%%% leftmost Y %%%
\draw (-10,-2) -- (-10,2);
\draw (-10,0) -- (-11.5,0) node[left=-3pt]{$v$};
%%% hairs %%%
\draw[ut] (-5,-6) -- (-5,-7.5) node[below=-3pt]{$t$};
\draw[ut] (5,-6) -- (5,-7.5) node[below=-3pt]{$u$};
\end{tikzpicture}
\ +  
\begin{tikzpicture}[baseline=0pt, x=1.2mm, y=1.2mm]
\draw (0,8) node{\footnotesize ${\sf \sharp 1}$};
\draw (-8,0) node{\footnotesize ${\sf \sharp 2}$};
%%% dotted edges %%%
\draw[red, very thick, dotted, ->-] (0,-3) -- (0,3);
\draw (10,-3) -- (10,3);
\draw[red, very thick, dotted, ->-] (-10,-5) -- (-10,-2);
\draw[red, very thick, dotted, -<-] (-10,5) -- (-10,2);
%%% middle vertical edges %%%
\draw (0,3) -- (0,6);
\draw (0,-3) -- (0,-6);
%%% lower horizontal edge %%%
\draw (-10,-5) to[bend right=30] (-9,-6) -- (9,-6) to[bend right=30] (10,-5) -- (10,-3);
%%% upper horizontal edge %%%
\draw (-10,5) to[bend left=30] (-9,6) -- (9,6) to[bend left=30] (10,5) -- (10,3);
%%% leftmost Y %%%
\draw (-10,-2) -- (-10,2);
\draw (-10,0) -- (-11.5,0) node[left=-3pt]{$v$};
%%% hairs %%%
\draw[ut] (-5,-6) -- (-5,-7.5) node[below=-3pt]{$t$};
\draw[ut] (5,-6) -- (5,-7.5) node[below=-3pt]{$u$};
\end{tikzpicture}
\ \overset{\beta}{\longmapsto} U. 
\end{align*}
In the same manner as in the proof of Lemma~\ref{lem:elementsR}(1), we have
\[
\Theta_R(t,u,v,1) = \Theta_R(vt,u,1,1) - U.
\]
The first term on the right hand side is  in the image of the map \eqref{eq:cokbeta} from the case $\deg v + \deg w = 0$.
This settles the case $\deg v + \deg w = 1$. 

{\em Step 2-$(iii)$}.
Finally, we consider the case $\deg v + \deg w \ge 2$.
We need a lemma.

\begin{lemma} \label{lem:induction1and2}
The element in Lemma~{\rm\ref{lem:elementsR}(3)} belongs to the image of the map~\eqref{eq:cokbeta} when $a,b \in H$.
\end{lemma}
\begin{proof}
We may assume that $x=1$. 
The following sample computation proves the assertion for the case where the degrees of $t,u,v,w$ are $2,3,3,2$, respectively. 
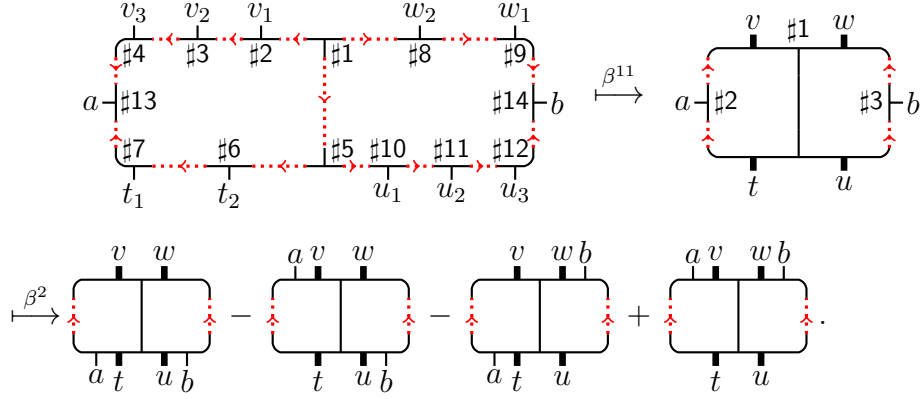
\begin{align*}
& \begin{tikzpicture}[baseline=0pt, x=1.2mm, y=1.2mm]
%%% dotted edges %%%%
\draw[red, very thick, dotted, ->-] (0,5) -- (0,-5);
\draw[red, very thick, dotted, ->-] (-23,5) -- (-23,2);
\draw[red, very thick, dotted, ->-] (-23,-5) -- (-23,-2);
\draw[red, very thick, dotted, ->-] (23,5) -- (23,2);
\draw[red, very thick, dotted, ->-] (23,-5) -- (23,-2);
\draw[red, very thick, dotted, ->-] (-2,7) -- (-5,7);
\draw[red, very thick, dotted, ->-] (-9,7) -- (-12,7);
\draw[red, very thick, dotted, ->-] (-16,7) -- (-19,7);
\draw[red, very thick, dotted, ->-] (-2,-7) -- (-8,-7);
\draw[red, very thick, dotted, ->-] (-13,-7) -- (-19,-7);
\draw[red, very thick, dotted, ->-] (2,-7) -- (5,-7);
\draw[red, very thick, dotted, ->-] (9,-7) -- (12,-7);
\draw[red, very thick, dotted, ->-] (16,-7) -- (19,-7);
\draw[red, very thick, dotted, ->-] (2,7) -- (8,7);
\draw[red, very thick, dotted, ->-] (13,7) -- (19,7);
%%% #1 %%% 
\draw (0,5) -- (0,7);
\draw (-2,7) -- (2,7);
\draw (2,5.25) node{\footnotesize ${\sf \sharp 1}$};
%%% #2 %%%
\draw (-5,7) -- (-9,7);
\draw (-7,7) -- (-7,8.5) node[above=-3pt]{$v_1$};
\draw (-7,5.25) node{\footnotesize ${\sf \sharp 2}$};
%%% #3 %%%
\draw (-12,7) -- (-16,7);
\draw (-14,7) -- (-14,8.5) node[above=-3pt]{$v_2$};
\draw (-14,5.25) node{\footnotesize ${\sf \sharp 3}$};
%%% #4 %%%
\draw (-19,7) -- (-22,7) to[bend right=30] (-23,6) -- (-23,5);
\draw (-21,7) -- (-21,8.5) node[above=-3pt]{$v_3$};
\draw (-21,5.25) node{\footnotesize ${\sf \sharp 4}$};
%%% #5 %%%
\draw (0,-5) -- (0,-7);
\draw (-2,-7) -- (2,-7);
\draw (2,-5.25) node{\footnotesize ${\sf \sharp 5}$};
%%% #6 %%%
\draw (-8,-7) -- (-13,-7);
\draw (-10.5,-7) -- (-10.5,-8.5) node[below=-3pt]{$t_2$};
\draw (-10.5,-5.25) node{\footnotesize ${\sf \sharp 6}$};
%%% #7 %%%
\draw (-19,-7) -- (-22,-7) to[bend left=30] (-23,-6) -- (-23,-5);
\draw (-21,-7) -- (-21,-8.5) node[below=-3pt]{$t_1$};
\draw (-21,-5.25) node{\footnotesize ${\sf \sharp 7}$};
%%% #8 %%%
\draw (8,7) -- (13,7);
\draw (10.5,7) -- (10.5,8.5) node[above=-3pt]{$w_2$};
\draw (10.5,5.25) node{\footnotesize ${\sf \sharp 8}$};
%%% #9 %%% 
\draw (19,7) -- (22,7) to[bend left=30] (23,6) -- (23,5);
\draw (21,7) -- (21,8.5) node[above=-3pt]{$w_1$};
\draw (21,5.25) node{\footnotesize ${\sf \sharp 9}$};
%%% #10 %%%
\draw (5,-7) -- (9,-7);
\draw (7,-7) -- (7,-8.5) node[below=-3pt]{$u_1$};
\draw (7,-5.25) node{\footnotesize ${\sf \sharp 10}$};
%%% #11 %%%
\draw (12,-7) -- (16,-7);
\draw (14,-7) -- (14,-8.5) node[below=-3pt]{$u_2$};
\draw (14,-5.25) node{\footnotesize ${\sf \sharp 11}$};
%%% #12 %%%
\draw (19,-7) -- (22,-7) to[bend right=30] (23,-6) -- (23,-5);
\draw (21,-7) -- (21,-8.5) node[below=-3pt]{$u_3$};
\draw (20.5,-5.25) node{\footnotesize ${\sf \sharp 12}$};
%%% #13 %%%
\draw (-23,-2) -- (-23,2);
\draw (-23,0) -- (-24.5,0) node[left=-3pt]{$a$};
\draw (-20.5,0) node{\footnotesize ${\sf \sharp 13}$};
%%% #14 %%%
\draw (23,-2) -- (23,2);
\draw (23,0) -- (24.5,0) node[right=-3pt]{$b$};
\draw (20.5,0) node{\footnotesize ${\sf \sharp 14}$};
\end{tikzpicture}
\ \overset{\beta^{11}}{\longmapsto} \ 
\begin{tikzpicture}[baseline=0pt, x=1.2mm, y=1.2mm]
\draw (0,8) node{\footnotesize ${\sf \sharp 1}$};
\draw (-8,0) node{\footnotesize ${\sf \sharp 2}$};
\draw (8,0) node{\footnotesize ${\sf \sharp 3}$};
%%% dotted edges %%%
\draw (0,-3) -- (0,3);
\draw[red, very thick, dotted, ->-] (-10,-5) -- (-10,-2);
\draw[red, very thick, dotted, -<-] (-10,5) -- (-10,2);
\draw[red, very thick, dotted, ->-] (10,-5) -- (10,-2);
\draw[red, very thick, dotted, -<-] (10,5) -- (10,2);
%%% middle vertical edges %%%
\draw (0,3) -- (0,6);
\draw (0,-3) -- (0,-6);
%%% lower horizontal edge %%%
\draw (-10,-5) to[bend right=30] (-9,-6) -- (9,-6) to[bend right=30] (10,-5);
%%% upper horizontal edge %%%
\draw (-10,5) to[bend left=30] (-9,6) -- (9,6) to[bend left=30] (10,5);
%%% leftmost Y %%%
\draw (-10,-2) -- (-10,2);
\draw (-10,0) -- (-11.5,0) node[left=-3pt]{$a$};
%%% rightmost Y %%%
\draw (10,-2) -- (10,2);
\draw (10,0) -- (11.5,0) node[right=-3pt]{$b$};
%%% hairs %%%
\draw[ut] (-5,-6) -- (-5,-7.5) node[below=-3pt]{$t$};
\draw[ut] (5,-6) -- (5,-7.5) node[below=-3pt]{$u$};
\draw[ut] (-5,6) -- (-5,7.5) node[above=-3pt]{$v$};
\draw[ut] (5,6) -- (5,7.5) node[above=-3pt]{$w$};
\end{tikzpicture} \\
\ \overset{\beta^{2}}{\longmapsto} \ &
\begin{tikzpicture}[baseline=-3pt, x=1mm, y=1.2mm]
\newsmalltheta{2}
%%% hairs %%%
%% lower left %%
\draw[ut] (-3,-4) -- (-3,-5.5) node[below=-3pt]{$t$};
\draw (-6,-4) -- (-6,-5.5) node[below=-3pt]{$a$};
%% lower right %%
\draw[ut] (3,-4) -- (3,-5.5) node[below=-3pt]{$u$};
\draw (6,-4) -- (6,-5.5) node[below=-3pt]{$b$};
%% upper left %%
\draw[ut] (-3,4) -- (-3,5.5) node[above=-3pt]{$v$};
%% upper right %%
\draw[ut] (3,4) -- (3,5.5) node[above=-3pt]{$w$};
\end{tikzpicture}
\ - \ 
\begin{tikzpicture}[baseline=-3pt, x=1mm, y=1.2mm]
\newsmalltheta{2}
%%% hairs %%%
%% lower left %%
\draw[ut] (-3,-4) -- (-3,-5.5) node[below=-3pt]{$t$};
%% lower right %%
\draw[ut] (3,-4) -- (3,-5.5) node[below=-3pt]{$u$};
\draw (6,-4) -- (6,-5.5) node[below=-3pt]{$b$};
%% upper left %%
\draw[ut] (-3,4) -- (-3,5.5) node[above=-3pt]{$v$};
\draw (-6,4) -- (-6,5.5) node[above=-3pt]{$a$};
%% upper right %%
\draw[ut] (3,4) -- (3,5.5) node[above=-3pt]{$w$};
\end{tikzpicture}
\ - \ 
\begin{tikzpicture}[baseline=-3pt, x=1mm, y=1.2mm]
\newsmalltheta{2}
%%% hairs %%%
%% lower left %%
\draw[ut] (-3,-4) -- (-3,-5.5) node[below=-3pt]{$t$};
\draw (-6,-4) -- (-6,-5.5) node[below=-3pt]{$a$};
%% lower right %%
\draw[ut] (3,-4) -- (3,-5.5) node[below=-3pt]{$u$};
%% upper left %%
\draw[ut] (-3,4) -- (-3,5.5) node[above=-3pt]{$v$};
%% upper right %%
\draw[ut] (3,4) -- (3,5.5) node[above=-3pt]{$w$};
\draw (6,4) -- (6,5.5) node[above=-3pt]{$b$};
\end{tikzpicture}
\ + \ 
\begin{tikzpicture}[baseline=-3pt, x=1mm, y=1.2mm]
\newsmalltheta{2}
%%% hairs %%%
%% lower left %%
\draw[ut] (-3,-4) -- (-3,-5.5) node[below=-3pt]{$t$};
%% lower right %%
\draw[ut] (3,-4) -- (3,-5.5) node[below=-3pt]{$u$};
%% upper left %%
\draw[ut] (-3,4) -- (-3,5.5) node[above=-3pt]{$v$};
\draw (-6,4) -- (-6,5.5) node[above=-3pt]{$a$};
%% upper right %%
\draw[ut] (3,4) -- (3,5.5) node[above=-3pt]{$w$};
\draw (6,4) -- (6,5.5) node[above=-3pt]{$b$};
\end{tikzpicture} \ .
\end{align*}
It is easy to generalize this argument for any $t,u,v,w\in \mathcal{T}$.
\end{proof}

Returning to the case $\deg v + \deg w \ge 2$, we first assume that both $\deg v$ and $\deg w$ are positive. 
Then one can write $v = v'a$ and
%$w = w'b$
$w = bw'$ 
for some pure tensors $v',w'$ and $a,b \in H$.
We have 
\begin{align*}
& \Theta_R(t,u,v,w) - \Theta_R(at,u,v',w)
- \Theta_R(t,ub,v,w') + \Theta_R(at,ub,v',w') \\
= & \  
\begin{tikzpicture}[baseline=-3pt, x=1mm, y=1.2mm]
\newsmalltheta{2}
%%% hairs %%%
%% upper left %%
\draw[ut] (-3,4) -- (-3,5.5) node[above=-3pt]{$v'$};
\draw (-6,4) -- (-6,5.5) node[above=-3pt]{$a$};
%% upper right %%
\draw[ut] (3,4) -- (3,5.5) node[above=-3pt]{$w'$};
\draw (6,4) -- (6,5.5) node[above=-3pt]{$b$}; 
%% lower left %%
\draw[ut] (-3,-4) -- (-3,-5.5) node[below=-3pt]{$t$};
%% lower right %%
\draw[ut] (3,-4) -- (3,-5.5) node[below=-3pt]{$u$};
\end{tikzpicture}
\ - \ 
\begin{tikzpicture}[baseline=-3pt, x=1mm, y=1.2mm]
\newsmalltheta{2}
%%% hairs %%%
%% upper left %%
\draw[ut] (-3,4) -- (-3,5.5) node[above=-3pt]{$v'$};
%% upper right %%
\draw[ut] (3,4) -- (3,5.5) node[above=-3pt]{$w'$};
\draw (6,4) -- (6,5.5) node[above=-3pt]{$b$}; 
%% lower left %%
\draw[ut] (-3,-4) -- (-3,-5.5) node[below=-3pt]{$t$};
\draw (-6,-4) -- (-6,-5.5) node[below=-3pt]{$a$};
%% lower right %%
\draw[ut] (3,-4) -- (3,-5.5) node[below=-3pt]{$u$};
\end{tikzpicture}
\ - \ 
\begin{tikzpicture}[baseline=-3pt, x=1mm, y=1.2mm]
\newsmalltheta{2}
%%% hairs %%%
%% upper left %%
\draw[ut] (-3,4) -- (-3,5.5) node[above=-3pt]{$v'$};
\draw (-6,4) -- (-6,5.5) node[above=-3pt]{$a$};
%% upper right %%
\draw[ut] (3,4) -- (3,5.5) node[above=-3pt]{$w'$};
%% lower left %%
\draw[ut] (-3,-4) -- (-3,-5.5) node[below=-3pt]{$t$};
%% lower right %%
\draw[ut] (3,-4) -- (3,-5.5) node[below=-3pt]{$u$};
\draw (6,-4) -- (6,-5.5) node[below=-3pt]{$b$}; 
\end{tikzpicture}
\ + \ 
\begin{tikzpicture}[baseline=-3pt, x=1mm, y=1.2mm]
\newsmalltheta{2}
%%% hairs %%%
%% upper left %%
\draw[ut] (-3,4) -- (-3,5.5) node[above=-3pt]{$v'$};
%% upper right %%
\draw[ut] (3,4) -- (3,5.5) node[above=-3pt]{$w'$};
%% lower left %%
\draw[ut] (-3,-4) -- (-3,-5.5) node[below=-3pt]{$t$};
\draw (-6,-4) -- (-6,-5.5) node[below=-3pt]{$a$};
%% lower right %%
\draw[ut] (3,-4) -- (3,-5.5) node[below=-3pt]{$u$};
\draw (6,-4) -- (6,-5.5) node[below=-3pt]{$b$}; 
\end{tikzpicture} \ .
\end{align*}
By Lemma~\ref{lem:induction1and2}, the right hand side is in the image of the map \eqref{eq:cokbeta}, and by the inductive assumption the second to fourth terms in the left hand side are also in its image.
Therefore, the same holds for the first term in the left hand side, which completes the case of $\deg v, \deg w >0$. 

Next, consider the case where either $\deg v$ or $\deg w$ is zero.
By symmetry, we may assume that $\deg w = 0$.
One can write $v= av'$ for some pure tensor $v'$ and $a \in H$.
Using the IHX and AS relations, we compute
\[
\Theta_R(t,u,v,1) = \Theta_R(t,u,v',a)
+ \,
\begin{tikzpicture}[baseline=-3pt, x=1mm, y=1.2mm]
\newsmalltheta{1}
%%% hairs %%%
%% upper left %%
\draw[ut] (-6,4) -- (-6,5.5) node[above=-3pt]{$v'$};
%% upper right %%
\draw (3,4) -- (3,5.5) node[above=-3pt]{$a$};
%% lower left %%
\draw[ut] (-4.5,-4) -- (-4.5,-5.5) node[below=-3pt]{$t$};
%% lower middle %%
\draw[ut] (0,-3) -- (1.5,-3) node[right=-3pt]{$u$};
\end{tikzpicture}
\, + \, 
\begin{tikzpicture}[baseline=-3pt, x=1mm, y=1.2mm]
\newsmalltheta{2}
%%% hairs %%%
%% upper left %%
\draw[ut] (-6,4) -- (-6,5.5) node[above=-3pt]{$v'$};
%% upper right %%
\draw (3,4) -- (3,5.5) node[above=-3pt]{$a$};
%% lower left %%
\draw[ut] (-4.5,-4) -- (-4.5,-5.5) node[below=-3pt]{$t$};
%% lower middle %%
\draw[ut] (0,-3) -- (1.5,-3) node[right=-3pt]{$u$};
\end{tikzpicture}
\, + \, 
\begin{tikzpicture}[baseline=-3pt, x=1mm, y=1.2mm]
\newsmalltheta{3}
%%% hairs %%%
%% upper left %%
\draw[ut] (-6,4) -- (-6,5.5) node[above=-3pt]{$v'$};
%% upper right %%
\draw (3,4) -- (3,5.5) node[above=-3pt]{$a$};
%% lower left %%
\draw[ut] (-4.5,-4) -- (-4.5,-5.5) node[below=-3pt]{$t$};
%% lower middle %%
\draw[ut] (0,-3) -- (1.5,-3) node[right=-3pt]{$u$};
\end{tikzpicture} \ .
\]
Applying the first case to the first term and the rest on the right hand side, we conclude that the left hand side is in the image of the map \eqref{eq:cokbeta}.

This completes the proof of Theorem~\ref{thm:omega2}.
\qed

\subsection{Proof of Theorem~\ref{thm:omega2_presentation}} \label{subsec:pf_Thm1.1}

Combining Theorem~\ref{thm:omega2} and the presentation of $C_{1,2}\H$ given in \eqref{eq:C12H}, we prove Theorem~\ref{thm:omega2_presentation}.  
For pure tensors $t,u,v,w \in \mathcal{T}$, let 
\[
\Theta(t,u,v,w) := \, 
\begin{tikzpicture}[baseline=-3pt, x=1.2mm, y=1.4mm]
\newsmalltheta{2}
%%% lower left %%% 
\draw[ut] (-4,-4) -- (-4,-5.5) node[below=-3pt]{$t$}; 
%%% lower right %%%
\draw[ut] (4,-4) -- (4,-5.5) node[below=-3pt]{$u$}; 
%%% upper left %%%
\draw[ut] (-4,4) -- (-4,5.5) node[above=-3pt]{$v$}; 
%%% upper right %%% 
\draw[ut] (4,4) -- (4,5.5) node[above=-3pt]{$w$}; 
\end{tikzpicture}
\, = \,  
\begin{tikzpicture}[baseline=-3pt, x=1.5mm, y=1.8mm]
\newsmalltheta{2}
%%% lower left %%% 
\draw (-7,-4) -- (-7,-5.5) node[below=-3pt]{$t_1$}; 
\draw (-4.5,-5) node{$\dots$};
\draw (-2,-4) -- (-2,-5.5) node[below=-3pt]{$t_i$}; 
%%% lower right %%%
\draw (2,-4) -- (2,-5.5) node[below=-3pt]{$u_1$};
\draw (4.5,-5) node{$\dots$};
\draw (7,-4) -- (7,-5.5) node[below=-3pt]{$u_j$};
%%% upper left %%%
\draw (-2,4) -- (-2,5.5) node[above=-3pt]{$v_1$}; 
\draw (-4.5,5) node{$\dots$};
\draw (-7,4) -- (-7,5.5) node[above=-3pt]{$v_j$}; 
%%% upper right %%% 
\draw (7,4) -- (7,5.5) node[above=-3pt]{$w_1$};
\draw (4.5,5) node{$\dots$};
\draw (2,4) -- (2,5.5) node[above=-3pt]{$w_l$}; 
\end{tikzpicture} \ ,
\]
\[
\mathcal{D}(t,u,v,w) := \, 
\begin{tikzpicture}[baseline=-3pt, x=1.4mm, y=1.4mm]
\newdumbbell 
%%% lower left %%% 
\draw[ut] (-6.5,-4) -- (-6.5,-5.5) node[below=-3pt]{$t$};
%%% lower right %%% 
\draw[ut] (6.5,-4) -- (6.5,-5.5) node[below=-3pt]{$u$};
%%% upper left %%% 
\draw[ut] (-6.5,4) -- (-6.5,5.5) node[above=-3pt]{$v$};
%%% upper right %%% 
\draw[ut] (6.5,4) -- (6.5,5.5) node[above=-3pt]{$w$};
\end{tikzpicture}
\, = \, 
\begin{tikzpicture}[baseline=-3pt, x=1.8mm, y=1.8mm]
\newdumbbell 
%%% lower left %%% 
\draw (-8.5,-4) -- (-8.5,-5.5) node[below=-3pt]{$t_1$};
\draw (-6.5,-5) node{$\dots$};
\draw (-4.5,-4) -- (-4.5,-5.5) node[below=-3pt]{$t_i$};
%%% lower right %%% 
\draw (4.5,-4) -- (4.5,-5.5) node[below=-3pt]{$u_1$};
\draw (6.5,-5) node{$\dots$};
\draw (8.5,-4) -- (8.5,-5.5) node[below=-3pt]{$u_j$};
%%% upper left %%% 
\draw (-4.5,4) -- (-4.5,5.5) node[above=-3pt]{$v_1$};
\draw (-6.5,5) node{$\dots$}; 
\draw (-8.5,4) -- (-8.5,5.5) node[above=-3pt]{$v_k$};
%%% upper right %%% 
\draw (8.5,4) -- (8.5,5.5) node[above=-3pt]{$w_1$}; 
\draw (6.5,5) node{$\dots$}; 
\draw (4.5,4) -- (4.5,5.5) node[above=-3pt]{$w_l$};
\end{tikzpicture} \ .
\]
Any generator of $C_{1,2}\H$ in \eqref{eq:C12generators} can be written as a linear combination of these elements,
%$\Theta(t,u,v,w)$ and $\mathcal{D}(t,u,v,w)$, 
since we can move the hairs colored with $x$ to other horizontal edges.
Thus, the space $C_{1,2}\H$ is generated by these elements, and so is $\widetilde{\Omega}_2$.
By \eqref{eq:C12H} and Theorem~\ref{thm:omega2}, the following five types of relations give a complete set of relations in $\widetilde{\Omega}_{2}$.
\begin{description}
\item[\rm (ML)]
The multi-linearity relation on $H$-colorings.
\item[\rm (CS)]
This comes from the core symmetry relation for $C_{1,2}\H$.
For the theta graph, the reflections with respect to the horizontal and vertical axes give 
\[
\Theta(t,u,v,w) = \Theta(\overline{v},\overline{w},\overline{t},\overline{u}), 
\qquad 
\Theta(t,u,v,w) = \Theta(\overline{u}, \overline{t}, \overline{w}, \overline{v}).
\]
For the dumbbell graph, turning the left cycle upside down gives
\[
\mathcal{D}(t,u,v,w) = \mathcal{D}(\overline{v},u,\overline{t},w),
\]
and the reflection with respect to the vertical axis gives
\[
\mathcal{D}(t,u,v,w) = \mathcal{D}(\overline{u}, \overline{t}, \overline{w}, \overline{v}).
\]

\item[\rm (HB)]
The handle balance relation. 
In $C_{1,2}\H$ this corresponds to the core IHX relations, and comes from the ambiguity of rewriting generators in \eqref{eq:C12generators} in terms of $\Theta(t,u,v,w)$ and $\mathcal{D}(t,u,v,w)$. 
In more detail, consider a generator of dumbbell type in \eqref{eq:C12generators}.
The extra hairs colored with $x$ can be moved to either left or right cycle.
The two choices give rise to a relation among generators of dumbbell type.
We are reduced to the case where there is only one extra hair, and in this case we have 
\begin{align*}
\begin{tikzpicture}[baseline=-3pt, x=1.2mm, y=1.2mm]
\newdumbbell 
%%% middle %%%
\draw (0,0) -- (0,1.5) node[above=-3pt]{$a$}; 
%%% lower left %%% 
\draw[ut] (-8,-4) -- (-8,-5.5) node[below=-3pt]{$t$};
%%% lower right %%% 
\draw[ut] (8,-4) -- (8,-5.5) node[below=-3pt]{$u$};
%%% upper left %%% 
\draw[ut] (-8,4) -- (-8,5.5) node[above=-3pt]{$v$};
%%% upper right %%% 
\draw[ut] (8,4) -- (8,5.5) node[above=-3pt]{$w$};
\end{tikzpicture} 
\, = & \  
\begin{tikzpicture}[baseline=-3pt, x=1.2mm, y=1.2mm]
\newdumbbell 
%%% lower left %%% 
\draw[ut] (-8,-4) -- (-8,-5.5) node[below=-3pt]{$t$};
%%% lower right %%% 
\draw[ut] (8,-4) -- (8,-5.5) node[below=-3pt]{$u$};
%%% upper left %%% 
\draw[ut] (-8,4) -- (-8,5.5) node[above=-3pt]{$v$};
\draw (-5,4) -- (-5,5.5) node[above=-3pt]{$a$};
%%% upper right %%% 
\draw[ut] (8,4) -- (8,5.5) node[above=-3pt]{$w$};
\end{tikzpicture} 
\ - \ 
\begin{tikzpicture}[baseline=-3pt, x=1.2mm, y=1.2mm]
\newdumbbell 
%%% lower left %%% 
\draw[ut] (-8,-4) -- (-8,-5.5) node[below=-3pt]{$t$};
\draw (-5,-4) -- (-5,-5.5) node[below=-3pt]{$a$};
%%% lower right %%% 
\draw[ut] (8,-4) -- (8,-5.5) node[below=-3pt]{$u$};
%%% upper left %%% 
\draw[ut] (-8,4) -- (-8,5.5) node[above=-3pt]{$v$};
%%% upper right %%% 
\draw[ut] (8,4) -- (8,5.5) node[above=-3pt]{$w$};
\end{tikzpicture} \\
\ = & \   
\begin{tikzpicture}[baseline=-3pt, x=1.2mm, y=1.2mm]
\newdumbbell 
%%% lower left %%% 
\draw[ut] (-8,-4) -- (-8,-5.5) node[below=-3pt]{$t$};
%%% lower right %%% 
\draw[ut] (8,-4) -- (8,-5.5) node[below=-3pt]{$u$};
%%% upper left %%% 
\draw[ut] (-8,4) -- (-8,5.5) node[above=-3pt]{$v$};
%%% upper right %%% 
\draw[ut] (8,4) -- (8,5.5) node[above=-3pt]{$w$};
\draw (5,4) -- (5,5.5) node[above=-3pt]{$a$};
\end{tikzpicture} 
\ - \ 
\begin{tikzpicture}[baseline=-3pt, x=1.2mm, y=1.2mm]
\newdumbbell 
%%% lower left %%% 
\draw[ut] (-8,-4) -- (-8,-5.5) node[below=-3pt]{$t$};
%%% lower right %%% 
\draw[ut] (8,-4) -- (8,-5.5) node[below=-3pt]{$u$};
\draw (5,-4) -- (5,-5.5) node[below=-3pt]{$a$};
%%% upper left %%% 
\draw[ut] (-8,4) -- (-8,5.5) node[above=-3pt]{$v$};
%%% upper right %%% 
\draw[ut] (8,4) -- (8,5.5) node[above=-3pt]{$w$};
\end{tikzpicture} \, .
\end{align*}
Therefore, we obtain
\[
\mathcal{D}(t,u,av,w) - \mathcal{D}(ta,u,v,w) 
= \mathcal{D}(t,u,v,wa) - \mathcal{D}(t,au,v,w).
\]
A similar argument for the theta graph yields
\[
\Theta(t,u,v,wa) - \Theta(t,u,av,w) 
= \Theta(t,au,v,w) - \Theta(ta,u,v,w).
\]

\item[\rm (CC)]
The core change relation. 
This has already be explained in Section~\ref{subsec:C12H}.
We have
$
\mathcal{D}(t,u,v,w) = \Theta(t,u,v,w) + \Theta(\overline{v},u,\overline{t},w)
$.
\item[\rm (GC3)]
We have $\Theta_R(t,u,v,w) = 0$, namely
\[
\begin{tikzpicture}[baseline=-3pt, x=1.2mm, y=1.44mm]
\newsmalltheta{1}
%%% hairs %%%
%% upper left %%
\draw[ut] (-4.5,4) -- (-4.4,5.5) node[above=-3pt]{$v$};
%% upper right %%
\draw[ut] (4.5,4) -- (4.5,5.5) node[above=-3pt]{$w$};
%% lower left %%
\draw[ut] (-4.5,-4) -- (-4.5,-5.5) node[below=-3pt]{$t$};
%% lower right %%
\draw[ut] (4.5,-4) -- (4.5,-5.5) node[below=-3pt]{$u$};
\end{tikzpicture}
\, + \, 
\begin{tikzpicture}[baseline=-3pt, x=1.2mm, y=1.44mm]
\newsmalltheta{2}
%%% hairs %%%
%% upper left %%
\draw[ut] (-4.5,4) -- (-4.5,5.5) node[above=-3pt]{$v$};
%% upper right %%
\draw[ut] (4.5,4) -- (4.5,5.5) node[above=-3pt]{$w$};
%% lower left %%
\draw[ut] (-4.5,-4) -- (-4.5,-5.5) node[below=-3pt]{$t$};
%% lower right %%
\draw[ut] (4.5,-4) -- (4.5,-5.5) node[below=-3pt]{$u$};
\end{tikzpicture}
\, + \, 
\begin{tikzpicture}[baseline=-3pt, x=1.2mm, y=1.44mm]
\newsmalltheta{3}
%%% hairs %%%
%% upper left %%
\draw[ut] (-4.5,4) -- (-4.5,5.5) node[above=-3pt]{$v$};
%% upper right %%
\draw[ut] (4.5,4) -- (4.5,5.5) node[above=-3pt]{$w$};
%% lower left %%
\draw[ut] (-4.5,-4) -- (-4.5,-5.5) node[below=-3pt]{$t$};
%% lower right %%
\draw[ut] (4.5,-4) -- (4.5,-5.5) node[below=-3pt]{$u$};
\end{tikzpicture}
=0.
\]
The first term is equal to 
\begin{align*}
\begin{tikzpicture}[baseline=-3pt, x=2mm, y=2mm]
%%%
%%% Fat Theta(t,u,v,w) %%%
%%%
%%% middle vertical line %%%
\draw (0,6) -- (0,-6);  
%%% right dotted edge %%%
\draw[red, very thick, dotted, -<-] (10,4) -- (10,-4);   
%%% left dotted edge %%%
\draw[red, very thick, dotted, -<-] (-10,4) -- (-10,-4); 
%%% big boundaries %%%
\draw (0,6) to (-9,6) to[bend right=30] (-10,5) to (-10,4);
\draw (0,-6) to (-9,-6) to[bend left=30] (-10,-5) to (-10,-4);
\draw (0,6) to (9,6) to[bend left=30] (10,5) to (10,4);
\draw (0,-6) to (9,-6) to[bend right=30] (10,-5) to (10,-4);
%%% hairs %%%
%% upper right %%
\draw[ut] (6,6) -- (6,7.5) node[above=-3pt]{$w$};
%% lower right %%
\draw[ut] (6,-6) -- (6,-7.5) node[below=-3pt]{$u$};
%% middle vertical staff %%
\draw (0,5) -- (-1.5,5) node[left=-3pt]{$v_1$};
\draw (0,1) -- (-1.5,1) node[left=-3pt]{$v_k$};
\draw (-3,3.5) node{$\vdots$};
\draw (0,-1) -- (-1.5,-1) node[left=-3pt]{$t_1$};
\draw (0,-5) -- (-1.5,-5) node[left=-3pt]{$t_i$};
\draw (-3,-2.5) node{$\vdots$};
\end{tikzpicture} \, 
& = \sum_{\substack{A \sqcup B = \{ 1,\ldots,i \} \\ C \sqcup D = \{ 1,\ldots, k\}} } \, 
\begin{tikzpicture}[baseline=-3pt, x=2mm, y=2mm]
\newsmalltheta{2}
\draw[ut] (6,4) -- (6,5.5) node[above=-3pt]{$w$};
\draw[ut] (3,4) -- (3,5.5) node[above=-3pt]{$v_D$};
\draw[ut] (-3,4) -- (-3,2.5) node[below=-3pt]{$v_C$};
\draw[ut] (-3,-4) -- (-3,-2.5) node[above=-3pt]{$t_A$};
\draw[ut] (3,-4) -- (3,-5.5) node[below=-3pt]{$t_B$};
\draw[ut] (6,-4) -- (6,-5.5) node[below=-3pt]{$u$}; 
\end{tikzpicture}
 \\
&= \Theta(\overline{t'},t''u,\overline{v'},wv'').
\end{align*}
Here, we have used the IHX relation to move away the hairs attached to the handle of the core theta graph: $v$'s go above and $t$'s below.
Similarly, the third term is equal to $\Theta(tu'',\overline{u'},w''v,\overline{w'})$.
Thus we obtain
\[
\Theta(t,u,v,w)
+ \Theta(\overline{t'},t''u,\overline{v'},wv'')
+ \Theta(tu'',\overline{u'},w''v,\overline{w'}) = 0.
\]
\end{description}

By (CC) we can remove generators of the dumbbell type and write everything in terms of generators of the theta type.
We can check that the core symmetry and handle balance relations concerning the dumbbell graph follow from the corresponding relations for the theta graph.
This completes the proof of Theorem~\ref{thm:omega2_presentation}.

\section{The Johnson cokernel in degree $6$} \label{sec:deg6}
The general linear group $\GL=\GL(2g,\Q)$ naturally acts on the $H$-colorings of leaves of hairy Lie graphs,
and $C_{1,r}\H(n)$ and $\widetilde{\Omega}_{r,n+2-2r}$ are natural $\GL$-representation spaces. 

\subsection{The $2$-loop spaces in low degrees}
\label{subsec:table}
As explained in Section~\ref{subsec:tracemap}, we have $\Sp$-isomorphisms
\[
\h(n)\cong C_1 \H_{\rm top}(n)
\quad \text{ and } \quad 
\cok(n)\cong \bigoplus_{r=1}^{\lfloor n/2 \rfloor +1} \widetilde{\Omega}_{r,\braket{n+2-2r}}
\]
for $g\gg n$.
As a corollary of these isomorphisms and \cite[Table 1]{MSS15AM},
using \cite[(25.39)]{FH91},
we obtain the $\GL$-decompositions of $C_{1,r}\H(n)$ and $\widetilde{\Omega}_{r,n+2-2r}$ for small $n$.
With respect to the 2-loop part, we have the following:
\begin{proposition} \label{prop:GL_table}
When $g \gg n$,
the $\GL$-decompositions of $C_{1,2}\H(n)$ and $\widetilde{\Omega}_{2,n-2}$ up to Lie degree $7$ are as follows: 
\[
\renewcommand{\arraystretch}{1.5}
\begin{array}{|c|c|c|}
\hline
n&C_{1,2}\H(n)&\widetilde{\Omega}_{2,n-2}\\\hline
3& \emptyset & \emptyset \\
4& 3[2] & [2]\\
5& [3]\, 3[21]\, 2[1^3] & 2[21]\, 2[1^3]\\
6& 6[4]\, 9[31]\, 12[2^2]\, 9[21^2]\, 6[1^4] & 2[4]\, 3[31]\, 3[2^2]\, 3[21^2]\, 2[1^4] \\
7& 
\renewcommand{\arraystretch}{1.0}
\begin{array}{c}
3[5]\, 15[41]\, 18[32]\, 25[31^2] \\
18[2^21]\, 16[21^3]\, 3[1^5]
\end{array}
\renewcommand{\arraystretch}{1.5}
& 6[41]\, 6[32]\, 11[31^2]\, 6[2^21]\, 5[21^3] \\ \hline
\end{array}
\renewcommand{\arraystretch}{1.0}
\]
Here, we simply denote $[2] = [2]_\GL$, etc.
\end{proposition}

\subsection{$[1^4]_\GL$-components in $C_{1,2}\H(6)$} \label{subsec:1^4}

According to Proposition~\ref{prop:GL_table}, there are stably $6$ copies of $[1^4]_{\GL}$ in $C_{1,2}\H(6)$.
We give a direct proof of this. 

\begin{proposition} \label{prop:dim6}
When $g\ge2$,
the multiplicity of the irreducible representation $[1^4]_{\GL}=\Lambda^4H$ in $C_{1,2}\H(6)$ is $6$.
\end{proposition}
\begin{proof}
First, consider a nontrivial $\GL$-homomorphism
$\Phi\colon H^{\otimes4}\to\Lambda^4H$.
As is well-known, the multiplicity of $[1^4]_\GL$ in the $\GL$-module $H^{\otimes4}$ is one when $g\ge2$.
Hence, by Schur's lemma, there exists some $k \in \Q$ such that
\[
\Phi(a_1\otimes a_2\otimes a_3\otimes a_4)=k(a_1\wedge a_2\wedge a_3\wedge a_4)
\]
for all $a_1,a_2,a_3,a_4\in H$. 

Next, let $\pi:C_{1,2}\H(6)\to \Lambda^4H$ be a $\GL$-homomorphism. 
Pick a tree-shaped Jacobi diagram $J$ of degree $6$ with $2$ dotted edges attached.
There are $4$ leaves in $J$ which we denote by $v_1$, $v_2$, $v_3$, $v_4$.
By setting their colors, we have a $\GL$-homomorphism $f\colon H^{\otimes 4}\to C_{1,2}\H(6)$ defined by
\[
f(a_1\otimes a_2\otimes a_3\otimes a_4)=J_{a_1,a_2,a_3,a_4},
\]
where $J_{a_1,a_2,a_3,a_4}\in C_{1,2}\H(6)$ is
obtained by coloring the leaf $v_i$ with $a_i$.
As we saw, there exists some $k(J) \in \Q$ such that
\[
\pi(J_{a_1,a_2,a_3,a_4})=(\pi\circ f)(a_1\otimes a_2\otimes a_3\otimes a_4)=k(J)\, a_1\wedge a_2\wedge a_3\wedge a_4.
\]
Note also that 
\begin{equation} \label{eq:piJsigma}
\pi(J_{a_{\sigma(1)}, a_{\sigma(2)},a_{\sigma(3)}, a_{\sigma(4)}})=\sign(\sigma)\pi(J_{a_1,a_2,a_3,a_4})
\end{equation}
for any permutation $\sigma$ of degree $4$.

In the following,
we will always draw hairs of the diagram $J$ vertically and make them distinguishable by their horizontal positions.
Then we denote the 4 leaves of $J$, ordered from left to right, by $v_1,v_2,v_3,v_4$. 
With this convention, we sometimes omit the $H$-coloring of $J_{a_1,a_2,a_3,a_4}$ from the figure. 
For example, 
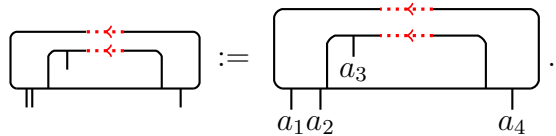
\[
\thetagraph21001:=
\begin{tikzpicture}[scale=0.7, baseline={(0,0.3)},show background rectangle,
background rectangle/.style={fill=none},
        inner frame sep=1mm]
  \coordinate (origin) at (0,0); 
\draw (2.35,0) to[bend right=30] (2.5,0.15)--(2.5,1.35) to[bend right=30] (2.35,1.5) -- (0.5,1.5);
\draw (-2.35,0) to[bend left=30] (-2.5,0.15) -- (-2.5,1.35) to[bend left=30] (-2.35,1.5)-- (-0.5,1.5);
\draw (-2.35,0)--(2.35,0);
\draw (1.5,0) -- (1.5,0.85) to[bend right=20] (1.35,1) -- (0.5,1);
\draw (-1.5,0)-- (-1.5,0.85) to[bend left=20] (-1.35,1) --(-0.5,1);
\draw [red, very thick, dotted, ->-={.5}{red}] (0.5,1.5)--(-0.5,1.5);
\draw [red, very thick, dotted, ->-={.5}{red}] (0.5,1)--(-0.5,1);
\draw (-2-0.175,0) -- (-2-0.175,-0.4) [below=-1mm] node{$a_1$} ;
\draw (-2+0.375,0) -- (-2+0.375,-0.4) [below=-1mm] node{$a_2$};
\draw (-1+0.075-0.15+0.075,1) -- (-1+0.075-0.15+0.075,1-0.4) [below=-1mm] node{$a_3$};
\draw (2-0.075+0.15-0.075,0) -- (2-0.075+0.15-0.075,-0.4) [below=-1mm] node{$a_4$};
\end{tikzpicture}.
\]
We now compute the value of $\pi:C_{1,2}H(6)\to \Lambda^4H$ in some explicit way.
Consider the following diagrams:
\begin{align*}
& J^1 = \thetagraph40000, \ 
J^2 = \thetagraph20200, \ 
J^3 = \thetagraph10210, \\
& J^4 = \thetagraph30100, \ 
J^5 = \thetagraph20110, \ 
J^6 = \thetagraph02020, \\ 
&  J^7 = \thetagraph03010.
\end{align*}
We can check that any hairy Lie graph in $C_{1,2}\H(6)$ colored with $a_1$, $a_2$, $a_3$, $a_4\in H$ is written as a linear combination of the above $7$ diagrams (viewed as elements in $C_{1,2}\H(6)$ by the above convention) and those obtained from such diagrams by permuting the $H$-coloring.
Let us denote 
\begin{align}
& p = k(J^1),\ 
q = k(J^2), \ 
r = k(J^3), \ 
s = k(J^4), \ \nonumber \\ 
& t = k(J^5), \
u = k(J^6), \ 
v = k(J^7). \label{eq:pqrstuv} 
\end{align}
Together with the property~\eqref{eq:piJsigma},
these values determine the map $\pi$. 
We give explicit values for the generators of theta type in \eqref{eq:C12generators}.
If we remove the $H$-coloring of any generator of theta type, then up to the AS relation we obtain one of the following $22$ diagrams:
\begin{equation} \label{eq:manyhairs}
\begin{array}{cccc}
\thetagraph00400,&\thetagraph10300,&\thetagraph20200,&\thetagraph11200,\\
2q&q&q&0\\
\thetagraph10210,&\thetagraph10201,&\thetagraph30100,&\thetagraph21100,\\
r&q+r&s&-q+s\\
\thetagraph20110,&\thetagraph20101,&\thetagraph11110,\\
t&q+t&-r+t 
\end{array}
\end{equation}
and
\[
\begin{array}{cccc}
\thetagraph40000,&\thetagraph31000,&\thetagraph30010,&\thetagraph30001,\\
p&p-s&-s-v&-v\\
\thetagraph22000,&\thetagraph20020,&\thetagraph20002,&\thetagraph21010,\\
p-q&-q+u&u&q+r+v\\
\thetagraph21001,&\thetagraph20011,&\thetagraph11011.\\
-q-t-v&-q-t+u&-q+r-2t+u
\end{array}
\]
Here, the values of $k$ are written under each diagram. 
These values are computed by expressing each diagram as a linear combination of $J^1,\ldots,J^7$ by the IHX and AS relations.
Using the core change relation we can obtain the values for the generators of dumbbell type, which we omit.

We need to check that these values are consistent with all the relations in $C_{1,2}\H$.
By the core IHX relation 
\[
\thetagraph20110 + \thetagraph21010-\thetagraph30010=0,
\]
we obtain $t + (q+r+v) - (-s-v) = 0$, namely
\begin{equation} \label{eq:qrst2v}
q+r+s+t+2v=0.
\end{equation}
Conversely, this condition ensures that the assignment~\eqref{eq:pqrstuv} gives rise to a well-defined map $\pi: C_{1,2}\H(6) \to \Lambda^4 H$. 
We first check that any core IHX relation among theta graphs is a consequence of \eqref{eq:qrst2v}.
To see this, we need to consider the IHX relations around the two trivalent vertices of the core theta graph.
Furthermore, it suffices to consider the IHX relation which moves hairs attached to the bottom horizontal edge in the $11$ diagrams in \eqref{eq:manyhairs}.
Also, exactly in the same way as \cite[Proposition~4.2]{NSS22JT}, the core IHX relations between dumbbell graphs and the core change relation between dumbbell and theta graphs are satisfied without any further condition on $p,q,r,s,t,u,v$.

We have shown that $\pi$ is determined by $7$ parameters $p,q,r,s,t,u,v$ subject to the condition \eqref{eq:qrst2v}.
This proves the proposition. 
\end{proof}
We want to construct a $\GL$-homomorphism $C_{1,2}\H(6) \to \Lambda^4 H$ which factors through $\widetilde{\Omega}_2(6)$.
By Proposition~\ref{subsec:table}, in a stable range,
the multiplicity of $[1^4]_\GL$ in $C_{1,2}\H(6)$ and in $\widetilde{\Omega}_2(6)$ is $6$ and $2$, respectively.
The next lemma gives explicit homomorphisms which detect the two components in $\widetilde{\Omega}_2(6)$.
\begin{lemma} \label{lem:3r=2t}
Let $g \ge 2$ and keep the notation as in the proof of Proposition~{\rm\ref{prop:dim6}}.
Assume that $q+r+s+t+2v=0$.
The homomorphism $\pi :C_{1,2}\H(6) \to \Lambda^4 H$ factors through $\widetilde{\Omega}_2(6)$ if and only if $p=-q=u=v$ and $3r-2t=0$.
\end{lemma}

\begin{proof}
Since $\widetilde{\Omega}_2=C_{1,2}\H/R$, where the space $R$ is defined just after Theorem~\ref{thm:omega2}, the homomorphism $\pi$ factors through $\widetilde{\Omega}_2(6)$ if and only $\pi(C_{1,2}\H(6) \cap R) = 0$.

First, assume that $\pi(C_{1,2}\H(6) \cap R)=0$. 
The following four linear combinations are elements of $C_{1,2}\H(6) \cap R$:
\begin{align*}
&\thetagraph03010-\thetagraph30001+\thetagraph00400,\\
&\thetagraph02020+\thetagraph20002+\thetagraph00400,\\
&\thetagraph22000+\thetagraph20200-\thetagraph02200,\\
&\thetagraph11101
-\thetagraph10210
+\thetagraph01111.
\end{align*}
Hence, we obtain
\[
v-(-v)+2q=u+u+2q=(p-q)+q-(-q)=(-r+t)-r+(-r+t)=0.
\]
These equalities are equivalent to $p=-q=u=v$ and $3r-2t=0$.

Next, we show that $\pi$ factors through $\widetilde{\Omega}_2(6)$ assuming the equalities $p=-q=u=v$ and $3r-2t=0$.
With Proposition~\ref{prop:dim6} in mind, we denote 
\[
\pi = \pi_{p,q,r,s,t,u}\colon C_{1,2}\H (6) \to\Lambda^4H
\]
to clarify the values of $p,q,r,s,t,u$.
The remaining parameter $v$ is given by $v=-\dfrac{1}{2}(q+r+s+t)$.
As we saw above, if $\pi$ factors through $\widetilde{\Omega}_2(6)$, $p,q,r,s,t,u$ must satisfy $p=-q=u=v$ and $3r-2t=0$.
Hence, every homomorphism which factors through $\widetilde{\Omega}_2(6)$ must be a linear combination of $\pi_{1,-1,0,-1,0,1}$ and $\pi_{0,0,2,0,3,0}$.
Assume that $g$ is sufficiently large.
Then, the multiplicity of $[1^4]_{\GL}$ in $\widetilde{\Omega}_2(6)$ is $2$.
Hence, $\pi_{1,-1,0,-1,0,1}$ and $\pi_{0,0,2,0,3,0}$ must factor through $\widetilde{\Omega}_2(6)$.

We now consider the case of arbitrary genus $g \ge 2$.
To clarify the genus, we denote 
\[
C_{1,2}\H(6)[g] =C_{1,2}\H(6),\
\Omega_2(6)[g]=\Omega_2(6),\ \text{and } \Lambda^4H[g]=\Lambda^4H.
\]
Let $g'$ be a sufficiently large integer.
For $(p,q,r,s,t,u)=(1,-1,0,-1,0,1)$ or $(0,0,2,0,3,0)$, 
the composite map
\[
\widetilde{\Omega}_2(6)[g] \xrightarrow{\text{inc.}} \widetilde{\Omega}_2(6)[g'] \xrightarrow{\pi_{p,q,r,s,t,u}}
\Lambda^4H[g']
\xrightarrow{\text{proj.}}
\Lambda^4H[g]
\]
is well-defined, and its composition with the projection $C_{1,2}\H(6)[g] \to \widetilde{\Omega}_2(6)[g]$ coincides with $\pi_{p,q,r,s,t,u}\colon C_{1,2}\H(6)[g]\to \Lambda^4H[g]$. 
Hence, $\pi_{p,q,r,s,t,u}$ factors through $\widetilde{\Omega}_2(6)$ for any $g\ge 2$.
\end{proof}

%%%%%%%
\subsection{Proof of Theorem~\ref{thm:main2}}
%%%%%%%
Theorem~\ref{thm:main2} is deduced from Theorem~\ref{thm:cok(6)} and the following:

\begin{proposition} \label{prop:14120}
When $g\ge6$, 
\[
[1^4]_\Sp+[1^2]_\Sp+[0]_\Sp\subset \Im( \widetilde{\Tr}_2 \colon \h(6) \cap \Ker  \widetilde{\Tr}_1 \to \widetilde{\Omega}_2(6)).
\]
\end{proposition}

For the proof of Proposition~\ref{prop:14120}, we use the following diagrammatic notation to describe a specific linear combination of hairy Lie graphs.
Let $S = \{ s_1,\ldots, s_n\}$ be a finite subset of $H\times H$ and $m$ an integer not greater than $n$.
Let $X$ be a tree-shaped Jacobi diagram with several dotted edges, such that among its leaves, exactly $2m$ leaves are not joined to dotted edges.
Suppose that the $2m$ leaves decompose into ordered $m$-pairs $\{ (v_i, w_i)\}_{i=1}^m$.
We denote by $X(x_1, y_1,\ldots, x_m, y_m)$ the hairy Lie graph obtained from $X$ by coloring the leaves $v_i$, $w_i$ with $x_i$, $y_i\in H$, respectively, for all $1 \le i \le m$.
Then we define
\[
X_S:=\sum X(x_1, y_1,\ldots, x_m, y_m),
\]
where the sum is taken over all
ordered $m$-pairs
$\{ (x_i,y_i)\}_{i=1}^m$ of elements of $S$ such that the set $\{ x_i, y_i \}_{i=1}^m$ is linearly independent in $H$.
We draw $X_S$ as a diagram obtained from $X$ by adding $m$ dashed lines directed from $v_i$ to $w_i$ for each $1\le i\le m$.
We call $S$ the \emph{coloring set} of $X_S$.

For example, when $S=\{(a_1,b_1),(-b_1,a_1),(a_2,a_3)\}$,
\begin{align*}
\kushiitwo14&=
\hspace*{-0.9cm}
\begin{tikzpicture}[domain=-3:4,scale=0.45, baseline={(0,0.4)}]
\draw (-3,0)--(4,0);
\draw (-2,0)--(-2,1) node[above, yshift=-1mm] {$b_1$};
\draw (-1,0)--(-1,1) node[above, yshift=-1mm] {$a_3$};
\draw (0,0)--(0,1) coordinate(P4);
\draw (1,0)--(1,1) coordinate (P5);
\draw (2,0)--(2,1) node[above, yshift=-1mm] {$a_2$};
\draw (3,0)--(3,1) node[above, yshift=-1mm] {$a_1$};
\draw[red, very thick, dotted,   ->-={.5}{red}] (-3,0) to [out=200,in=-20] (4,0);
\draw[red, very thick, dotted, -<-={.5}{red}] (P4) to [out=80,in=100] (P5);
\end{tikzpicture}
\hspace*{-0.9cm}
-
\hspace*{-0.9cm}
\begin{tikzpicture}[domain=-3:4,scale=0.45, baseline={(0,0.4)}]
\draw (-3,0)--(4,0);
\draw (-2,0)--(-2,1) node[above, yshift=-1mm] {$a_1$};
\draw (-1,0)--(-1,1) node[above, yshift=-1mm] {$a_3$};
\draw (0,0)--(0,1) coordinate(P4);
\draw (1,0)--(1,1) coordinate (P5);
\draw (2,0)--(2,1) node[above, yshift=-1mm] {$a_2$};
\draw (3,0)--(3,1) node[above, yshift=-1mm] {$b_1$};
\draw[red, very thick, dotted,   ->-={.5}{red}] (-3,0) to [out=200,in=-20] (4,0);
\draw[red, very thick, dotted, -<-={.5}{red}] (P4) to [out=80,in=100] (P5);
\end{tikzpicture}
\hspace*{-0.9cm}\\
&\quad+
\hspace*{-0.9cm}
\begin{tikzpicture}[domain=-3:4,scale=0.45, baseline={(0,0.4)}]
\draw (-3,0)--(4,0);
\draw (-2,0)--(-2,1) node[above, yshift=-1mm] {$a_3$};
\draw (-1,0)--(-1,1) node[above, yshift=-1mm] {$b_1$};
\draw (0,0)--(0,1) coordinate(P4);
\draw (1,0)--(1,1) coordinate (P5);
\draw (2,0)--(2,1) node[above, yshift=-1mm] {$a_1$};
\draw (3,0)--(3,1) node[above, yshift=-1mm] {$a_2$};
\draw[red, very thick, dotted,   ->-={.5}{red}] (-3,0) to [out=200,in=-20] (4,0);
\draw[red, very thick, dotted, -<-={.5}{red}] (P4) to [out=80,in=100] (P5);
\end{tikzpicture}
\hspace*{-0.9cm}
-
\hspace*{-0.9cm}
\begin{tikzpicture}[domain=-3:4,scale=0.45, baseline={(0,0.4)}]
\draw (-3,0)--(4,0);
\draw (-2,0)--(-2,1) node[above, yshift=-1mm] {$a_3$};
\draw (-1,0)--(-1,1) node[above, yshift=-1mm] {$a_1$};
\draw (0,0)--(0,1) coordinate(P4);
\draw (1,0)--(1,1) coordinate (P5);
\draw (2,0)--(2,1) node[above, yshift=-1mm] {$b_1$};
\draw (3,0)--(3,1) node[above, yshift=-1mm] {$a_2$};
\draw[red, very thick, dotted,   ->-={.5}{red}] (-3,0) to [out=200,in=-20] (4,0);
\draw[red, very thick, dotted, -<-={.5}{red}] (P4) to [out=80,in=100] (P5);
\end{tikzpicture}
\hspace*{-0.9cm}.
\end{align*}

\begin{proof}[Proof of Proposition~\ref{prop:14120}]
First, we prove that 
$[0]_\Sp\subset \Im( \widetilde{\Tr}_2 \colon \h(6) \cap \Ker  \widetilde{\Tr}_1 \to \widetilde{\Omega}_2(6))$.
For $S:=\{(a_j,b_j), (-b_j,a_j)\}_{j=1}^4$,
we show that 
\[
X_S:=-5\kushii0
+2\kushiii
\in\Ker \widetilde{\Tr}_1.
\]
By the definition of dashed lines, we compute 
\begin{align}
\widetilde{\Tr}_1\left(\kushii0\right)
&=2\left(\kushii1+\kushii2\right. \nonumber \\
&\left.+\kushii3+\kushii4\right), \label{eq:Tr1_dashed}
\end{align}
where the coloring set of the right hand side is also $S$.
Since the set $S$ is stationary under the involution $(x,y) \mapsto (-y,x)$, changing the direction of a dashed line gives a minus sign. 
As shown in \cite[Proof of Theorem~1.1]{Con16}, one can slide a hair along a dotted edge in $\widetilde{\Omega}_1$.
Note that this is the same as the relation (C2) for another loop space $\Omega$ whose definition will be recalled in Section~\ref{sec:Omega1_Omega_2}, and that $\widetilde{\Omega}_1 = \Omega_1$.
By using this relation and the AS and IHX relations, the right hand side of \eqref{eq:Tr1_dashed} is equal to
$-4
 \begin{tikzpicture}[scale=0.4, baseline={(0,0)}]
 \draw ([shift={(0,0)}]50:3) arc [radius=3, start angle = 50, end angle=370];
  \draw [red, very thick, dotted, ->-={.5}{red}] ([shift={(0,0)}]10:3) arc [radius=3, start angle = 10, end angle=50];
 \draw (0:3)--(0:2) coordinate (A);
 \draw (60:3)--(60:2) coordinate (B);
 \draw (120:3)--(120:2) coordinate (C);
 \draw (180:3)--(180:2) coordinate (D);
  \draw (240:3)--(240:2) coordinate (E);
 \draw (300:3)--(300:2) coordinate (F);
\draw [blue, very thick, dashed, ->-={.5}{blue}] (B) to[bend left=30] (C);
\draw [blue, very thick, dashed, ->-={.2}{blue}] (A) to[bend right=20] (E);
\draw [blue, very thick, dashed, ->-={.8}{blue}] (F) to[bend right=20] (D);
 \end{tikzpicture}$. 
Here, we have used the following equalities which are the consequences of core symmetry:
\[
\begin{tikzpicture}[scale=0.4, baseline={(0,0)}]
 \draw ([shift={(0,0)}]50:3) arc [radius=3, start angle = 50, end angle=370];
  \draw [red, very thick, dotted, ->-={.5}{red}] ([shift={(0,0)}]10:3) arc [radius=3, start angle = 10, end angle=50];
 \draw (0:3)--(0:2) coordinate (A);
 \draw (60:3)--(60:2) coordinate (B);
 \draw (120:3)--(120:2) coordinate (C);
 \draw (180:3)--(180:2) coordinate (D);
  \draw (240:3)--(240:2) coordinate (E);
 \draw (300:3)--(300:2) coordinate (F);
\draw [blue, very thick, dashed, ->-={.5}{blue}] (B) to[bend left=30] (C);
\draw [blue, very thick, dashed, ->-={.5}{blue}] (A) -- (D);
\draw [blue, very thick, dashed, ->-={.5}{blue}] (F) to[bend right=30] (E);
 \end{tikzpicture}
 =
 \begin{tikzpicture}[scale=0.4, baseline={(0,0)}]
 \draw ([shift={(0,0)}]50:3) arc [radius=3, start angle = 50, end angle=370];
  \draw [red, very thick, dotted, ->-={.5}{red}] ([shift={(0,0)}]10:3) arc [radius=3, start angle = 10, end angle=50];
 \draw (0:3)--(0:2) coordinate (A);
 \draw (60:3)--(60:2) coordinate (B);
 \draw (120:3)--(120:2) coordinate (C);
 \draw (180:3)--(180:2) coordinate (D);
  \draw (240:3)--(240:2) coordinate (E);
 \draw (300:3)--(300:2) coordinate (F);
\draw [blue, very thick, dashed, ->-={.2}{blue}] (B) -- (E);
\draw [blue, very thick, dashed, ->-={.2}{blue}] (A) -- (D);
\draw [blue, very thick, dashed, ->-={.2}{blue}] (F) -- (C);
 \end{tikzpicture}
=
 \begin{tikzpicture}[scale=0.4, baseline={(0,0)}]
 \draw ([shift={(0,0)}]50:3) arc [radius=3, start angle = 50, end angle=370];
  \draw [red, very thick, dotted, ->-={.5}{red}] ([shift={(0,0)}]10:3) arc [radius=3, start angle = 10, end angle=50];
 \draw (0:3)--(0:2) coordinate (A);
 \draw (60:3)--(60:2) coordinate (B);
 \draw (120:3)--(120:2) coordinate (C);
 \draw (180:3)--(180:2) coordinate (D);
  \draw (240:3)--(240:2) coordinate (E);
 \draw (300:3)--(300:2) coordinate (F);
\draw [blue, very thick, dashed, ->-={.5}{blue}] (A) -- (D);
\draw [blue, very thick, dashed, ->-={.3}{blue}] (B) to[bend right=15] (F);
\draw [blue, very thick, dashed, ->-={.3}{blue}] (C) to[bend left=15] (E);
 \end{tikzpicture}
=0.
\]
In the same way, we obtain the equality
\[
\widetilde{\Tr}_1\left(\kushiii\right)
= -10\begin{tikzpicture}[scale=0.4, baseline={(0,0)}]
 \draw ([shift={(0,0)}]50:3) arc [radius=3, start angle = 50, end angle=370];
  \draw [red, very thick, dotted, ->-={.5}{red}] ([shift={(0,0)}]10:3) arc [radius=3, start angle = 10, end angle=50];
 \draw (0:3)--(0:2) coordinate (A);
 \draw (60:3)--(60:2) coordinate (B);
 \draw (120:3)--(120:2) coordinate (C);
 \draw (180:3)--(180:2) coordinate (D);
  \draw (240:3)--(240:2) coordinate (E);
 \draw (300:3)--(300:2) coordinate (F);
 \draw [blue, very thick, dashed, ->-={.5}{blue}] (B) to[bend left=30] (C);
\draw [blue, very thick, dashed, ->-={.2}{blue}] (A) to[bend right=20] (E);
\draw [blue, very thick, dashed, ->-={.8}{blue}] (F) to[bend right=20] (D);
 \end{tikzpicture},
\]
and hence $\widetilde{\Tr}_1(X_S) = 0$.
Now, we compute the value $(\pi\circ \widetilde{\Tr}_2)(X_S)$, where $\pi: \widetilde{\Omega}_2(6) \to \Lambda^4 H$ is the homomorphism in Lemma~\ref{lem:3r=2t}. 
\begin{align*}
&\widetilde{\Tr}_2\left(\kushii0\right)\\
&=8\left(\kushiitwo12+\kushiitwo13+\kushiitwo14\right.\\
&\left.+\kushiitwo23+\kushiitwo24+\kushiitwo34\right),
\end{align*}
where the coloring set of the right hand side is also $S$.
Then $\pi$ sends it to 
\begin{align*}
&8\{-2q-(q+r)-u-(q+r)-u-8q\}\Bigl(8\sum_{1\le i<j\le 4}a_i\wedge b_i\wedge a_j\wedge b_j\Bigr)\\
&=-128(6q+r+u)\sum_{1\le i<j\le 4}a_i\wedge b_i\wedge a_j\wedge b_j.
\end{align*}
In the same way, we obtain that 
%\[
%(\pi\circ\widetilde{\Tr}_2)\left(\kushiii\right)=64(2q+u+r-4t)\sum_{1\le %i<j\le 4}a_i\wedge b_i\wedge a_j\wedge b_j.
%\]
\begin{align*}
&\widetilde{\Tr}_2\left(\kushiii\right)\\
&=8\left(\kushiiitwo23+\kushiiitwo24+\kushiiitwo12\right.\\
&\left.+\kushiiitwo34+\kushiiitwo13+\kushiiitwo14\right).
\end{align*}
The map $\pi$ sends it to
\begin{align*}
&8\{-2t+(-q+u)+q-2t+(q+r)+q\}\Bigl(8\sum_{1\le i<j\le 4}a_i\wedge b_i\wedge a_j\wedge b_j\Bigr)\\
&=64(2q+u+r-4t)\sum_{1\le i<j\le 4}a_i\wedge b_i\wedge a_j\wedge b_j.
\end{align*}
Hence, 
\[
(\pi \circ \widetilde{\Tr}_2) (X_S) = 256(16q+3r+3u-2t)\sum_{1\le i<j\le 4}a_i\wedge b_i\wedge a_j\wedge b_j.
\]
Let $C\colon H^{\otimes 4}\to \Q$ be the contraction map defined by
\[
C(x\otimes y\otimes  z\otimes w)=(x\cdot y)(z\cdot w).
\]
Since $C(\sum_{1\le i<j\le 4}a_i\wedge b_i\wedge a_j\wedge b_j)$ is nontrivial, we see that when $p=1$, $q=s=-1$, $u=1$, and $r=t=0$, namely when $\pi = \pi_{1,-1,0,-1,0,1}$ in the notation of the proof of Lemma~\ref{lem:3r=2t}, we obtain a nontrivial $\Sp$-homomorphism 
\[
\h(6) \cap \Ker\widetilde{\Tr}_1 \xrightarrow{\widetilde{\Tr}_2} \widetilde{\Omega}_2(6) \xrightarrow{C \circ \pi} \Q.
\]
Hence, $[0]_\Sp$ is contained in the image of $\widetilde{\Tr}_2: \h(6) \cap \Ker\widetilde{\Tr}_1 \to \widetilde{\Omega}_2(6)$.

Next, we treat $[1^2]_\Sp$ and $[1^4]_\Sp$.
Let $S':=\{(a_1,a_2),(-a_2,a_1)\}\cup \{(a_i,b_i),(-b_i,a_i)\}_{i=3}^5$. 
A similar computation shows that $X_{S'}\in \Ker\widetilde{\Tr}_1$,
and 
\[
(\pi\circ\widetilde{\Tr}_2)(X_{S'})
= 256(16q+3r+3u-2t)\sum_{i=3}^5 a_1\wedge a_2\wedge a_i\wedge b_i.
\]
Let $C'\colon H^{\otimes 4}\to H^{\otimes2}$ be the contraction map defined by
\[
C'(x\otimes y\otimes  z\otimes w)=(x\cdot y)z\otimes w.
\]
When $\pi = \pi_{1,-1,0,-1,0,1}$,
$(C'\circ \pi\circ\widetilde{\Tr}_2)(X_{S'})$ is a nonzero constant multiple of a maximal vector $a_1\wedge a_2$ in $[1^2]_\Sp$.
In the same way, for $S'':=\{(a_1,a_2),(-a_2,a_1)\}\cup\{(a_3,a_4),(-a_4,a_3)\}\cup\{(a_i,b_i),(-b_i,a_i)\}_{i=5}^6$,
we see that $X_{S''}\in \Ker\widetilde{\Tr}_1$, and that the value $(\pi_{1,-1,0,-1,0,1}\circ\widetilde{\Tr}_2)(X_{S''})$ is a nonzero constant multiple of a maximal vector $a_1\wedge a_2\wedge a_3\wedge a_4$ in $[1^4]_\Sp$.
\end{proof}

%%%%
\section{Comparison of two $2$-loop traces} \label{sec:Omega1_Omega_2}

In this section, we review another trace map $\Tr^C$ defined in \cite{Con15} and compare its $2$-loop part with $\widetilde{\Tr}_2$.

\subsection{Another $r$-loop space} \label{subsec:another_r}

Let $\Omega$ be the quotient space of $C_1 \H$ by the following three relations: 
\[ 
\begin{tikzpicture}[scale=0.4, baseline={(0,-0.1)}]
 \draw ([shift={(0,0)}]-150:2) arc [radius=2, start angle = -150, end angle=150];
  \draw [red, very thick, dotted, -<-={.5}{red}]([shift={(0,0)}]140:2) arc [radius=2, start angle = 140, end angle=220];
  \draw (2,0)--(5,0);
\end{tikzpicture}
=0; \tag{C1}
\]
\[
\begin{tikzpicture}[scale=0.5, baseline={(0,3)}]
 \draw ([shift={(0,0)}]60:6) arc [radius=6, start angle = 60, end angle=70];
 \draw [red, very thick, dotted, -<-={.5}{red}]([shift={(0,0)}]70:6) arc [radius=6, start angle = 70, end angle=90];
\draw ([shift={(0,0)}]90:6) arc [radius=6, start angle = 90, end angle=110];
\draw ([shift={(0,0)}]110:6) arc [radius=6, start angle = 110, end angle=120];
\draw (100:6)-- ++(0,1.5)[above=-1mm] node{$v$};
\end{tikzpicture}
=
\begin{tikzpicture}[scale=0.5, baseline={(0,3)}]
 \draw ([shift={(0,0)}]60:6) arc [radius=6, start angle = 60, end angle=70];
 \draw ([shift={(0,0)}]70:6) arc [radius=6, start angle = 70, end angle=90];
\draw [red, very thick, dotted, -<-={.5}{red}]([shift={(0,0)}]90:6) arc [radius=6, start angle = 90, end angle=110];
\draw ([shift={(0,0)}]110:6) arc [radius=6, start angle = 110, end angle=120];
\draw (80:6)-- ++(0,1.5)[above=-1mm] node{$v$};
\end{tikzpicture}
\qquad \text{for any $v \in H$};
\tag{C2}
\]
\vspace{0.5em}
\[
\begin{tikzpicture}[scale=0.3, baseline={(0,-1)}]
\draw (0,0)--(2,-2);
\draw (0,0)--(0,-2);
\draw (0,0)--(-2,-2);
\draw (2,-2)--(4,-4);
\draw [red, very thick, dotted, ->-={.5}{red}] (0,-2)--(0,-4);
\draw [red, very thick, dotted, ->-={.5}{red}] (-2,-2)--(-4,-4);
\end{tikzpicture}
+
\begin{tikzpicture}[scale=0.3, baseline={(0,-1)}]
\draw (0,0)--(2,-2);
\draw (0,0)--(0,-2);
\draw (0,0)--(-2,-2);
\draw [red, very thick, dotted, ->-={.5}{red}] (2,-2)--(4,-4);
\draw (0,-2)--(0,-4);
\draw [red, very thick, dotted, ->-={.5}{red}] (-2,-2)--(-4,-4);
\end{tikzpicture}
+
\begin{tikzpicture}[scale=0.3, baseline={(0,-1)}]
\draw (0,0)--(2,-2);
\draw (0,0)--(0,-2);
\draw (0,0)--(-2,-2);
\draw [red, very thick, dotted, ->-={.5}{red}] (2,-2)--(4,-4);
\draw [red, very thick, dotted, ->-={.5}{red}] (0,-2)--(0,-4);
\draw (-2,-2)--(-4,-4);
\end{tikzpicture}
=0;
\tag{C3}
\]
\vspace{0.5em}

\begin{remark}\label{rem:moveC2'}
Applying the IHX relation and (C2) several times,
we also have a relation in $\Omega$ which moves a rooted $H$-colored tree-shaped Jacobi diagram $T$ as follows.
\[
\begin{tikzpicture}[scale=0.5, baseline={(0,3)}]
 \draw ([shift={(0,0)}]60:6) arc [radius=6, start angle = 60, end angle=70];
 \draw [red, very thick, dotted, -<-={.5}{red}]([shift={(0,0)}]70:6) arc [radius=6, start angle = 70, end angle=90];
\draw ([shift={(0,0)}]90:6) arc [radius=6, start angle = 90, end angle=110];
\draw ([shift={(0,0)}]110:6) arc [radius=6, start angle = 110, end angle=120];
\draw (100:6)-- ++(0,1.5)[above=-1mm] node{$T$};
\end{tikzpicture}
=
\begin{tikzpicture}[scale=0.5, baseline={(0,3)}]
 \draw ([shift={(0,0)}]60:6) arc [radius=6, start angle = 60, end angle=70];
 \draw ([shift={(0,0)}]70:6) arc [radius=6, start angle = 70, end angle=90];
\draw [red, very thick, dotted, -<-={.5}{red}]([shift={(0,0)}]90:6) arc [radius=6, start angle = 90, end angle=110];
\draw ([shift={(0,0)}]110:6) arc [radius=6, start angle = 110, end angle=120];
\draw (80:6)-- ++(0,1.5)[above=-1mm] node{$T$};
\end{tikzpicture}.
\]
\end{remark}

Then the space $\Omega$ decomposes as 
\[
\Omega = \bigoplus_{n \ge 0} \Omega(n) = 
\bigoplus_{n \ge 0} \bigoplus_{r \ge 0} \Omega_r(n),
\]
where $n$ and $r$ stand for the Lie degree and the number of dotted edges, respectively.
We also use the notation $\Omega_{r,m}= \Omega_r(m+2r-2)$ and $\Omega_{r}=\bigoplus_{m=0}^{\infty} \Omega_{r,m}$, where $m$ denotes the homological degree.
It was shown by Conant that the Conant-Kassabov-Vogtmann trace induces the map $\Tr^{\rm C}: \h \to \Omega$ which vanishes on $\m$.

\begin{proposition}\label{prop:refine}
There is a natural map $\widetilde{\Omega} \to \Omega$ such that the following diagram commutes. 
\[
\xymatrix@R=0.5em@C=4em{
 & \widetilde{\Omega} \ar[dd] \\
\h \ar[ur]^{\widetilde{\Tr}} \ar[dr]_{\Tr^{\rm C}} & \\
 & \Omega
}
\]
\end{proposition}

\begin{proof}
As in the proof of Theorem~\ref{thm:omega2}, let $X\in C_n\H^{\ord}(n)$ be a hairy Lie graph such that $\beta^{n-1}(X)\ne 0$ in $C_1 \H(n)$. 
We show that $\beta^{n-1}(X)$ vanishes when it is sent to $\Omega(n)$.
Then, we have a natural surjective homomorphism $\widetilde{\Omega}(n)= C_1 \H(n)/ \beta^{n-1}(C_n \H^{\ord}(n))\to\Omega(n)$ which makes the diagram commute.

We use the terms introduced in the proof of Theorem~\ref{thm:omega2}.
Let $Y_i$ be the last tripod in $X$ satisfying $N(Y_i)\ge1$.
When $i=1$, $Y_1$ has a self-loop edge and the other tripods have no extra dotted edges.
Hence, $\beta^{n-1}(X)$ is of the form in the left hand side of (C1), and it vanishes in $\Omega(n)$.
Hence, $\beta^{n-1}(X)$ vanishes when it is sent to $\Omega(n)$.
When $i\ge 2$, there are three cases concerning the extra dotted edges of $Y_i$:
\begin{enumerate}
\item[(i)] $N(Y_i)=2$ and $Y_i$ has no self-loop edge,
\item[(ii)] $N(Y_i)=1$ and $Y_i$ has one self-loop edge,
\item[(iii)] $N(Y_i)=1$ and $Y_i$ has no self-loop edge.
\end{enumerate}
In each case, we analyze the result of applying $\beta$ to $\beta^{i-2}(X)$. 

In the case (i), the three backward edges of $Y_i$ are connected to the preceding tripods.
If we apply $\beta$ to each term of $\beta^{i-2}(X)$,
the result is of the form in the left hand side of (C3).
In the case (ii), $\beta^{i-2}(X)$ consists of hairy Lie graphs with one self-loop edge, and all terms in $\beta^{i-1}(X)$ are of the form in the left hand side of (C1).
In the case (iii), two univalent vertices of $Y_i$ are connected to the preceding tripods by the backward edges of $Y_i$.
Assume that the remaining univalent vertex of $Y_i$ is not incident to a dotted edge, i.e., it is colored with an element of $H$.
Applying $\beta$ to each hairy Lie graph in $\beta^{i-2}(X)$, we obtain an element of the form of the difference between the two sides of (C2). 
Finally, in the case (iii), assume that the remaining univalent vertex of $Y_i$ is incident to a dotted edge, which is connected to one of the succeeding tripods.
Since $Y_i$ is the last tripod which has extra dotted edges,
this dotted edge turns into a solid edge in $\beta^{n-1}(X)$.
The complement of the solid edge consists of two connected components,
and the component containing the succeeding tripods is a rooted $H$-colored tree-shaped Jacobi diagram.
By Remark~\ref{rem:moveC2'}, $\beta^{n-1}(X)$ vanishes in $\Omega(n)$. 
\end{proof}

The $1$-loop parts $\Omega_1$ and $\widetilde{\Omega}_1$ are the same and determined by the following proposition.

\begin{proposition}[{\cite[Theorem 4.1]{Con15}, \cite[Theorem 1.1]{Con16}}] \label{prop:Omega1}
\[
\Omega_{1,m} \cong \widetilde{\Omega}_{1,m} \cong H^{\otimes m}/D_{2m}.
\]
\end{proposition}

Under this isomorphism, the pure tensor $x_1\otimes \cdots \otimes x_m$ corresponds to the following hairy Lie graph:
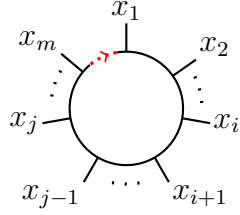
\[
\begin{tikzpicture}[scale=0.25, baseline={(0,0.1)}]
 \draw ([shift={(0,0)}]100:3) arc [radius=3, start angle = 100, end angle=-230];
 \draw [red, very thick, dotted, -<-={.5}{red}] ([shift={(0,0)}]100:3) arc [radius=3, start angle = 100, end angle=130];
 \draw (90:3)--(90:4.5) node[above=-3pt]{$x_1$};
 \draw (35:3)--(35:4.5) node[above right=-3pt]{$x_2$};
  \draw[loosely dotted, very thick] (25:4) arc (25:0:4);
  \draw[loosely dotted, very thick] (150:4) arc (150:175:4);
 \draw (-10:3)--(-10:4.5) node[right=-3pt]{$x_i$};
 \draw (-60:3)--(-60:4.5) node[below right=-3pt]{$x_{i+1}$};
 \draw[loosely dotted, very thick] (-80:4) arc (-80:-105:4);
 \draw (240:3)--(240:4.5) node[below left=-3pt]{$x_{j-1}$};
 \draw (190:3)--(190:4.5) node[left=-3pt]{$x_j$};
 \draw (140:3)--(140:4.5) node[above left=-3pt] {$x_m$};
\end{tikzpicture}
\]

\subsection{Factorization of $\Tr^C_2$}

The map $\Tr^C$ decomposes into components:
$\Tr^C = \bigoplus_{r=1}^{\infty} \Tr_r^C: \h \to \Omega_r$. 
We show that $\Tr^{\rm C}_2$ factors through $\Tr^C_1$.

Let $\Phi\colon \Omega_1 \to \Omega_2$ be a homomorphism defined by 
\[
\begin{tikzpicture}[scale=0.25, baseline={(0,0.1)}]
 \draw ([shift={(0,0)}]100:3) arc [radius=3, start angle = 100, end angle=-230];
 \draw [red, very thick, dotted, -<-={.5}{red}] ([shift={(0,0)}]100:3) arc [radius=3, start angle = 100, end angle=130];
 \draw (90:3)--(90:4.5) node[above=-3pt]{$x_1$};
 \draw (35:3)--(35:4.5) node[above right=-3pt]{$x_2$};
  \draw[loosely dotted, very thick] (25:4) arc (25:0:4);
  \draw[loosely dotted, very thick] (150:4) arc (150:175:4);
 \draw (-10:3)--(-10:4.5) node[right=-3pt]{$x_i$};
 \draw (-60:3)--(-60:4.5) node[below right=-3pt]{$x_{i+1}$};
 \draw[loosely dotted, very thick] (-80:4) arc (-80:-105:4);
 \draw (240:3)--(240:4.5) node[below left=-3pt]{$x_{j-1}$};
 \draw (190:3)--(190:4.5) node[left=-3pt]{$x_j$};
 \draw (140:3)--(140:4.5) node[above left=-3pt] {$x_n$};
\end{tikzpicture}
\mapsto
\sum_{i<j}
(x_i\cdot x_j)
\left\{
\begin{tikzpicture}[scale=0.25, baseline={(0,0)}]
 \draw ([shift={(0,0)}]100:3) arc [radius=3, start angle = 100, end angle=-230];
 \draw [red, very thick, dotted, -<-={.5}{red}] ([shift={(0,0)}]100:3) arc [radius=3, start angle = 100, end angle=130];
 \draw (90:3)--(90:4.5) node[above=-3pt]{$x_1$};
 \draw (35:3)--(35:4.5) node[above right=-3pt]{$x_2$};
  \draw[loosely dotted, very thick] (25:4) arc (25:0:4);
  \draw[loosely dotted, very thick] (150:4) arc (150:175:4);
 \draw (-10:3)-- node(A){} ++ (-1.5,0) ;
 \draw (-60:3)--(-60:4.5) node[below right=-3pt]{$x_{i+1}$};
 \draw[loosely dotted, very thick] (-80:4) arc (-80:-105:4);
 \draw (240:3)--(240:4.5) node[below left=-3pt]{$x_{j-1}$};
 \draw (190:3)-- node(B){} ++ (1.5,0) ;
\draw [red, very thick, dotted, ->-={.5}{red}] (A)--(B);
 \draw (140:3)--(140:4.5) node[above left=-3pt]{$x_n$};
\end{tikzpicture}
+
\begin{tikzpicture}[scale=0.25, baseline={(0,0)}]
 \draw ([shift={(0,0)}]-20:3) arc [radius=3, start angle = -20, end angle=310];
  \draw [red, very thick, dotted, -<-={.5}{red}] ([shift={(0,0)}]310:3) arc [radius=3, start angle = 310, end angle=340];
 \draw (90:3)--(90:4.5) node[above=-3pt]{$x_1$};
 \draw (35:3)--(35:4.5) node[above right=-3pt]{$x_2$};
  \draw[loosely dotted, very thick] (25:4) arc (25:0:4);
  \draw[loosely dotted, very thick] (150:4) arc (150:175:4);
 \draw (-10:3)-- node(A){} ++ (-1.5,0) ;
 \draw (-60:3)--(-60:4.5) node[below right=-3pt]{$x_{i+1}$};
 \draw[loosely dotted, very thick] (-80:4) arc (-80:-105:4);
 \draw (240:3)--(240:4.5) node[below left=-3pt]{$x_{j-1}$};
 \draw (190:3)-- node(B){} ++ (1.5,0) ;
\draw [red, very thick, dotted, ->-={.5}{red}] (A)--(B);
 \draw (140:3)--(140:4.5) node[above left=-3pt]{$x_n$};
\end{tikzpicture}
\right\}.
\] 

\begin{remark} \label{rem:aboutPhi}
The definition of the map $\Phi$ makes use of 
the description of $\Omega_1$ in Proposition~\ref{prop:Omega1}.
The first term is obtained by just adding a dotted edge between the $i$th and $j$th leaves.
If the core of the hairy Lie graph thus obtained is a theta graph, then the second term is obtained from the first term by sliding the initial dotted edge across the newly added edges.
If the core is a dumbbell graph, there will be more than one term.
For example, let $T_0, T_1, T_2$ be rooted $H$-colored tree-shaped Jacobi diagrams such that the coloring of any of their leaves has trivial intersections with $a_1$, $b_1$ and the colorings of any other leaves.
Then, if 
\[
X = \begin{tikzpicture}[baseline=0pt, x=1.4mm, y=1.4mm]
\draw (0,-3) -- (0,3);
\draw (0,3) to[bend right=30] (-1,4) -- (-7,4) to[bend right=30] (-8,3) -- (-8,2);
\draw (0,-3) to[bend left=30] (-1,-4) -- (-7,-4) to[bend left=30] (-8,-3) -- (-8,-2);
\draw[red, very thick, dotted, ->-] (-8,-2) -- (-8,2);
\draw (0,0) -- (5,0);
\draw (5,0) -- (7,3) -- (10,3) node[right=-3pt]{$a_1$};
\draw (7,3) -- (7,5) node[above=-3pt]{$T_1$};
\draw (5,0) -- (7,-3) -- (10,-3) node[right=-3pt]{$b_1$};
\draw (7,-3) -- (7,-5) node[below=-3pt]{$T_2$};
\draw (-4,4) -- (-4,6) node[above=-3pt]{$T_0$};
\end{tikzpicture},
\]
then we have
\[
\Phi(X) = \, 
\begin{tikzpicture}[baseline=0pt, x=1.4mm, y=1.4mm]
\draw (0,-3) -- (0,3);
\draw (0,3) to[bend right=30] (-1,4) -- (-7,4) to[bend right=30] (-8,3) -- (-8,2);
\draw (0,-3) to[bend left=30] (-1,-4) -- (-7,-4) to[bend left=30] (-8,-3) -- (-8,-2);
\draw[red, very thick, dotted, ->-] (-8,-2) -- (-8,2);
\draw (0,0) -- (5,0);
\draw (5,0) -- (7,3) -- (10,3);
\draw (7,3) -- (7,5) node[above=-3pt]{$T_1$};
\draw (5,0) -- (7,-3) -- (10,-3);
\draw (7,-3) -- (7,-5) node[below=-3pt]{$T_2$};
\draw (-4,4) -- (-4,6) node[above=-3pt]{$T_0$};
\draw[red, very thick, dotted, ->-] (10,3) to[bend left=80] (10,-3); 
\end{tikzpicture} 
\ + \ 
\begin{tikzpicture}[baseline=0pt, x=1.4mm, y=1.4mm]
\draw[red, very thick, dotted, -<-] (0,-2) -- (0,2);
\draw (0,2) -- (0,3) to[bend right=30] (-1,4) -- (-7,4) to[bend right=30] (-8,3) -- (-8,2);
\draw (0,-2) -- (0,-3) to[bend left=30] (-1,-4) -- (-7,-4) to[bend left=30] (-8,-3) -- (-8,-2);
\draw (-8,-2) -- (-8,2);
\draw (0,3) -- (5,3);
\draw (2,3) -- (2,5) node[above=-3pt]{$T_1$};
\draw (0,-3) -- (5,-3);
\draw (2,-3) -- (2,-5) node[below=-3pt]{$T_2$};
\draw (-4,4) -- (-4,6) node[above=-3pt]{$T_0$};
\draw[red, very thick, dotted, ->-] (5,3) to[bend left=80] (5,-3); 
\end{tikzpicture}
\ - \ 
\begin{tikzpicture}[baseline=0pt, x=1.4mm, y=1.4mm]
\draw[red, very thick, dotted, -<-] (0,-2) -- (0,2);
\draw (0,2) -- (0,3) to[bend right=30] (-1,4) -- (-7,4) to[bend right=30] (-8,3) -- (-8,2);
\draw (0,-2) -- (0,-3) to[bend left=30] (-1,-4) -- (-7,-4) to[bend left=30] (-8,-3) -- (-8,-2);
\draw (-8,-2) -- (-8,2);
\draw (0,3) -- (5,3);
\draw (2,3) -- (2,5) node[above=-3pt]{$T_2$};
\draw (0,-3) -- (5,-3);
\draw (2,-3) -- (2,-5) node[below=-3pt]{$T_1$};
\draw (-4,4) -- (-4,6) node[above=-3pt]{$T_0$};
\draw[red, very thick, dotted, -<-] (5,3) to[bend left=80] (5,-3); 
\end{tikzpicture} \ . 
\]
\end{remark}

\begin{proposition}\label{prop:factorthrough}
\[
\Phi\circ \Tr^C_1=3\, \Tr^C_2\colon \h \to \Omega_2.
\]
\end{proposition}

\begin{proof}
Let $X$ be an $H$-colored tree-shaped Jacobi diagram.
On the one hand, $\Tr^C_2(X)$ is a summation over all unordered ways of adding two dotted edges to $X$.
On the other hand, $\Tr^C_1(X)$ is a summation over all ways of adding one dotted edge to $X$. 
Applying $\Phi$ to each term of $\Tr^C_1(X)$ is a summation over all choices of a pair of its leaves.
Thus, $(\Phi \circ \Tr_1^C)(X)$ is a summation over all ordered pairs of disjoint 2-element subsets of the leaves of $X$, where the first pair is consumed by $\Tr^C_1$ and the second pair by $\Phi$.
We will show that, for each unordered pair $\{ S,S'\}$ of disjoint 2-element subsets of the leaves of $X$, three times its contribution to $\Tr^C_2(X)$ is equal to the sum of contributions to $(\Phi \circ \Tr^C_1)(X)$ from the two ordered pairs $(S,S')$ and $(S',S)$.
For this purpose, we view the colorings of the leaves of $X$ as indeterminates and proceed as follows.

{\em Step 1}.
Fix an integer $n \ge 2$.
Let $\{ w_p\}_{p=1}^{n+2}$ be a set of formal symbols and set $W = \bigoplus_{p=1}^{n+2} \Q w_p$. 
By a $w$-colored hairy Lie graph, we mean a hairy Lie graph whose leaves are colored with distinct elements of $\{ w_p \}_{p=1}^{n+2}$.  
All the hairy Lie graphs that we consider for the moment have one tree and Lie degree $n$. 
We introduce three spaces $\h^F(n)$, $\Omega^F_{1,n}$ and $\Omega^F_{2,n-2}$ as follows.
\begin{itemize}
    \item $\h^F(n)$ is the space spanned by $w$-colored hairy Lie graphs with no dotted edge, modulo the multi-linearity, IHX and AS relations.
    \item $\Omega^F_{1,n}$ is spanned by elements of the form $(w_i \wedge w_j) \otimes X$, where $X$ is a $w$-colored hairy Lie graph with one dotted edge such that its set of colorings is $\{ w_p \}_{p=1}^{n+2} \setminus \{ w_i, w_j\}$.
    Here, the first component of the expression is understood to be an element in $\Lambda^2 W$, and the second component is considered modulo the same relations as in $\Omega_{1,n}$: the multi-linearity, IHX and AS relations, as well as the relations (C1), (C2), and (C3).
    \item $\Omega^F_{2,n-2}$ is defined similarly to $\Omega^F_{1,n}$:
    it is spanned by elements of the form $\big((w_i \wedge w_j)\cdot (w_k \wedge w_l)\big) \otimes X$, where $X$ is a $w$-colored hairy Lie graph with two dotted edges such that its set of colorings is $\{ w_p \}_{p=1}^{n+2} \setminus \{ w_i, w_j, w_k, w_l\}$. 
    Here, the first component is understood as an element in the second symmetric power of $\Lambda^2 W$.
\end{itemize}
We can define the trace maps $\Tr^C_1: \h^F(n) \to \Omega^F_{1,n}$ and $\Tr^C_2: \h^F(n) \to \Omega^F_{2,n-2}$ and the map $\Phi: \Omega^F_{1,n} \to \Omega^F_{2,n-2}$ in exactly the same manner as their original versions, where one uses the wedge product $w_i \wedge w_j$ instead of the intersection pairing. 

{\em Step 2}.
We claim that 
\[
\Phi \circ \Tr^C_1 = 3\, \Tr^C_2: 
\h^F(n) \to \Omega^F_{2,n-2}.
\]

\begin{proof}[Proof of the claim]
For any unordered pair $\{ (w_i,w_j), (w_k,w_l) \}$ of disjoint 2-element subsets of $\{w_p\}_{p=1}^{n+2}$ and for any $w$-colored tree-shaped Jacobi diagram $X$, it is enough to
compare the terms in $\Tr^C_2(X)$ which are of the form $\big((w_i \wedge w_j) \cdot (w_k \wedge w_l)\big) \otimes Y$ for some $Y$ with those in $(\Phi\circ \Tr^C_1)(X)$.
Since $\h^F(n)$ is generated by elements of the form 
\[
\begin{tikzpicture}[scale=0.7, baseline={(0,0.5)}]
\draw (0,0) node[left] {$w_i$} -- (5,0) node[right] {$w_j$};
\draw (1,0)--(1,1);
\node at (1.6,0.8) {$\cdots$};
\draw (2,0)--(2,1) node[above=-2pt]{$w_k$};
\node at (2.6,0.8) {$\cdots$};
\draw (3,0) -- (3,1) node[above=-2pt]{$w_l$};
\node at (3.6,0.8) {$\cdots$};
\draw (4,0)--(4,1); 
\end{tikzpicture}
\quad \text{and} \quad 
\begin{tikzpicture}[scale=0.7, baseline={(0,0.5)}]
\draw (0,0) node[left] {$w_i$} -- (5,0) node[right] {$w_j$};
\draw (1,0)--(1,1);
\node at (1.6,0.8) {$\cdots$};
\draw (2,0)--(2,1) node[above=-2pt]{$w_l$};
\node at (2.6,0.8) {$\cdots$};
\draw (3,0) -- (3,1) node[above=-2pt]{$w_k$};
\node at (3.6,0.8) {$\cdots$};
\draw (4,0)--(4,1); 
\end{tikzpicture} \ , 
\]
we may assume that $X$ is one of these elements. 
Furthermore, by exchanging the role of $k$ and $l$, it is sufficient to consider the first one only.
In this case, $\Tr^C_1(X)$ is equal to 
\[
(w_i \wedge w_j) \otimes \, 
\begin{tikzpicture}[scale=0.7, baseline=5pt]
\draw (0.5,0) -- (4.5,0);
\draw (1,0)--(1,0.8);
\node at (1.6,0.6) {$\cdots$};
\draw (2,0)--(2,0.8) node[above=-2pt]{$w_k$};
\node at (2.6,0.6) {$\cdots$};
\draw (3,0) -- (3,0.8) node[above=-2pt]{$w_l$};
\node at (3.6,0.6) {$\cdots$};
\draw (4,0)--(4,0.8); 
%%% dotted edge %%%
\draw[red, very thick, dotted, ->-] (0.5,0) to[bend left=30] (0.25,0.25) -- (0.25,1.25) to[bend left=30] (0.5,1.5) -- (4.5,1.5) to[bend left=30] (4.75,1.25) -- (4.75,0.25) to[bend left=30] (4.5,0); 
\end{tikzpicture}
\ + \ 
(w_k \wedge w_l) \otimes 
\begin{tikzpicture}[scale=0.7, baseline=5pt]
\draw (0.5,0) node[left=-2pt]{$w_i$} -- (4.5,0) node[right=-2pt]{$w_j$};
\draw (1,0)--(1,0.8);
\node at (1.6,0.6) {$\cdots$};
\draw (2,0)--(2,0.8);
\node at (2.6,0.6) {$\cdots$};
\draw (3,0) -- (3,0.8);
\node at (3.6,0.6) {$\cdots$};
\draw (4,0)--(4,0.8); 
%%% dotted edge %%%
\draw[red, very thick, dotted, ->-] (2,0.8) to[bend left=80] (3,0.8);
\end{tikzpicture}
+ \cdots,  
\]
and $(\Phi \circ \Tr^C_1)(X) = \big( (w_i \wedge w_j)\cdot (w_k \wedge w_l)\big) \otimes Z + \cdots$, where  
\begin{align*}
Z &= \begin{tikzpicture}[scale=0.7, baseline=8pt]
\draw (0.5,0) -- (4.5,0);
\draw (1,0)--(1,0.8);
\node at (1.6,0.6) {$\cdots$};
\draw (2,0)--(2,0.8);
\node at (2.6,0.6) {$\cdots$};
\draw (3,0) -- (3,0.8);
\node at (3.6,0.6) {$\cdots$};
\draw (4,0)--(4,0.8); 
%%% dotted edge %%%
\draw[red, very thick, dotted, ->-] (0.5,0) to[bend left=30] (0.25,0.25) -- (0.25,1.25) to[bend left=30] (0.5,1.5) -- (4.5,1.5) to[bend left=30] (4.75,1.25) -- (4.75,0.25) to[bend left=30] (4.5,0); 
\draw[red, very thick, dotted, ->-] (2,0.8) to[bend left=80] (3,0.8);
\end{tikzpicture}
\ + \ 
\begin{tikzpicture}[scale=0.7, baseline=8pt]
\draw (0.5,0) -- (2.2,0);
\draw (2.8,0) -- (4.5,0);
\draw (1,0)--(1,0.8);
\node at (1.6,0.6) {$\cdots$};
\draw (2,0)--(2,0.8);
\node at (2.6,0.6) {$\cdots$};
\draw (3,0) -- (3,0.8);
\node at (3.6,0.6) {$\cdots$};
\draw (4,0)--(4,0.8); 
%%% dotted edge %%%
\draw (0.5,0) to[bend left=30] (0.25,0.25) -- (0.25,1.25) to[bend left=30] (0.5,1.5) -- (4.5,1.5) to[bend left=30] (4.75,1.25) -- (4.75,0.25) to[bend left=30] (4.5,0); 
\draw[red, very thick, dotted, ->-] (2.8,0) -- (2.2,0);
\draw[red, very thick, dotted, ->-] (2,0.8) to[bend left=80] (3,0.8);
\end{tikzpicture} \\
& \quad + 
\begin{tikzpicture}[scale=0.7, baseline=8pt]
\draw (0.5,0) -- (4.5,0);
\draw (1,0)--(1,0.8);
\node at (1.6,0.6) {$\cdots$};
\draw (2,0)--(2,0.8);
\node at (2.6,0.6) {$\cdots$};
\draw (3,0) -- (3,0.8);
\node at (3.6,0.6) {$\cdots$};
\draw (4,0)--(4,0.8); 
%%% dotted edge %%%
\draw[red, very thick, dotted, ->-] (0.5,0) to[bend left=30] (0.25,0.25) -- (0.25,1.25) to[bend left=30] (0.5,1.5) -- (4.5,1.5) to[bend left=30] (4.75,1.25) -- (4.75,0.25) to[bend left=30] (4.5,0); 
\draw[red, very thick, dotted, ->-] (2,0.8) to[bend left=80] (3,0.8);
\end{tikzpicture}
\ + \ 
\begin{tikzpicture}[scale=0.7, baseline=8pt]
\draw (0.5,0) -- (2.2,0);
\draw (2.8,0) -- (4.5,0);
\draw (1,0)--(1,0.8);
\node at (1.6,0.6) {$\cdots$};
\draw (2,0)--(2,0.8);
\node at (2.6,0.6) {$\cdots$};
\draw (3,0) -- (3,0.8);
\node at (3.6,0.6) {$\cdots$};
\draw (4,0)--(4,0.8); 
%%% dotted edge %%%
\draw[red, very thick, dotted, ->-] (0.5,0) to[bend left=30] (0.25,0.25) -- (0.25,1.25) to[bend left=30] (0.5,1.5) -- (4.5,1.5) to[bend left=30] (4.75,1.25) -- (4.75,0.25) to[bend left=30] (4.5,0); 
\draw[red, very thick, dotted, ->-] (2.8,0) -- (2.2,0);
\draw (2,0.8) to[bend left=80] (3,0.8);
\end{tikzpicture} \ . 
\end{align*}
By (C3), the sum of the second and fourth terms is equal to the first term. 
Hence 
\[
Z = 3 \ 
\begin{tikzpicture}[scale=0.7, baseline=8pt]
\draw (0.5,0) -- (4.5,0);
\draw (1,0)--(1,0.8);
\node at (1.6,0.6) {$\cdots$};
\draw (2,0)--(2,0.8);
\node at (2.6,0.6) {$\cdots$};
\draw (3,0) -- (3,0.8);
\node at (3.6,0.6) {$\cdots$};
\draw (4,0)--(4,0.8); 
%%% dotted edge %%%
\draw[red, very thick, dotted, ->-] (0.5,0) to[bend left=30] (0.25,0.25) -- (0.25,1.25) to[bend left=30] (0.5,1.5) -- (4.5,1.5) to[bend left=30] (4.75,1.25) -- (4.75,0.25) to[bend left=30] (4.5,0); 
\draw[red, very thick, dotted, ->-] (2,0.8) to[bend left=80] (3,0.8);
\end{tikzpicture} \ .
\]
Divided by three, it is exactly the same as the corresponding term in $\Tr^C_2(X)$.
This proves the claim. 
\end{proof}

{\em Step 3}. 
Finally, we deduce the assertion of the proposition from the claim.
Given any homology classes $x_1,\ldots,x_{n+2}$, one can define natural evaluation maps $\h^F(n) \to \h(n)$, $\Omega^F_{1,n} \to \Omega_{1,n}$, and $\Omega^F_{2,n-2} \to \Omega_{2,n-2}$, by replacing $w_p$ with $x_p$ and $w_i \wedge w_j$ with $(x_i \cdot x_j)$.
Any element in $\h(n)$ is a linear combination of elements obtained from the first map for various choices of $x_1, \ldots, x_{n+2}$. 
These evaluation maps intertwine the two versions of the trace maps $\Tr^C_1$, $\Tr^C_2$ and $\Phi$, one defined on $\h^F(n)$ and $\Omega^F_{1,n}$, and the other defined on $\h(n)$ and $\Omega_{1,n}$. 
Thus, the validity of the claim implies the equality $\Phi \circ \Tr^C_1 = 3\, \Tr^C_2$ on $\h(n)$. 
\end{proof}

\begin{remark}
We do not know whether $\Tr^C_r$ for $r\ge3$ factors through $\Tr^C_1$.
\end{remark}

\def\cprime{$'$}


\begin{thebibliography}{10}

\bibitem{AKKN_hg}
A.~Alekseev, N.~Kawazumi, Y.~Kuno, and F.~Naef.
\newblock The {G}oldman-{T}uraev {L}ie bialgebra and the {K}ashiwara-{V}ergne problem in higher genera.
\newblock arXiv:1804.09566v3, 2023.

\bibitem{Asa96}
M.~Asada.
\newblock On the filtration of topological and pro-{$l$} mapping class groups of punctured {R}iemann surfaces.
\newblock {\em J. Math. Soc. Japan}, 48(1):13--36, 1996.

\bibitem{AN95}
M.~Asada and H.~Nakamura.
\newblock On graded quotient modules of mapping class groups of surfaces.
\newblock {\em Israel J. Math.}, 90(1-3):93--113, 1995.


\bibitem{Con15}
J.~Conant.
\newblock The {J}ohnson cokernel and the {E}nomoto-{S}atoh invariant.
\newblock {\em Algebr. Geom. Topol.}, 15(2):801--821, 2015.

\bibitem{Con16}
J.~Conant.
\newblock {A}ddendum to ``{T}he {J}ohnson cokernel and the {E}nomoto-{S}atoh invariant'': the {ES}-trace detects all top-level partitions.
\newblock arXiv:1610.05220, 2016.

\bibitem{Con17}
J.~Conant.
\newblock The {$\textsf {Lie}$} {L}ie algebra.
\newblock {\em Quantum Topol.}, 8(4):667--714, 2017.

\bibitem{CK16}
J.~Conant and M.~Kassabov.
\newblock Hopf algebras and invariants of the {J}ohnson cokernel.
\newblock {\em Algebr. Geom. Topol.}, 16(4):2325--2363, 2016.

\bibitem{CKV13}
J.~Conant, M.~Kassabov, and K.~Vogtmann.
\newblock Hairy graphs and the unstable homology of {${\rm Mod}(g,s)$}, {${\rm Out}(F_n)$} and {${\rm Aut}(F_n)$}.
\newblock {\em J. Topol.}, 6(1):119--153, 2013.

\bibitem{CKV15}
J.~Conant, M.~Kassabov, and K.~Vogtmann.
\newblock Higher hairy graph homology.
\newblock {\em Geom. Dedicata}, 176:345--374, 2015.

\bibitem{ES14}
N.~Enomoto and T.~Satoh.
\newblock New series in the {J}ohnson cokernels of the mapping class groups of surfaces.
\newblock {\em Algebr. Geom. Topol.}, 14(2):627--669, 2014.

\bibitem{FNW23}
M.~Felder, F.~Naef, and T.~Willwacher.
\newblock Stable cohomology of graph complexes.
\newblock {\em Selecta Math. (N.S.)}, 29(2):Paper No. 23, 72, 2023.

\bibitem{FH91}
W.~Fulton and J.~Harris.
\newblock {\em Representation theory}, volume 129 of {\em Graduate Texts in Mathematics}.
\newblock Springer-Verlag, New York, 1991.
\newblock A first course, Readings in Mathematics.

\bibitem{GaGe17}
S.~Garoufalidis and E.~Getzler.
\newblock Graph complexes and the symplectic character of the {T}orelli group.
\newblock arXiv:1712.03606, 2017.

\bibitem{GL05}
S.~Garoufalidis and J.~Levine.
\newblock Tree-level invariants of three-manifolds, {M}assey products and the {J}ohnson homomorphism.
\newblock In {\em Graphs and patterns in mathematics and theoretical physics}, volume~73 of {\em Proc. Sympos. Pure Math.}, pages 173--203. Amer. Math. Soc., Providence, RI, 2005.


\bibitem{Hai97}
R.~Hain.
\newblock Infinitesimal presentations of the {T}orelli groups.
\newblock {\em J. Amer. Math. Soc.}, 10(3):597--651, 1997.

\bibitem{Ha20}
R.~Hain.
\newblock Johnson homomorphisms.
\newblock {\em EMS Surv. Math. Sci.}, 7(1):33--116, 2020.

\bibitem{Joh83I}
D.~Johnson.
\newblock The structure of the {T}orelli group. {I}. {A} finite set of generators for {${\mathcal I}$}.
\newblock {\em Ann. of Math. (2)}, 118(3):423--442, 1983.

\bibitem{Joh83sur}
D.~Johnson.
\newblock A survey of the {T}orelli group.
\newblock In {\em Low-dimensional topology ({S}an {F}rancisco, {C}alif., 1981)}, volume~20 of {\em Contemp. Math.}, pages 165--179. Amer. Math. Soc., Providence, RI, 1983.

\bibitem{KK15AIF}
N.~Kawazumi and Y.~Kuno.
\newblock Intersection of curves on surfaces and their applications to mapping class groups.
\newblock {\em Ann. Inst. Fourier (Grenoble)}, 65(6):2711--2762, 2015.

\bibitem{KK16}
N.~Kawazumi and Y.~Kuno.
\newblock The {G}oldman-{T}uraev {L}ie bialgebra and the {J}ohnson homomorphisms.
\newblock In {\em Handbook of {T}eichm\"uller theory. {V}ol. {V}}, volume~26 of {\em IRMA Lect. Math. Theor. Phys.}, pages 97--165. Eur. Math. Soc., Z\"urich, 2016.

\bibitem{KuRW20}
A.~Kupers and O.~Randal-Williams.
\newblock On the cohomology of {T}orelli groups.
\newblock {\em Forum Math. Pi}, 8:e7, 83, 2020.

\bibitem{MaSa20}
G.~Massuyeau and T.~Sakasai.
\newblock Morita's trace maps on the group of homology cobordisms.
\newblock {\em J. Topol. Anal.}, 12(3):775--818, 2020.

\bibitem{Mat96}
M.~Matsumoto.
\newblock Galois representations on profinite braid groups on curves.
\newblock {\em J. Reine Angew. Math.}, 474:169--219, 1996.

\bibitem{Mat13}
M.~Matsumoto.
\newblock Introduction to arithmetic mapping class groups.
\newblock In {\em Moduli spaces of {R}iemann surfaces}, volume~20 of {\em IAS/Park City Math. Ser.}, pages 319--356. Amer. Math. Soc., Providence, RI, 2013.

\bibitem{Mor89}
S.~Morita.
\newblock Casson's invariant for homology {$3$}-spheres and characteristic classes of surface bundles. {I}.
\newblock {\em Topology}, 28(3):305--323, 1989.

\bibitem{Mor93}
S.~Morita.
\newblock Abelian quotients of subgroups of the mapping class group of surfaces.
\newblock {\em Duke Math. J.}, 70(3):699--726, 1993.

\bibitem{Mor99}
S.~Morita.
\newblock Structure of the mapping class groups of surfaces: a survey and a prospect.
\newblock In {\em Proceedings of the {K}irbyfest ({B}erkeley, {CA}, 1998)}, volume~2 of {\em Geom. Topol. Monogr.}, pages 349--406. Geom. Topol. Publ., Coventry, 1999.

\bibitem{MSS15AM}
S.~Morita, T.~Sakasai, and M.~Suzuki.
\newblock Structure of symplectic invariant {L}ie subalgebras of symplectic derivation {L}ie algebras.
\newblock {\em Adv. Math.}, 282:291--334, 2015.

\bibitem{Nak96}
H.~Nakamura.
\newblock Coupling of universal monodromy representations of {G}alois-{T}eichm\"uller modular groups.
\newblock {\em Math. Ann.}, 304(1):99--119, 1996.

\bibitem{NSS22JT}
Y.~Nozaki, M.~Sato, and M.~Suzuki.
\newblock On the kernel of the surgery map restricted to the 1-loop part.
\newblock {\em J. Topol.}, 15(2):587--619, 2022.

\bibitem{NSS23TAMS}
Y.~Nozaki, M.~Sato, and M.~Suzuki.
\newblock A non-commutative {R}eidemeister-{T}uraev torsion of homology cylinders.
\newblock {\em Trans. Amer. Math. Soc.}, 376(7):5045--5088, 2023.

\bibitem{Sak17}
T.~Sakasai.
\newblock Johnson homomorphisms and symplectic representation theory.
\newblock Workshop: Johnson homomorphisms and related topics, \url{https://www.ms.u-tokyo.ac.jp/~sakasai/workshop/Johnson2017/Sakasai.pdf}, 2017.

\bibitem{Sat16}
T.~Satoh.
\newblock A survey of the {J}ohnson homomorphisms of the automorphism groups of free groups and related topics.
\newblock In {\em Handbook of {T}eichm\"uller theory. {V}ol. {V}}, volume~26 of {\em IRMA Lect. Math. Theor. Phys.}, pages 167--209. Eur. Math. Soc., Z\"urich, 2016.

\end{thebibliography}
\end{document}